\documentclass[12pt]{amsart}
\usepackage[left=1in,top=1in,right=1in,bottom=1in,letterpaper]{geometry}
\usepackage[all]{xy}
\usepackage{epsfig}
\usepackage{amsmath}
\usepackage{amssymb}
\usepackage{amscd}
\usepackage{bm}
\usepackage{bbm}
\def\bbone{{\mathbbm{1}}}
\usepackage{url}
\usepackage{verbatim} 
\usepackage{color}
\usepackage{epsfig}
\usepackage{stmaryrd}
\usepackage{amsthm}
\usepackage{pifont}
\usepackage{mathrsfs}
\usepackage{wasysym}
\usepackage{hyperref}
\usepackage{graphicx}

\title[]{Renormalization of Polygon Exchange Maps arising from Corner Percolation}
\author[]{W. Patrick Hooper}
\thanks{Support was provided by N.S.F. Postdoctoral Fellowship DMS-0803013, N.S.F. Grant DMS-1101233 and a PSC-CUNY Award (funded by The Professional Staff Congress and The City University of New York).}
\address{
The City College of New York\\
New York, NY, USA 10031}
\email{whooper@ccny.cuny.edu}

\ifdefined\theorem
\else
\newtheorem{theorem}{Theorem}
\newtheorem{proposition}[theorem]{Proposition}
\newtheorem{lemma}[theorem]{Lemma}
\newtheorem{remark}[theorem]{Remark}
\newtheorem{corollary}[theorem]{Corollary}

\theoremstyle{definition}
\newtheorem{definition}[theorem]{Definition}

\fi
\ifdefined\openquestion
\else
\newtheorem{openquestion}[theorem]{Open Question}
\fi

\ifdefined\question
\else
\newtheorem{question}[theorem]{Question}
\fi

\newlength{\savearraycolsep}
\newenvironment{narrowarray}[2]%
	{\setlength{\savearraycolsep}{\arraycolsep}%
	\setlength{\arraycolsep}{#1}%
	\begin{array}{#2}}%
	{\end{array}\setlength{\arraycolsep}{\savearraycolsep}}

\newenvironment{citemize}{
\begin{itemize}
  \setlength{\itemsep}{0em}
  \setlength{\parskip}{0pt}
  \setlength{\parsep}{0pt}
  \setlength{\parindent}{1.5em}}{\end{itemize}
}

\newcommand{\boldred}[1]{{\bf \textcolor{red}{#1}}}
\newcommand{\red}[1]{{\textcolor{red}{#1}}}

\newcommand{\set}[2]{{\{#1~:~\textrm{#2}\}}}

\def\N{\mathbb{N}}%
\def\Q{\mathbb{Q}}%
\def\R{\mathbb{R}}%
\def\Z{\mathbb{Z}}%
\def\til{\widetilde}
\def\wh{\widehat}
\def\sm{\smallsetminus}

\def\0{{\mathbf{0}}}
\def\1{{\mathbf{1}}}
\def\a{{\mathbf{a}}}
\def\b{{\mathbf{b}}}
\def\ba{{\mathbf{a}}}
\def\bb{{\mathbf{b}}}
\def\bc{{\mathbf{c}}}
\def\bd{{\mathbf{d}}}
\def\be{{\mathbf{e}}}
\def\f{{\mathbf{f}}}
\def\g{{\mathbf{g}}}
\def\h{{\mathbf{h}}}

\def\m{{\mathbf{m}}}
\def\bn{{\mathbf{n}}}

\def\bp{{\mathbf{p}}}
\def\bq{{\mathbf{q}}}
\def\s{{\mathbf{s}}}

\def\v{{\mathbf{v}}} %
\def\bw{{\mathbf{w}}}

\def\bx{{\mathbf{x}}}

\def\by{{\mathbf{y}}}
\def\z{{\mathbf{z}}}
\def\bz{{\mathbf{z}}}

\def\bpi{{\mathbf{\pi}}}

\def\sB{{\mathcal{B}}}
\def\sC{{\mathcal{C}}}

\def\sG{{\mathcal{G}}}

\def\sM{{\mathcal{M}}}

\def\sO{{\mathcal{O}}}

\def\sR{{\mathcal{R}}}
\def\sS{{\mathcal{S}}}
\def\sT{{\mathcal{T}}}

\def\half{{\frac{1}{2}}}
\def\thalf{{\textstyle \frac{1}{2}}}

\newcommand{\floor}[1]{\ensuremath{{\lfloor #1 \rfloor}}}%
\makeatletter
\def\imod#1{\allowbreak\mkern10mu({\operator@font mod}\,\,#1)}
\makeatother

\def\and{{\quad \textrm{and} \quad}}

\def\omegaalt{\omega^{\textrm{alt}}}
\def\id{{\textrm{id}}}
\def\barK{{\overline{K}}}
\newcommand{\ret}[2]{{R}_{#1}(#2)}
\def\NUC{{\mathit NUC}}
\def\NS{{\mathit NS}}
\def\Step{{\mathcal S}}
\def\tPsi{{\widetilde \Psi}}
\def\meas{{m}}
\def\hsC{{\wh \sC}}
\def\aux{{a}}

\newif\ifdraft\drafttrue
\draftfalse

\newcommand{\name}[1]{\label{#1}{\ifdraft{\textcolor{blue}{\{\textrm{#1}\}}}\else\ignorespaces\fi}}

\newcommand{\tiling}[1]{{[#1]}}
\newcommand{\cyl}{{\mathit{cyl}}}

\begin{document}
\begin{abstract}
We describe a family $\{\Psi_{\alpha, \beta}\}$ of polygon exchange transformations
parameterized by points $(\alpha, \beta)$ in the square $[0, \half] \times [0, \half]$.
Whenever $\alpha$ and $\beta$ are irrational, $\Psi_{\alpha, \beta}$ has periodic orbits
of arbitrarily large period.
We show that for almost all parameters, the polygon exchange map has the property that almost every point is periodic. However, there is a dense set of irrational parameters for which this fails.
By choosing parameters carefully,
the measure of non-periodic points can be made arbitrarily close to full measure.
These results are powered by a notion of renormalization which holds in a more general setting.
Namely, we consider a renormalization of tilings arising from the Corner Percolation Model.
\end{abstract}
\maketitle

\section{Introduction}
\name{sect:introduction}
Let $X$ be a finite disjoint union of polygons in the plane. A {\em polygon exchange map} of $X$, $T:X \to X$, cuts $X$ into finitely many polygonal pieces, and applies a translation to each piece
so that the image $T(X)$ has full area in $X$. There is some ambiguity of definition on the boundaries of the pieces.

Polygon exchange maps are natural generalizations of interval exchange maps, and yet comparatively little is understood about the dynamics of a generic polygon exchange map. However, some polygon exchange maps are well understood using the idea of renormalization.
As a simple example of renormalization, a first return map of $T$ to a union of polygonal subsets might be affinely conjugate to the original map $T$. Once a renormalization procedure is found, we can hope to exploit it to deduce detailed information about the dynamical system. Papers on polygon exchange maps following this philosophy include \cite{AKT01}, \cite{LKV04}, \cite{L07} and \cite{Schwartz10}. 

In this paper, we give the first example of a two dimensional parameter space of polygon exchange maps which is 
invariant under a renormalization operation. In our case, this means that each map in the family admits a return map which is affinely conjugate
to a map in the family. (This family will be called $\{\til \Psi_{\alpha, \beta}\}$.)
We then exploit this renormalization operation to understand the dynamical behavior of these maps.

The polygon exchange maps we describe are in fact {\em rectangle exchange maps}. That is, all the polygons used define the map are rectangles with horizontal and vertical sides.
To define these maps, consider the planar lattice $\Lambda \subset \R^2$ generated by the vectors $(\half, \half)$ and $(-\frac{1}{2}, \half)$. This lattice contains $\Z^2$ as an index two subgroup. Let $Y$ be the torus $\R^2/\Lambda$. 
A fundamental domain for the action of $\Lambda$ by translation on $\R^2$ is given by the union of the two squares
$$A_1=[0,\thalf) \times [0, \thalf) \and A_{-1}=[0,\thalf) \times [\thalf,1).$$ 
Let $N$ be the finite set of four elements, 
\begin{equation}
\name{eq:Nset}
N=\{(1,0), (-1,0), (0,1), (0,-1)\} \subset \R^2.
\end{equation}
We think of $Y \times N$ as a disjoint union of four copies of the torus $Y$. 
Fix two parameters $\alpha, \beta \in [0, \half]$. We define the rectangle exchange map $\Psi_{\alpha, \beta}:Y \times N \to Y \times N$ according to the following rule. If $(x,y) \in A_s \pmod{\Lambda}$ with $s \in \{\pm 1\}$ and $\v=(a,b) \in N$, then
\begin{equation}
\name{eq:psi intro}
\Psi_{\alpha, \beta}\big((x,y),\v\big)=\big((x+b s\alpha,y+a s \beta) \pmod{\Lambda},(bs,as)\big).
\end{equation}
Note that $(bs,as) \in N$. So fixing this data, only one coordinate changes in moving from $(x,y)$ to $(x+b s\alpha,y+a s \beta)$. Figure \ref{fig:polygon exchange} illustrates a map in this family.

\begin{figure}[h]
\includegraphics[width=4in]{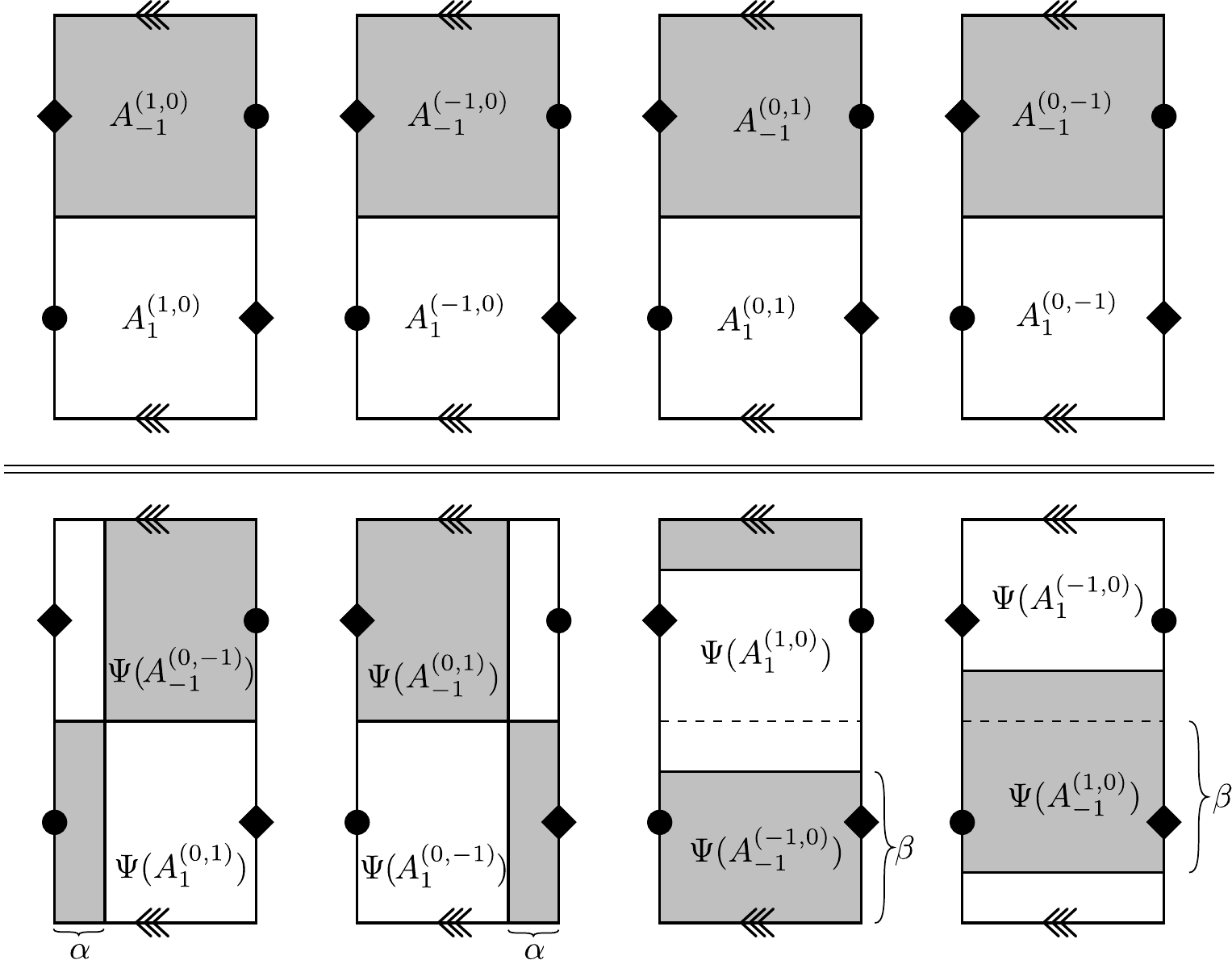}
\caption{This illustrates the map $\Psi=\Psi_{\alpha, \beta}$ defined in equation \ref{eq:psi intro}. Above the line indicates
the sets $A_s^{(a,b)}=A_s \times \{(a,b)\}$, and below illustrates their images under $\Psi$.
In both cases, the tori are drawn $Y \times \{(1,0)\}$, $Y \times \{(-1,0)\}$, $Y \times \{(0,1)\}$ and $Y \times \{(0,-1)\}$, from left to right.}
\name{fig:polygon exchange}
\end{figure}

These maps have many periodic trajectories. In fact,
\begin{theorem}
\name{thm:large periods}
Whenever $\alpha$ and $\beta$ are irrational, there are points in $Y \times N$ 
which are periodic under $\Psi_{\alpha, \beta}$ of arbitrary large period.
\end{theorem}

\begin{remark}
It follows that $\Psi_{\alpha, \beta}$ is not conjugate to a product of interval exchange maps.
\end{remark}

Every periodic point has an open neighborhood of points which are periodic and have the same period. It is natural to ask ``what is the total area of periodic points?''
Let $\lambda$ be Lebesgue measure
on $Y \times N$, rescaled so that $\lambda(Y \times N)=1$. Let $M(\alpha, \beta)$ denote the $\lambda$-measure
of the periodic points, i.e.,
$$M(\alpha, \beta)=\lambda \{p \in Y \times N~:~\Psi_{\alpha, \beta}^n(p)=p\quad \textrm{for some $n \geq 1$.}\}.$$
Our renormalization operation allows us to prove the following theorems.

\begin{theorem}[Periodicity almost everywhere]
\name{thm:periodicity}
$M(\alpha, \beta)=1$ for Lebesgue-almost every parameter $(\alpha, \beta) \in [0, \half] \times [0, \half]$.
\end{theorem}

However, this result does not hold for all irrational pairs $(\alpha, \beta)$.

\begin{theorem}
\name{thm:4}
For any $\epsilon>0$, there are irrationals $\alpha$ and $\beta$ so that
$M(\alpha, \beta)<\epsilon$. 
\end{theorem}

From this together with basic observations about the action of renormalization on the parameter space, we obtain:
\begin{corollary}
\name{cor:5}
There is a dense set of irrational parameters
$(\alpha, \beta)$ so that $M(\alpha, \beta) \neq 1$. 
\end{corollary}

Questions involving the measure-theoretic prevalence of periodic orbits for piecewise isometries are common in the literature.
Probably the first questions of this form appear in \cite{Ashwin97} and \cite[\S6]{Goetz00}. The above theorems highlight the subtlety of this question.
For the family $\{\Psi_{\alpha, \beta}\}$, we utilize a renormalization procedure to analyze $M(\alpha,\beta)$. We were able to understand the value of this function for almost every every pair $(\alpha, \beta)$ in Theorem \ref{thm:periodicity}, and for very specific pairs in
Theorem \ref{thm:4} and Corollary \ref{cor:5}. But, it is reasonable to ask if there is a nice characterization of the set 
$$\{(\alpha, \beta)~:~M(\alpha, \beta) \neq 1\}.$$
Or for instance, what is this set's Hausdorff dimension? These finer questions remain unanswered and appear difficult. 

\subsection{Renormalizing the polygon exchange maps}
\name{sect:intro renormalization}
A {\em renormalization} of a polygon exchange map $T:X \to X$, is the choice of a finite union $Y$ of polygonal subsets of $X$ with disjoint interiors such that the first return map $T_Y:Y \to Y$ is also a polygon exchange map. 

For the maps $\Psi_{\alpha, \beta}$, we actually renormalize on a double cover. Let $\til Y=\R^2/\Z^2$, and note that 
the natural projection $\pi:\til Y \to Y$ is a double cover. We define $\til A_s=\pi^{-1}(A_s)$ for $s \in \{\pm 1\}$. 
Then we define the lift of the map $\Psi_{\alpha, \beta}$ to be the map $\til \Psi_{\alpha, \beta}:\til Y \times N \to \til Y \times N$ given by 
\begin{equation}
\name{eq:psi til}
\til \Psi_{\alpha, \beta}\big((x,y),\v\big)=\big((x+b s\alpha,y+a s \beta) \pmod{\Z^2},(bs,as)\big),
\end{equation}
where $s \in \{\pm 1\}$ is chosen so that $(x,y) \in \til A_s$.

The maps $\til \Psi_{\alpha, \beta}$ are parameterized by a choice of $(\alpha, \beta)$ from the square $[0,\half] \times [0, \half]$.
We will show when $(\alpha, \beta)$ is taken from the open square $(0,\half) \times (0, \half)$,
a certain return map of $\til \Psi_{\alpha, \beta}$ is affinely conjugate to a map of the form $\til \Psi_{f(\alpha), f(\beta)}$.
Here $f$ is the map 
\begin{equation}
\name{eq:f}
f:[0, \thalf) \to [0, \thalf]\quad \textrm{is given by} \quad f(t)=\frac{t}{1-2 t} \pmod{G},
\end{equation}
where $G$ is the group of isometries of $\R$ preserving $\Z$. This group is generated by $t \mapsto -t$
and $t \mapsto 1-t$, so the interval $[0, \half]$ represents a fundamental domain for the group action.
We use $t \pmod{G}$ to denote the unique $g(t) \in [0, \half]$ with $g \in G$. To define the return map under consideration 
we define the rectangle
\begin{equation}
\name{eq:Z}
Z=[\alpha, 1-\alpha) \times [\beta,1-\beta) \subset \til Y.
\end{equation}
We define $\wh \Psi$ be the first return map of $\til \Psi_{\alpha,\beta}$ to $Z \times N$.
This map is affinely conjugate to the map $\til \Psi_{f(\alpha), f(\beta)}$ via a conjugating map of the form
\begin{equation}
\name{eq:phi}
\phi:Z \times N \to \til Y \times N; \quad \phi(x,y,\v)=\big(\psi_\alpha(x), \psi_\beta(y), \v\big).
\end{equation}
Here, we have used $\psi_t$ with $t \in \{\alpha, \beta\}$ to denote the maps
\begin{equation}
\name{eq:psi}
\psi_t:[t, 1-t) \to \R /\Z; \quad \psi_t(x)=\begin{cases}
\frac{x-\frac{1}{2}}{1-2 t}+\frac{1}{2} & \textit{if $\exists n \in \Z$ s.t. $n \leq \frac{t}{1-2 t}<n+\half$,}\\
\frac{\frac{1}{2}-x}{1-2 t} & \textit{otherwise.}
\end{cases}
\end{equation}
The two cases correspond to the possibility that we use an orientation preserving or reversing element of $G$ to move 
$\frac{t}{1-2 t}$ into $[0, \half]$. 

We now formally state our renormalization theorem. 
\begin{theorem}\name{thm:ren}
Fix parameters $\alpha, \beta \in (0, \half)$. Define $f$, $Z$, and $\phi$ as above. 
The first return map $\wh \Psi$ of $\til \Psi_{\alpha,\beta}$ to $Z \times N$ satisfies
$$\phi \circ \wh \Psi=\til \Psi_{f(\alpha),f(\beta)} \circ \phi.$$
\end{theorem}

The primary case of interest is when $\alpha$ and $\beta$ are irrational. Then, $f(\alpha)$ and $f(\beta)$ are also irrational. Therefore we can apply the above renormalization infinitely many times.

\subsection{Corner Percolation and Truchet tilings}
\name{sect:intro corner}
We will understand the family of rectangle exchange maps $\{\Psi_{\alpha, \beta}\}$ using a combinatorial tool we call the {\em arithmetic graph}, following Schwartz. (See \cite{S07}, for instance). In our case, this fundamental tool is connected to the corner percolation model introduced by B\'alint T\'oth, and studied in depth by G\'abor Pete \cite{Pete08}. (We give a different treatment of the topic in this paper.) 

The {\em corner percolation tiles} are the four $1 \times 1$ square tiles decorated by arcs as below. 
\begin{center}
\includegraphics[height=0.5in]{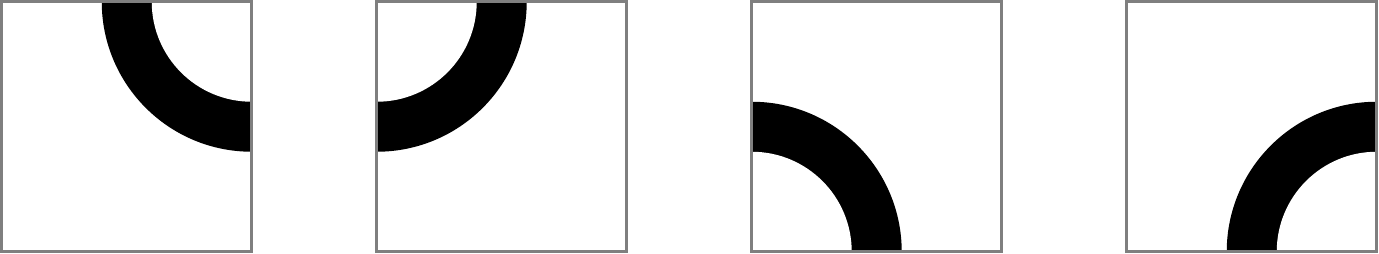}
\end{center}
Consider tilings of the plane by corner percolation tiles centered at the points in $\Z^2$. 
Any two adjacent tiles meet along a common edge.
We will say that such a tiling is a {\em corner percolation tiling} if for each pair of adjacent tiles meeting along a common edge $e$, either both the arcs of the tiles touch $e$, or
neither of the arcs touch $e$. So, in a corner percolation tiling, the arcs of the tiles join to form a family of simple curves in the plane. These are the {\em curves} of the tiling.

The {\em Truchet tiles} are the two $1 \times 1$ squares decorated by arcs as below.
\begin{center}
\includegraphics[height=0.5in]{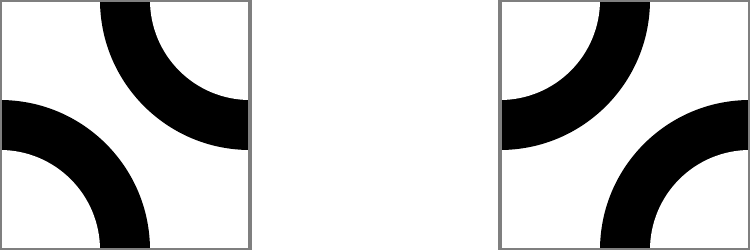}
\end{center}
We call the left tile $T_{-1}$ and the right tile $T_{1}$. The subscripts were chosen to indicate the slope of segments formed
by straightening the arcs to segments.

Given a function $\tau:\Z^2 \to \{\pm 1\}$, the {\em Truchet tiling determined by $\tau$} is the tiling of the plane formed by placing
a copy of the tile $T_{\tau(m,n)}$ centered at the point $(m,n)$ for each $(m,n) \in \Z^2$. 
We denote this tiling by $\tiling{\tau}$. 
Variations of these tilings were first studied for aesthetic reasons by S\'ebastien Truchet in the early 1700s \cite{Truchet}, and this version of tiles were first described by Smith and Boucher
 \cite{Smith87}. An example of a Truchet tiling relevant to this paper is given in figure
\ref{fig:arithmetic graph}.

\begin{figure}[b]
\includegraphics[width=6.5in]{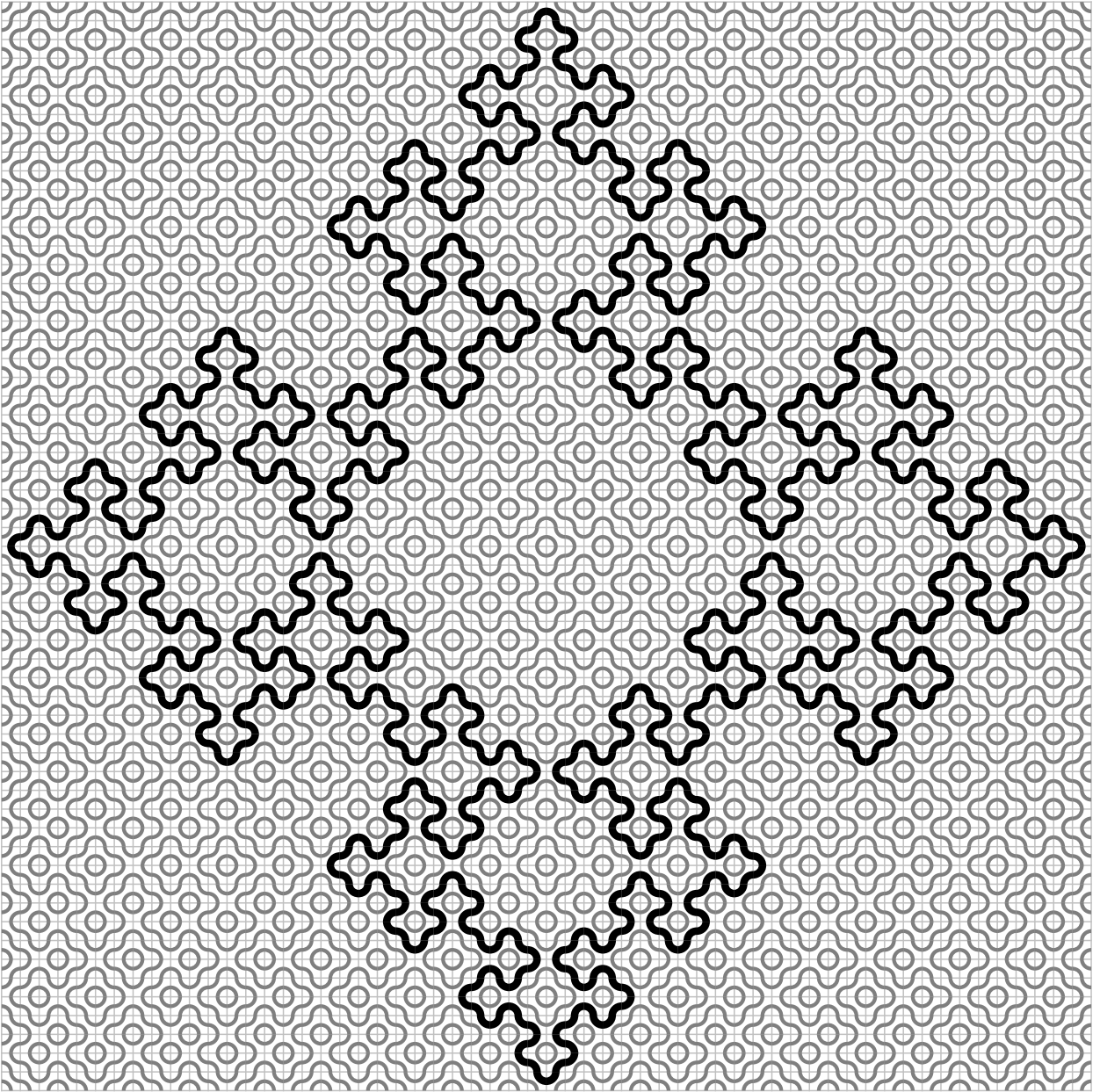}
\caption{This is the tiling $[\tau_{\alpha, \beta, x,y}]$ in the notation of \S\ref{sect:arithmetic graph} with $\alpha=\beta=\frac{2-\sqrt{2}}{2}$, $x=\frac{\sqrt{2}}{4}$ and
$y=\frac{2+\sqrt{2}}{4}$. This tiling is renormalization invariant in the sense of \S\ref{sect:renormalizing tilings}, explaining the apparent self-similarities.}
\name{fig:arithmetic graph}
\end{figure}

There is a two-to-one map from the corner percolation tiles to the Truchet tiles given by taking the union of the decorations of a corner percolation tile and its rotation by 180 degrees.
By applying this map to each tile in a corner percolation tiling, we obtain a {\em corner percolation induced Truchet tiling}. 
\begin{proposition}
\name{prop:corner perc induced}
Let $\tau:\Z^2 \to \{\pm 1\}$. The following statements are equivalent.
\begin{enumerate}
\item The Truchet tiling $[\tau]$ is induced by a corner percolation tiling.
\item There are maps $\Z \to \{\pm 1\}$ given by $m \mapsto \omega_m$ and $n \mapsto \eta_n$ so that $\tau(m,n)=\omega_m \eta_n$. 
\item For each $m,n \in \Z$, we have the following identity involving a product of values of $\tau$: 
$$\tau(m,n)  \tau(m+1,n)  \tau(m,n+1)  \tau(m+1,n+1)=1.$$
\end{enumerate}
\end{proposition}
The easiest way to prove this statement is to prove that the first and second statements are equivalent to the third. We leave the proof to the reader.

\subsection{Dynamics on Truchet tilings}
\name{sect:intro dynamics on truchet}
We will explain how to think of the space of Truchet tilings as a dynamical system.

Consider the unit square with horizontal and vertical sides centered at the origin. We let $N$ be the collection of four inward pointed unit normal vectors based at the midpoints of the edges of this square, as in equation \ref{eq:Nset}.
\begin{center}
\includegraphics[height=1in]{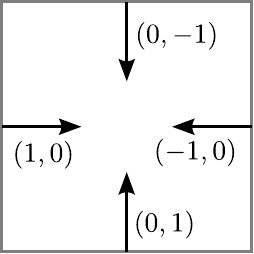}
\end{center}

Let $\sT$ denote the collection of all maps $\Z^2 \to \{\pm 1\}$. 
The collection of maps $\sT$ should be given the product topology (or equivalently, the topology of pointwise convergence on compact sets). 

We will define a dynamical system on $\sT \times N$. First we give an informal definition. Choose $(\tau, \v) \in \sT \times N$. The inward normal $\v \in N$ is a vector based at a midpoint of an edge of the square centered at the origin. The Truchet tiling determined by $\tau$ places the tile $T_{\tau(0,0)}$ at the origin.
We drag the vector inward along an arc of this tile keeping the vector tangent to the arc. After a quarter turn, we end up as a vector pointed out of the square centered at the origin.
So, the vector points into a square adjacent to the square at the origin. We translate the tiling and this vector
so that this adjacent square becomes centered at the origin.

Formally, this is the dynamical system $\Phi_0:\sT \times N \to \sT \times N$ given by 
\begin{equation}
\name{eq:Phi}
\Phi_0 \big(\tau, (a,b)\big)=\big(\tau \circ S_{s(b,a)}, s(b,a)\big),
\end{equation}
where $s=\tau(0,0) \in \{\pm 1\}$ and $S_{s(b,a)}$ is the translation of $\Z^2$ by the vector
$s(b,a)$.

It is important to note that because the corner percolation induced Truchet tilings are translation invariant, 
they are also $\Phi_0$ invariant. So, $\Phi_0$ restricts to an action on corner percolation induced Truchet tilings.
 
Let $\Omega_{\pm}$ to denote the collection of all maps $\Z \to \{\pm 1\}$. The set $\Omega_\pm$ is a shift space. We define the {\em shift map} $\sigma: \Omega_\pm \to \Omega_\pm$ by 
\begin{equation}
\name{eq:shift}
[\sigma(\omega)]_n=\omega_{n+1}.
\end{equation}
When $\Omega_\pm$ is equipped with its natural topology, $\sigma$ is a homeomorphism of $\Omega_\pm$. 

Consider the map 
$\Omega_{\pm} \times \Omega_{\pm} \to \sT$ given by $(\omega, \eta) \mapsto \tau_{\omega, \eta}$, where
$\tau_{\omega,\eta}$ denotes the map 
\begin{equation}
\name{eq:tau omega}
\tau_{\omega, \eta}: \Z^2 \to \{\pm 1\}; \quad \tau_{\omega, \eta}(m,n)=\omega_m \eta_n
\end{equation}
as in statement 2 of Proposition \ref{prop:corner perc induced}. This map is two-to-one, and the image is the collection of corner percolation induced Truchet tilings. 
There is a natural lift of the action of $\Phi_0$ on the image to the space $X=\Omega_\pm  \times \Omega_\pm \times N$. This lift is the map
 $\Phi:X \to X$ given by
\begin{equation}
\name{eq:Phi2}
\Phi\big(\omega, \eta, (a,b)\big)=\big(\sigma^{sb} (\omega), \sigma^{sa} (\eta), s(b,a)\big), \quad
\textrm{with $s=\omega_0 \eta_0 \in \{\pm 1\}$.}
\end{equation}

\subsection{Overview} \name{sect:overview} We have now introduced enough of the mathematical objects appearing in the paper, so we can give an overview of the ideas of this paper. 

The rectangle exchange maps $\til \Psi_{\alpha, \beta}$ are factors of the map $\Phi$ in the sense
that for all irrational $\alpha$ and $\beta$, there is an embedding 
\begin{equation}
\name{eq:factor}
\pi:\til Y \times N \to X=\Omega_\pm \times \Omega_\pm \times N \quad \text{so that} \quad \pi \circ \til \Psi_{\alpha, \beta}=\Phi \circ \pi.
\end{equation}
The map $\pi$ can be extended to a continuous embedding from a coding space as is often done
for interval exchange maps. See the coding construction for interval exchange maps in \cite{KH95}, for instance.

We will describe a renormalization operation for the map $\Phi$. Using the map $\pi$,
we are able to restrict this renormalization operation to a renormalization of the rectangle exchange
maps $\til \Psi_{\alpha, \beta}$. This enables us to prove Theorem \ref{thm:ren}. 

We are able to prove our measure theoretic results using a detailed analysis of the renormalization of these rectangle exchange maps. Of particular importance is a finite dimensional cocycle defined over
the renormalization dynamics of the parameter space. We call this the {\em return time cocycle}, because the cocycle conveys information about return times of the rectangle exchange maps to the subsets we use to define the return maps for our renormalization. 

We will now outline the proof of Periodicity Almost Everywhere (Theorem \ref{thm:periodicity}):
\begin{enumerate}
\item For each $\alpha, \beta \in [0,\half]$, we define the measure $\nu_{\alpha, \beta}$ on $X$ to be $\lambda \circ \pi^{-1}$,
where $\lambda$ is the Lebesgue probability measure on $\til Y \times N$ and $\pi$ is the embedding 
which was mentioned
in equation \ref{eq:factor} and
depends on $\alpha$ and $\beta$ . These measures are $\Phi$-invariant.
\name{item:factor}
\item We define the notion of a stable periodic orbit of $\Phi$ and let $\NS \subset X$ be the collection of points without a stable periodic orbit. If $\alpha$ and $\beta$ are irrational, then any point $z$ which is periodic
under $\tPsi_{\alpha, \beta}$ satisfies $\pi(z) \not \in \NS$. \name{item:stable}
\item
We define a nested sequence of Borel subsets $\sO_k \subset X$
so that $\NS = \bigcap_{k=0}^\infty \sO_k$ up to sets of $\nu_{\alpha, \beta}$-measure zero.
Then, $\nu_{\alpha, \beta}(\NS)$ is the limit of a decreasing sequence, $\lim_{k \to \infty} \nu_{\alpha, \beta}(\sO_k)$. \name{item:limit}
\item \name{item:inequality} Using the return time cocycle, we are able to find an expression for 
$\nu_{\alpha, \beta}(\sO_k)$. We then show there is a continuous function $g:(0,\half] \times (0, \half] \to \R$ which is strictly less than one on its domain so that for all $k \geq 1$,
$$\nu_{\alpha, \beta}(\sO_{k}) \leq g\big(f^{k-1}(\alpha),f^{k-1}(\beta)\big) \nu_{\alpha, \beta}(\sO_{k-1}),$$
where $f$ is defined as in equation \ref{eq:f}. 
It follows that if the orbit $\{(f \times f)^k(\alpha, \beta)~:~k \geq 0\}$ has an accumulation 
point in $(0,\half] \times (0, \half]$, then $\nu_{\alpha, \beta}(\NS)=0$ as desired. 
(We remark that $g(\alpha, \beta)$ tends to one if either $\alpha$ or $\beta$ tends to zero.)
\item 
We show that Lebesgue-a.e. pair $(\alpha, \beta)$ recurs under $f \times f$.
\name{item:zero}
\end{enumerate}
 
We will now say a few words about the proof of Theorem \ref{thm:4}, which says that irrational parameters $(\alpha, \beta)$ exist so that the total measure of periodic points of $\Psi_{\alpha, \beta}$ is as close to zero as we like.
By the above argument, if the measure of the non-periodic points of $\Psi_{\alpha, \beta}$ is to be positive, then 
the orbit of $(\alpha, \beta)$ under $f \times f$ must diverge in the sense that
$$\limsup_{k \to \infty} \min \big(f^k(\alpha), f^k(\beta_k)\big)=0.$$
To find such an $(\alpha, \beta)$, we observe that $f$ is semi-conjugate to the shift map on the full one-sided shift space defined over a countable alphabet. 
So, we can describe a pair $(\alpha, \beta)$ in terms of a symbolic coding of its $f \times f$-orbit. We understand the cocycle mentioned above in terms of this symbolic coding, and show that for appropriate choices $\nu_{\alpha,\beta}(\NS)$ can be made as close to one as we like. 

\subsection{Background on polygon exchange transformations}
\name{sect:background}
Few general results about rectangle and polygon exchange transformations are known. Our lack of understanding is highlighted by a question of Gowers \cite{G00}: are all rectangle exchanges recurrent? It is known that (vast generalizations of) polygon exchange transformations have zero entropy \cite{GH97}. And in \cite{H81}, a criterion is provided for a rectangle exchange to be minimal. 

A {\em piecewise rotation} is a collection of polygons $X$ in $\R^2$ together with a map $T:X \to X$. The map $T$ cuts $X$ into finitely many polygonal pieces and applies an orientation preserving Euclidean isometry to each piece. 
The image $T(X)$ must have full area in $X$. 

If on each polygonal piece, $T$ performs either a translation or a rotation by a rational multiple of $\pi$, then there is a natural construction of a polygon exchange map $S:Y \to Y$ together with a covering map $c:Y \to X$ so that $c \circ S=T \circ c$. Thus, studying such a {\em rational piecewise rotation} is closely related to studying a polygon exchange map. 

There are several examples of renormalizable piecewise rotations.
In \cite{AKT01}, a family of piecewise rotations is studied. Renormalization is used to understand a few of the maps in this family whose pieces are rotated by rational multiplies of $\pi$. Another example of a renormalizable piecewise rotation is provided in \cite{LKV04}. And in \cite{L07}, a general theory of renormalization of piecewise rotations is developed.  In all these cases, periodic points are shown to be of full measure in the dynamical system. 

Another topic of the papers \cite{AKT01}, \cite{LKV04} and \cite{L07} is to understand the dynamics on the set of points whose orbits are not periodic. (E.g., we would like to know if the dynamics are minimal or uniquely ergodic on this set.) These questions could be asked about maps in the family $\{\Psi_{\alpha, \beta}\}$, but we postpone investigating these questions until a subsequent paper.

In \cite{GP04}, renormalization arguments are used to explain that natural return maps of piecewise rotations may be piecewise rotations with a countable collection of polygons of continuity. By an observation of Hubert, this holds even for rectangle exchange maps and products of interval exchange maps \cite[\S 6.1]{GP04}.

Outer billiards also gives rise to polygon exchange maps. Fix a convex polygon $P$ in $\R^2$. There are two continuous choices of maps $\phi$ from $\R^2 \sm P$ to the space of tangent lines of $P$ so that each point $Q \in \R^2 \sm P$ is sent to a tangent line containing $Q$. Choose such a $\phi$. For a typical point $Q$, $\phi(Q)$ intersects $P$ in exactly one point $Z$, which is a vertex of $P$. We define $T(Q)$ to be the point obtained applying the central reflection through $Z$ to the point $Q$. This defines the {\em outer (or dual) billiards map}, a map $T:\R^2 \sm P \to \R^2 \sm P$. ($T$ is well defined and invertible off a finite number of rays.) We refer the reader to \cite{Tab} for an introduction to the subject. 

The square of the outer billiards map is a piecewise translation of $\R^2 \sm P$. Return maps of $T^2$ to polygonal sets give possible sources of polygon exchange maps. Maps of these forms are studied \cite{Tab95}, \cite{BC09}, and \cite{Schwartz10}.

A {\em polytope exchange transformation} is the 3-dimensional analog of a polygon exchange. 
Recently, Schwartz described a renormalization scheme for a polytope exchange map arising from a compactification of a first return map of outer billiards map in the Penrose kite \cite{Schwartz11}. Due to the complexity of this polytope exchange, Schwartz's renormalization result is proved with the aid of a computer. It is believed that other outer billiards systems should exhibit similar phenomena. 

The polygon exchange transformations that arise in this paper were concocted to share properties with the polytope and polygon exchange maps studied in \cite{Schwartz11}. In particular, the Truchet tilings we study share many properties with the arithmetic graph studied in \cite{S07}, \cite{S09} and \cite{Schwartz11}. Namely, both give decorations of the plane by simple curves which may be closed or bi-infinite. 

\subsection{Outline}
In section \ref{sect:curve following}, we formally define a curve following map for a Truchet tiling. Roughly, this map is the same as the definition of $\Phi_0$ only we do not translate the tiling. 

In section \ref{sect:arithmetic graph}, we describe the construction of the arithmetic graph, which connects the curve following map to the dynamics of our polygon exchange maps. We use this construction to define the map $\pi$ as in equation \ref{eq:factor}.

In section \ref{sect:renormalizing tilings}, we explain how to renormalize the curve following map for Truchet tilings arising from corner percolation. For most such tilings, we find a subset of tiles such that the return map of the curve following map to this subset is conjugate to the curve following map of a different tiling. This is the most important observation of the paper.

Subsection \ref{sect:renormalization2} takes the renormalization of the curve following map and promotes it to a renormalization of the map $\Phi$ defined in equation \ref{eq:Phi2}. Earlier subsections 
of section \ref{sect:corner} explain necessary background and definitions necessary to describe this version of renormalization. We define the set $\NS \subset X$ of non-stable periodic orbits in section \ref{sect:periodic orbits}. Proposition \ref{prop:stable} implies that periodic points $z$ of $\til \Psi_{\alpha, \beta}$ satisfy $\pi(z) \not \in \NS$ when $\alpha$ and $\beta$ are irrational. These observations were part of statement (\ref{item:stable}) of the outline of the proof of Theorem \ref{thm:periodicity} (Periodicity almost everywhere). 

In section \ref{sect:rectangle exchanges}, we explain how the renormalization of $\Phi$ induces the renormalization of the polygon exchange maps $\til \Psi_{\alpha, \beta}$ described by Theorem \ref{thm:ren}.

The next two sections of the paper deal with our measure theoretic results. We define the return time cocycle and state a theorem which describes the cocycle's relevance in subsection \ref{sect:limit formula}. This relevance includes a connection to the decreasing sequence of sets
$\sO_k$ mentioned in statement (\ref{item:limit}) of the outline of the proof of Theorem \ref{thm:periodicity}. The later subsections are concerned with 
explaining the construction of the cocycle and proving the main theorem of the section.

In section \ref{sect:calc}, we utilize the cocycle to prove our measure theoretic results. Subsection \ref{sect:recurrent case}, proves statement (\ref{item:inequality}) of the outline of the Periodicity Almost Everywhere theorem. Subsection \ref{sect:non-recurrent case} proves Theorem \ref{thm:4} and Corollary \ref{cor:5} of the introduction, which guarantee the existence of parameters for which the map $\Psi_{\alpha,\beta}$ is not periodic almost everywhere.

Finally, section \ref{sect:param space} concerns the dynamical behavior of the maps $f$ and $f \times f$. The map $f$ was defined in equation \ref{eq:f} and $f \times f$ is the action of renormalization on the parameter space. Many of our results are predicated on the understanding of these maps, e.g. statement (\ref{item:zero}) of the outline of the proof of Theorem \ref{thm:periodicity}. Our analysis of these maps is fairly standard, so we have postponed this discussion to the end of the paper. 

\subsection{Acknowledgements}
The author would like to thank the referee, whose helpful comments vastly improved the exposition of this paper.

\section{Following curves in Truchet tilings}
\name{sect:curve following}

We now present a useful concept for understanding Truchet tilings and the dynamics of the map $\Phi_0:\sT \times N \to \sT \times N$ defined in equation \ref{eq:Phi}. Recall that $\Phi_0(\tau, \v)$ moved the vector $\v \in N$ along the curve of the tile of the tiling $[\tau]$ centered at the origin, and then translated to keep the vector pointed into the square at the origin. 

We will now consider what happens if we forget the translation. In this case, the tiling remains fixed while the vector has moved away from the origin. Formally, we fix a Truchet tiling $[\tau]$ determined by a map $\tau:\Z^2 \to \{\pm 1\}$ and define the {\em curve following map} to be
\begin{equation}
\name{eq:C}
\sC:\Z^2 \times N \to \Z^2 \times N; \quad 
\big((m,n),(a,b)\big) \mapsto \big((m+sb,n+sa),s(b,a)\big),
\end{equation}
where $s=\tau(m,n)$. This map considers the inward pointing unit normal in direction $(a,b)$ based at a midpoint of an edge of the unit square centered at $(m,n)$. It moves the vector forward along the curve of the tile centered at $(m,n)$, keeping the vector tangent to the curve, and stops as soon as the vector leaves the tile. The new resulting vector points into the square centered
at $(m+sb,n+sa)$ and points in direction $s(b,a)$. 

We can recover the behavior of powers of the the map $\Phi_0$ applied to pairs of the form $(\tau, \v)$ from the curve following map for $\tau$. To do this, define the map
\begin{equation}
\name{eq:S0}
\sS_0:\Z^2 \times N \to \sT \times N; \quad 
(m,n, \v) \mapsto (\tau \circ S_{m,n}, \v),
\end{equation}
where $S_{m,n}:\Z^2 \to \Z^2$ is the translation $(x,y) \mapsto (x,y)+(m,n)$. 
Either by inspection or induction, it can be shown that for all $\tau \in \sT$ and all $k \in \Z$,
\begin{equation}
\name{eq:curve following relation1}
\Phi_0^k(\tau,\v)=\sS_0 \circ \sC^k\big(0,0,\v\big).
\end{equation}
Informally, the right hand side just waits to translate until we have moved $k$ steps forward, but we translate by the composition of the translations used when evaluating $\Phi_0^k(\tau,\v)$. 

Similarly, we can recover the behavior of the map $\Phi$. Fix $\omega$ and $\eta$ and define $\tau$ by $\tau(m,n)=\omega_m \eta_n$ as in \ref{eq:tau omega}.
Then we define an analog of $\sS_0$ and see that it satisfies a similar identity involving the curve following map of $\tau$. 
\begin{equation}
\name{eq:S}
\sS:\Z^2 \times N \to X; \quad (m,n,\v) \mapsto \big(\sigma^m(\omega),\sigma^n(\eta),\v).
\end{equation}
\begin{equation}
\name{eq:curve following relation}
\Phi^k(\omega, \eta,\v)=\sS \circ \sC^k\big(0,0,\v\big).
\end{equation}
Here $X=\Omega_\pm \times \Omega_\pm \times N$ is the domain of $\Phi$ as in the introduction.

\section{Construction of the arithmetic graph }
\name{sect:arithmetic graph}
In this section, we fully explain the connection between the family of polygon exchange maps $\{\Psi_{\alpha, \beta}\}$ and Truchet tilings which arise from corner percolation.

Consider the polygon exchange map $\til \Psi_{\alpha, \beta}:\til Y \times N \to \til Y \times N$ defined in Equation \ref{eq:psi intro}. Let
$(x_0, y_0) \in \til Y=\R^2 / \Z^2$ and choose a $\v \in N$. Then let $\big((x_1, y_1), \bw\big)=\til \Psi_{\alpha, \beta} \big((x_0,y_0),\v\big)$.
Observe that (modulo $\Z^2$) we have
$$(x_1,y_1)-(x_0,y_0) \in \{(\pm \alpha, 0), (0, \pm \beta)\}.$$
Fixing $(x,y) \in \til Y$, we define the map 
\begin{equation}
\name{eq:M}
M:\Z^2 \times N \to \til Y \times N; \quad M(m,n, \v)=(x+m \alpha, y+n \beta, \v).
\end{equation}
The argument above shows that $M(\Z^2 \times N)$ is $\til \Psi_{\alpha, \beta}$-invariant. Note also
that so long as $\alpha$ and $\beta$ are irrational, the map $M$ is injective. 

\begin{definition}[Arithmetic Graph]
The {\em arithmetic graph} associated to the irrational parameters $(\alpha, \beta) \in (0,\half) \times (0, \half)$
and the point $(x,y) \in \til Y$ is the directed graph whose vertices are points in $\Z^2 \times N$ with an edge
running from $(m_0, n_0, \v)$ to $(m_1, n_1, \bw)$ if and only if 
$$\til \Psi_{\alpha, \beta} \circ M(m_0, n_0, \v)=M(m_1, n_1, \bw).$$
\end{definition}

We will show that the arithmetic graph associated to $(\alpha, \beta)$ and $(x,y) \in \til Y$ 
is closely related to a Truchet tiling. Define
\begin{equation}
\name{eq:omega from x}
\omega_m=\begin{cases}
1 & \textrm{if $x+m\alpha \in [0, \half)$}\\
-1 & \textrm{if $x+m\alpha \in [\half,1)$}
\end{cases}
\and 
\eta_n=\begin{cases}
1 & \textrm{if $y+n \beta \in [0, \half)$}\\
-1 & \textrm{if $y+n \beta \in [\half,1)$.}
\end{cases}
\end{equation}
In these definitions, $x+m \alpha$ and $y+n \beta$ are taken to lie in $\R/\Z$.
We then define $\tau$ according to the rule $\tau(m,n)=\omega_m \eta_n$.

\begin{proposition}
\name{prop:graph iso}
Fix irrationals $\alpha, \beta \in (0, \half)$ and fix any $(x,y) \in \til Y$. Let $\omega$, $\eta$ and $\tau$ be as above. 
Then there is an edge joining $(m_0, n_0, \v)$ to $(m_1, n_1, \bw)$ in the arithmetic graph if and only if the curve following map of $\tau$ satisfies
$$\sC \big((m_0, n_0), \v\big)=\big((m_1, n_1), \bw\big).$$
\end{proposition}
\begin{proof}
We must show that for each $(m,n,\v) \in \Z^2 \times N$, we have 
$$\til \Psi_{\alpha, \beta} \circ M(m,n,\v)=M \circ \sC(m,n,\v).$$
Let $\v=(a,b)$ and $s=\tau(m, n)$. Then by the definitions of $\sC$ and $M$, we have 
$$\begin{narrowarray}{0pt}{rcl}
M \circ \sC\big(m,n, \v) & = & M\big((m+sb,n+sa),s(b,a)\big) \\
& = & 
\Big(\big(x+(m+sb)\alpha,y+(n+sa)\beta\big),s(b,a)\Big).
\end{narrowarray}$$
Observe that by definition of $\tau$ and $s$,
we have $(x+m\alpha,y+n\beta)\in \til A_s$. It follows that 
$$\begin{narrowarray}{0pt}{rcl}
\til \Psi_{\alpha, \beta} \circ M(m,n,\v) & = &
\til \Psi_{\alpha, \beta}(x+m\alpha, y+n\beta, \v) \\
& = & 
\big(x+m\alpha+bs\alpha, y+n\beta+as\beta, (bs,as)\big).
\end{narrowarray}$$
\end{proof}

We define the embedding map $\pi$ which appeared in section \ref{sect:overview} of the introduction by
\begin{equation}
\name{eq:pi}
\pi:\til Y \times N \to X; \quad (x,y,\v) \mapsto (\omega, \eta, \v)
\end{equation}
with $\omega$ and $\eta$ defined in terms of $\alpha$, $\beta$, $x$ and $y$ as in equation \ref{eq:omega from x}. We show this map satisfies equation \ref{eq:factor}:

\begin{proposition}
\name{prop:ag conj}
If $\alpha$ and $\beta$ are irrational, then $\pi \circ \tPsi_{\alpha, \beta}=\Phi \circ \pi$.
\end{proposition}
\begin{proof}
Fix $x$, $y$ and $\v$. Define $\omega$ and $\eta$ so that $\pi(x,y,\v)=(\omega, \eta, \v)$. 
Then, we have
$$\Phi \circ \pi(x,y,\v)=\Phi(\omega, \eta,\v)=\sS \circ \sC(0,0,\v),$$
by equation \ref{eq:curve following relation}. By Proposition \ref{prop:graph iso}, we continue:
$$\Phi \circ \pi(x,y,\v)=\sS \circ M^{-1} \circ \til \Psi_{\alpha, \beta} \circ M(0,0,\v).$$
Here we can invert $M$ because irrationality of $\alpha$ and $\beta$ implies the map $M$ is injective. We claim that the map
$$\sS \circ M^{-1}:M(\Z^2 \times N) \to X
\quad \text{is given by} \quad \sS \circ M^{-1}=\pi.$$
This will conclude the proof since $M(0,0,\v)=(x,y,\v)$. 
To prove this claim, we demonstrate that $\sS=\pi \circ M$. Fix any $(i,j,\bw) \in \Z^2 \times N$. Then 
$$\pi \circ M(i,j,\bw)=\pi(x+i \alpha, y+j\beta, \bw)=(\omega', \eta', \bw).$$
Here, $\omega'$ and $\eta'$ are defined as in equation \ref{eq:omega from x}, but with $x$ replaced by $x+i\alpha$
and $y$ by $y+j\beta$. By definition of $\omega$, $\eta$, $\omega'$ and $\eta'$, we have the desired identity
$$(\omega', \eta' ,\bw)=\big( \sigma^{i}(\omega), \sigma^{j}(\eta), \bw\big)=\sS(i,j,\bw).$$
\end{proof}

\section{Renormalization of the Truchet Tilings}
\name{sect:renormalizing tilings}
In this section, we will explain how the Truchet tilings induced by corner percolation tilings exhibit a ``renormalization operation.'' We call this operation a renormalization, because when interpreted dynamically the operation corresponds to a renormalization of the map 
$\Phi:X \to X$ defined in equation \ref{eq:Phi2} of the introduction.


\subsection{Renormalization}
For any $\omega \in \Omega_{\pm}$, we define the subset $K(\omega) \subset \Z$ to be
\begin{equation}
\name{eq:K1}
K(\omega)=\set{n \in \Z}{$\omega_n \neq -1$ or $\omega_{n+1}\neq 1$} \cap
\set{n \in \Z}{$\omega_{n-1}\neq-1$ or $\omega_{n} \neq 1$}.
%
\end{equation}
That is, $K(\omega)$ is the collection $n$ so that $\omega_n$ is not part of a subword of the form $-+$. 

Throughout this section, we will fix $\omega, \eta \in \Omega_\pm$, and define $\tau=\tau_{\omega, \eta}$
as in equation \ref{eq:tau omega} (i.e., $\tau(m,n)=\omega_m \eta_n$).
By Proposition \ref{prop:corner perc induced}, all Truchet tilings induced by corner percolation
are of this form. To describe the renormalization of $[\tau]$, 
we construct the two sets $K(\omega)$ and $K(\eta)$. We make the following assumption about these
sets:
\begin{equation}
\name{eq:assumption}
\text{The sets $K(\omega)$ and $K(\eta)$ have neither upper nor lower bounds.}
\end{equation}
We will assume that $\omega$ and $\eta$ satisfy this assumption throughout this section.
This condition guarantees that there exist increasing bijections
$$\kappa_1:\Z \to K(\omega) \and \kappa_2:\Z \to K(\eta).$$
Each of these bijections is unique up to precomposition with a translation of $\Z$. 

We are now ready to define the renormalization of the tiling $[\tau]$. The renormalized tiling is defined by removing rows and columns of tiles
from $[\tau]$ and then sliding the remaining tiles together. We define the {\em set of centers of the kept squares} to be
$$\barK=K(\omega) \times K(\eta).$$
We also define the bijection
$$\kappa=\kappa_1 \times \kappa_2:\Z^2 \to \barK.$$
This enables us to define the {\em renormalization of $\tau$} to be the map 
$\tau'=\tau \circ \kappa.$
We call $[\tau']$ the renormalization of $[\tau]$. It is uniquely defined up to translation.

\begin{figure}[htb]
\begin{center}
\includegraphics[scale=0.5]{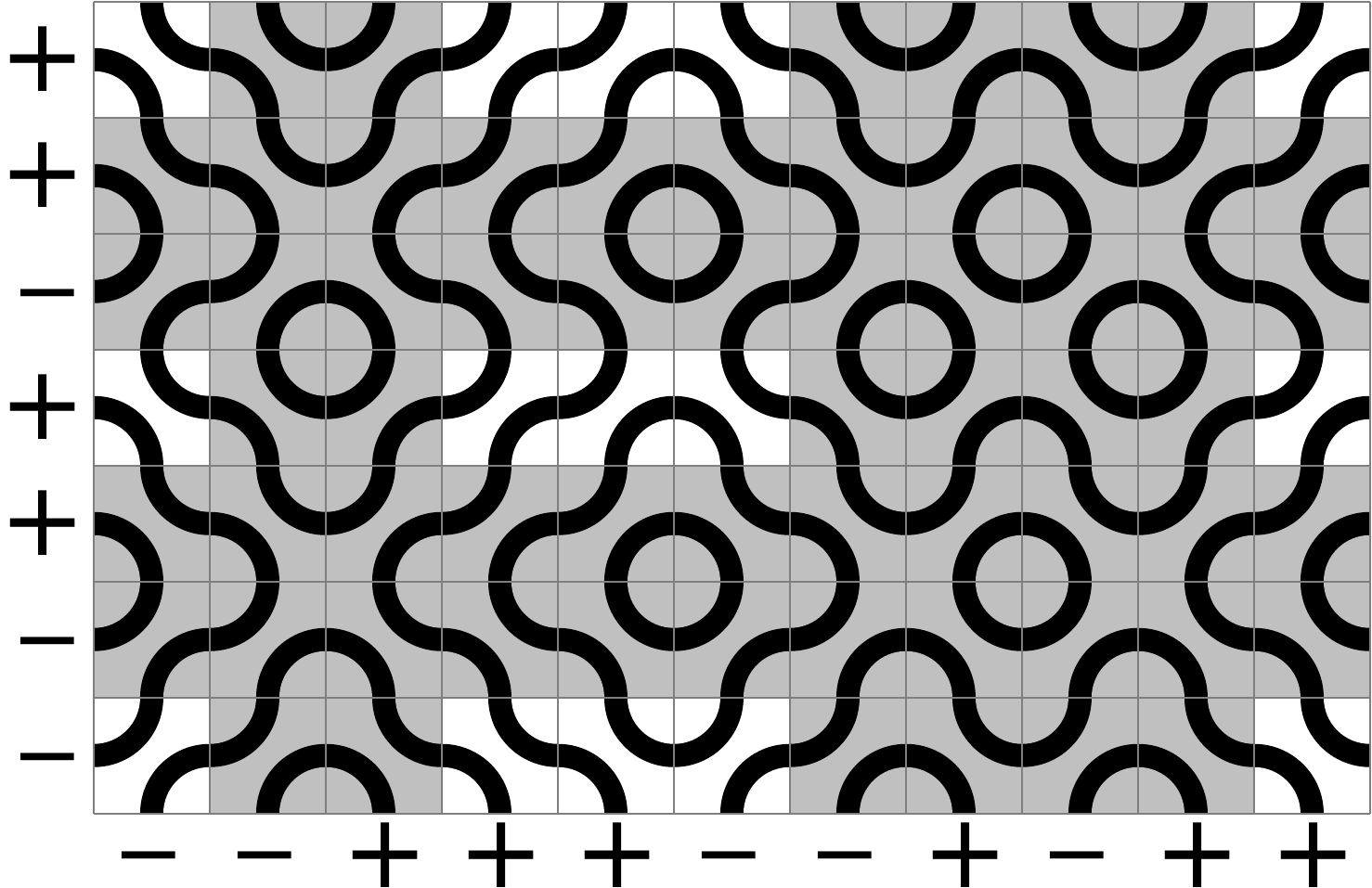}
\quad
\includegraphics[scale=0.5]{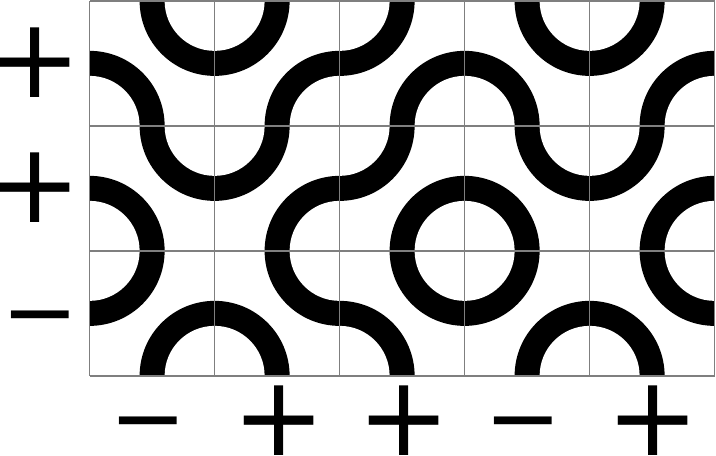}
\caption{The tiling $[\tau]=[\tau_{\omega, \eta}]$ is shown on the left. The sequence $\omega$ is shown below this tiling,
and $\eta$ is shown on the left. The set $\barK$ consists of the centers of white squares. The renormalized tiling $[\tau']$ is shown on the right.
}
\name{fig:even rectangles}
\end{center}
\end{figure}

The tiles whose centers lie in the set $\Z^2 \sm \barK$ are a union of rows and columns.
The tiling $[\tau']$ can be obtained from the tiling $[\tau]$ by collapsing all columns of tiles with centers in $\Z^2 \sm \barK$
to vertical lines, and collapsing all rows of tiles with centers in $\Z^2 \sm \barK$ to horizontal lines.
See figure \ref{fig:even rectangles}.

This paper exploits the relationship between the tiling $[\tau]$ and the renormalized tiling $[\tau']$, which we will informally state now and state formally in the theorem below.
First, whenever four tiles form a loop these four tiles are removed by the renormalization operation. Second, 
the renormalization operation preserves the identities of any curve in the tiling which is not a loop of length four. That is, some tiles making the curve
may be removed, but once the remaining tiles are slid together again, there is a new curve which visits the remaining tiles of the curve in the same order. 
Third, this process shortens all closed loops visiting more than four tiles. We then hope to apply this process repeatedly, shrinking long loops
until they eventually become loops of length four and disappear. This gives a mechanism to detect closed loops in the tiling.

In stating a theorem which makes this relationship between $[\tau]$ and $[\tau']$ rigorous, we will utilize the curve following map defined in section \ref{sect:curve following}.
We define 
$\sC$ and $\sC'$ to be the curve following maps defined in equation \ref{eq:C} with respect to the tilings $[\tau]$ and $[\tau']$, respectively. We also define $\hsC:\barK \times N \to \barK \times N$ to be the first return map of $\sC$ to $\barK \times N$. 
That is, when $(m,n,\v) \in \barK \times N$, we define
\begin{equation}
\name{eq:sR}
\hsC(m,n,\v)=\sC^k(m,n,\v) \quad \text{where} \quad k=\min~\{j>0~:~\sC^j(m,n,\v)\in \barK \times N\}.
\end{equation}
Informally, the map $\hsC$ takes a inward unit normal to a square whose center lies in $\barK$, then moves the vector along the curve of the tiling $[\tau]$ until the vector returns to a square whose center lies in $\barK$. 

\begin{theorem}[Tiling Renormalization]
\name{thm:tiling renormalization}
Assume $\omega, \eta \in \Omega_\pm$ satisfy the assumption given in equation \ref{eq:assumption}. In this case:
\begin{enumerate}
\item The first return map $\hsC$ of $\sC$ to $\barK \times N$ is well defined on all of $\barK \times N$.
\item Define $\til \kappa:\Z^2 \times N \to \barK \times N$ by $\til \kappa\big((m,n),\v\big)=\big(\kappa(m,n),\v)$. Then,
$$\hsC \circ \til \kappa=\til \kappa \circ \sC'.$$
\item The following statements are equivalent for any $(m,n, \v) \in \Z^2 \times N$. 
\begin{enumerate}
\item There is no $k>0$ so that $\sC^k(m,n,\v) \in \barK \times N$. 
\item There is no $k<0$ so that $\sC^k(m,n,\v) \in \barK \times N$. 
\item $\sC^4(m,n,\v)=(m,n,\v)$. 
\end{enumerate}
\item If there is a $p$ so that $\sC^p(m,n,\v)=(m,n,\v)$, then there is a $k>0$ so that $\sC^k(m,n,\v) \not \in \barK$. 
\end{enumerate}
\end{theorem}

We will also need to understand the return times of $\sC$ to $\barK \times N$ in terms of the tiling $[\tau]$. This is relevant to 
our measure theoretic results. For $(m,n,\v) \in \barK \times N$, we define the return time function
$$R(m,n,\v)=\min~\{j>0~:~\sC^j(m,n,\v)\in \barK \times N\}.$$
For $k=R(m,n,\v)$, we have $\hsC(m,n,\v)=\sC^k(m,n,\v)$. See equation \ref{eq:sR}. 

We can describe the return time in terms of the number of nearby rows and columns excised to produce $[\tau']$. 
To explain this we define a new {\em excision} function 
$$E:\barK \times N \to \Z_+; \quad E(m,n,\v)=\min~\{j>0~:~(m,n)+j\v \in \barK\}.$$
This represents one more than the number of adjacent rows or columns that will be removed, starting with the square opposite the edge in direction $\v$.
This is always well defined so long as $\omega$ and $\eta$ satisfy \ref{eq:assumption}.

\begin{theorem}[Tiling Return Time]
\name{thm:return time}
Suppose $\omega, \eta \in \Omega_\pm$ satisfy the assumption given in equation \ref{eq:assumption}. Choose any $\v=(a,b) \in N$. Define $s=\omega_m \eta_n$ and $\bw=(sb,sa) \in N$. 
Then,
$$R(m,n,\v)=2E(m,n, \bw)-1.$$
\end{theorem}

\begin{remark} By equation \ref{eq:C}, $\bw=(sb,sa)$ is the directional component of $\sC(\omega, \eta, \v)$.
\end{remark} 

\subsection{Proofs}\name{sect:ren proofs}
In this section we prove the renormalization theorems of the previous subsection. As above, we fix $\omega$ and $\eta$.

First we investigate loops visiting four squares in the tiling.
\begin{proposition}
If $\sC^4(m,n,\v)=(m,n,\v)$, then for all $k$ we have $\sC^k(m,n,\v) \not \in \barK \times N$.
\end{proposition}
\begin{proof}
Suppose $(m,n,\v)$ is tangent to a loop of length four. Four Truchet tiles coming together
to make a loop of length four come in exactly one configuration. Since the map $(\omega, \eta) \to \tau_{\omega, \eta}$ is two-to-one, there are exactly two local choices of $\omega$ and $\eta$ which give rise to a loop of length four. These choices are shown below:
\begin{center}
\includegraphics[scale=0.5]{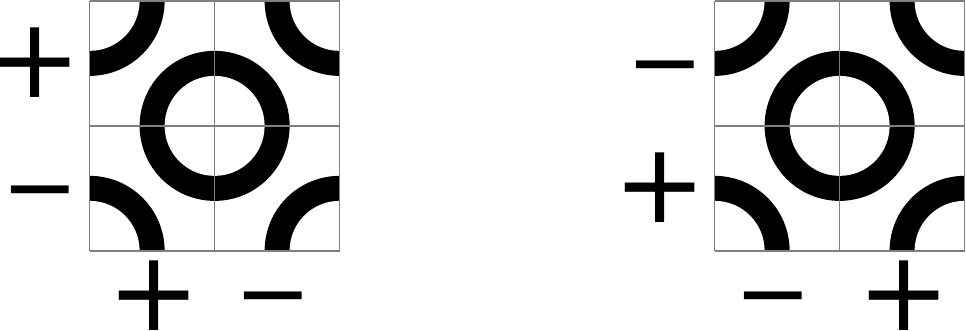}
\end{center}
All of the squares in either of these pictures lie in $\Z^2 \sm \barK$.
\end{proof}

We will now explain another possibility for what the curve through $(m,n,\v)$ looks like assuming $(m,n) \not \in \barK$.

\begin{definition}
A {\em horizontal box} is a subset of $\Z^2$ of the form
$$H=\big\{(i,j) \in \Z^2~:~\text{$i \in \{m+1, \ldots, m+2\ell\}$ and $j \in \{n, n+1\}$}\big\},$$ 
where $\ell, m, n \in \Z$ are constants with $\ell \geq 1$ so that 
$$\omega(m+i)=(-1)^i \quad \text{for $i=1, \ldots, 2\ell$}, \and \eta(n)=\eta(n+1).$$
A {\em vertical box} is a subset of $\Z^2$ of the form
$$V=\big\{(i,j) \in \Z^2~:~\text{$i \in \{m, m+1\}$ and $j \in \{n+1, \ldots, n+2\ell\}$}\big\},$$ 
where $\ell, m, n \in \Z$ are constants with $\ell \geq 1$ so that 
$$\omega(m)=\omega(m+1), \and \eta(n+i)=(-1)^i \quad \text{for $i=1, \ldots, 2\ell$}.$$
In both cases, we call $\ell$ the {\em length parameter} of the box.
\end{definition}

The tiles whose centers belong to a horizontal box must look like one of the following cases when $\ell=3$:
\begin{center}
\includegraphics[scale=0.5]{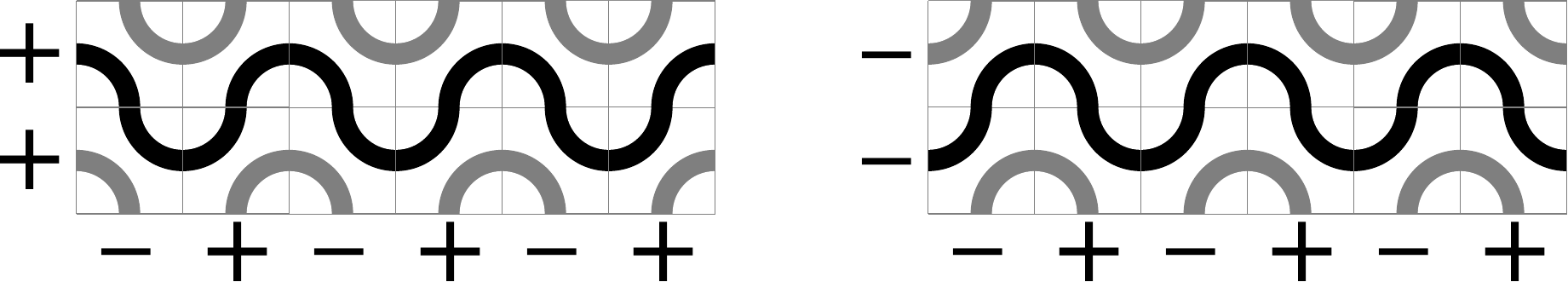}
\end{center}
Each horizontal box has a {\em central curve}, which visits all squares with centers in the horizontal box. 
This curve is depicted in black above.
The collection of tiles whose centers lie in a vertical box looks the same as the above pictures after applying a reflection in the line $x=y$. 

\begin{lemma}
\name{lem:xor}
Suppose $(m,n) \in \Z^2 \sm \barK$ and $\v \in N$. Then exactly one of the following statements holds.
\begin{enumerate}
\item $\sC^4(m,n,\v)=(m,n,\v).$
\item $(m,n,\v)$ is tangent to the central curve of a horizontal box.
\item $(m,n,\v)$ is tangent to the central curve of a vertical box.
\end{enumerate}
\end{lemma}
\begin{proof}
Suppose $(m,n) \in \Z^2 \sm \barK$. This means that either $m \not \in K(\omega)$ or $n \not \in K(\eta)$. By reflection in the line $y=x$, we may assume without loss of generality that $m \not \in K(\omega)$. This means that there is a choice of $m' \in \{m-1,m\}$ so that 
\begin{equation}
\name{eq:omp}
\omega_{m'}=-1 \and \omega_{m'+1}=1.
\end{equation}
Assuming this, we can draw all tiles with centers in the set $\{m', m'+1\} \times \{n-1,n,n+1\}$. There are eight possibilities:
\begin{center}
\includegraphics[scale=0.5]{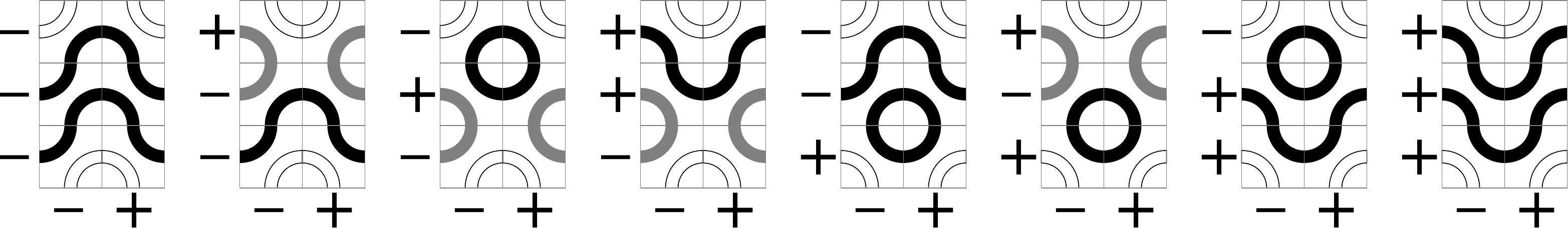}
\end{center}
We have colored the tilings by the following rules. All curves through $(m',n)$ and $(m'+1,n)$ have been colored black or gray. The black curves
are either closed loops of length four, or they are central curves of a horizontal box (with $\ell=1$). The gray curves are not yet part of a horizontal or vertical box and we need to do further analysis. The curves drawn in white and outlined are irrelevant to us because they do not (locally) pass through the tiles with centers $(m',n)$ or $(m'+1,n)$.

We further analyze the gray curves which come in pairs as above. Each gray curve visits two tiles of six in the above picture. 
For each gray curve, there is a choice of $n' \in \{n-1,n\}$ so that 
the curve visits only tiles with centers in the set $\{m',m'+1\} \times \{n',n'+1\}$. Furthermore we have 
$$\eta_{n'}=-1 \and \eta_{n'+1}=1.$$ 
Now we consider extending the tiling left and right. There are a total of four ways to extend depending on the choices of $\omega_{m'-1}$ and 
$\omega_{m'+2}$. The four possible collections of tiles with centers in the set $\{m'-1,m',m'+1, m'+2\} \times \{n',n'+1\}$ are show below:
\begin{center}
\includegraphics[scale=0.5]{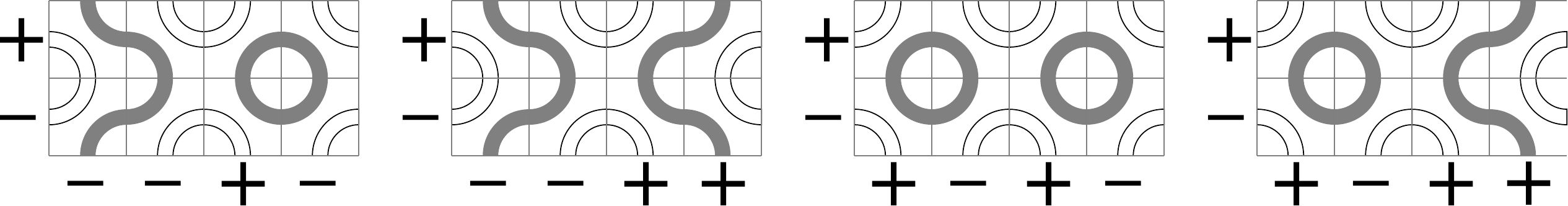}
\end{center}
Observe that in all cases, the gray curve is either a closed loop of length four, or is a central curve in a vertical box (with $\ell=1$). 

The above argument shows that each $(m,n,\v)$ satisfies one of the three statements in the lemma. We need to show the statements are mutually exclusive.
Clearly when $\sC^4(m,n,\v)=(m,n,\v)$, we can not have that $(m,n,\v)$ is tangent to a curve in a horizontal or vertical box. 
Now suppose that $(m,n,\v)$ was tangent to central curves of both horizontal and vertical boxes. 
Because of the $(m,n)$ lies in the horizontal box, there is an
$m' \in \{m-1,m\}$ so that equation \ref{eq:omp} holds. 
Because $(m,n)$ lies in a vertical box, there is an $m'' \in \{m-1,m\}$ so that 
$\omega_{m''}=\omega_{m''+1}$. This leaves two possibilities:
$$\omega_{m-1} \omega_m \omega_{m+1}=-++ \quad \text{or} \quad \omega_{m-1} \omega_m \omega_{m+1}=--+.$$
A similar argument shows that
$$\eta_{n-1} \eta_n \eta_{n+1}=-++ \quad \text{or} \quad \eta_{n-1} \eta_n \eta_{n+1}=--+.$$
Therefore, the tiles with centers in $\{m-1,m,m+1\} \times \{n-1,n,n+1\}$ have the following four possible configurations:
\begin{center}
\includegraphics[scale=0.5]{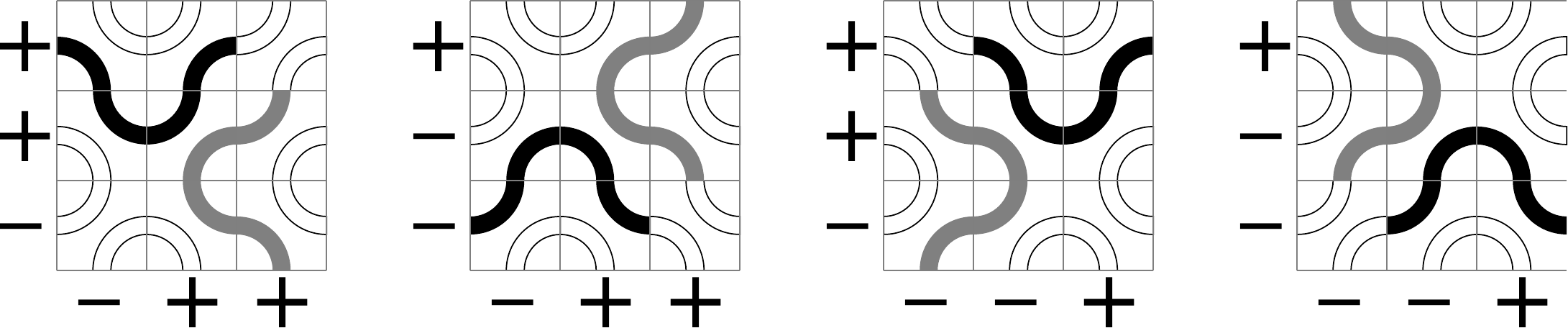}
\end{center}
In the above pictures, the central curve of the horizontal box containing $(m,n)$ is colored black, and the central curve of the vertical box containing
$(m,n)$ is colored gray. Observe that these curves are disjoint. This implies statements (2) and (3) are mutually exclusive.
\end{proof}

We call a horizontal (resp. vertical) box {\em maximal} if it is not contained in a larger horizontal (resp. vertical) box. 

\begin{proposition}
Assume $\omega, \eta \in \Omega_\pm$ satisfy the assumption given in equation \ref{eq:assumption}.
Then, every horizontal (resp. vertical) box is contained in a maximal horizontal (resp. vertical) box. 
\end{proposition}
\begin{proof}
Suppose a horizontal box was not contained in a largest maximal box. Then it would be contained in arbitrary large horizontal box.
Let $(m,n)$ be the point in the box with smallest coordinates. Then, we see that there are arbitrary long intervals $I$ containing $m$
so that $\omega$ alternates on $I$. But this is ruled out by the assumption given in equation \ref{eq:assumption}. A similar statement holds for vertical boxes.
\end{proof}

\begin{proposition}
\name{prop:leaving boxes}
Suppose $(m,n,\v)$ is tangent to the central curve in a maximal horizontal or vertical box $B$. Then, the smallest $k>0$ so that 
$\sC^k(m,n,\v)$ is no longer tangent to the central curve of $B$ satisfies $\sC^k(m,n,\v) \in \barK \times N$. Similarly, the largest $k<0$
so that $\sC^k(m,n,\v)$ is no longer tangent to the central curve of $B$ satisfies $\sC^k(m,n,\v) \in \barK \times N$.
\end{proposition}
\begin{proof}
We prove the statement for $k>0$; the other statement has a similar proof.
Let $(m',n',\v')=\sC^k(m,n,\v)$. Then $\v'$ is horizontal if $B$ is horizontal, and $\v'$ is vertical if $B$ is vertical. 
Suppose without loss of generality that $B$ and $\v'$ are horizontal. If $(m',n') \not \in \barK$, then Lemma \ref{lem:xor} implies that
$(m',n')$ is tangent to the central curve of a new horizontal or vertical box $B'$ and that the central curves of $B$ and $B'$ are disjoint. 
Since $(m',n',\v')$ is the initial entrance to the horizontal box $B'$ and $v'$ is horizontal, we know that $B'$ is horizontal. 
Observe that horizontal boxes can be joined so that their central curves connect only if $B \cup B'$ is a larger horizontal box. This contradicts maximality of $B$. 
\end{proof}

We now can prove our renormalization theorems.

\begin{proof}[Proof of Theorems \ref{thm:tiling renormalization} and \ref{thm:return time}.]
Statement (1) of the theorem follows from statement (2). We will now simultaneously prove statement (2) of the Tiling Renormalization Theorem and the Return Time Theorem. Choose any $(m',n',\v) \in \Z^2 \times N$ and write $\v=(a,b) \in N$. Define
$$(m,n)=\kappa(m',n'), \quad s=\tau'(m',n')=\tau(m,n) \and \bw=(sb,sa).$$
Then we have
$$\sC'(m',n',\v)=\big((m',n')+\bw,\bw\big)
\and
\sC\big(m,n,\v)=\big((m,n)+\bw,\bw\big).$$
The two statements we wish to prove follow respectively from
$$\hsC(m, n, \v)=\til \kappa\big((m', n')+\bw, \bw) \and R(m,n, \v)=2E\big((m,n)+\bw,\bw\big)-1.$$
For these equations, we may assume without loss of generality that
$\v$ is vertical. This means that $\bw$ is horizontal and we can write $\bw=(c,0)$ taking $c=sb \in \{\pm 1\}$ and $a=0$. 

First the consider the case that $(m,n)+\bw \in \barK$. This is the center of the square containing 
$\sC(m,n,\v)$, which means 
$$R(m, n,\v)=1 \and E(m,n,\bw)=1,$$
proving this case of the Return Time Theorem. In addition, we have 
$$\hsC(m, n, \v)=\sC(m, n, \v)=(m+c,n,\bw).$$
Because both $m \in K(\omega)$ and $m+c\in K(\omega)$ with $c \in \{\pm 1\}$, we 
have 
$$\kappa_1^{-1}(m+c)=\kappa_1^{-1}(m)+c=m'+c,$$ 
because $\kappa_1:\Z \to K(\omega)$ is an order preserving bijection. We have therefore
shown a special case of statement (2) of the Tiling Renormalization theorem, 
$$\hsC(m, n, \v)=(m+c,n,\bw)=\til \kappa\big((m', n')+\bw, \bw).$$

Otherwise we have $(m,n)+\bw \not \in \barK$. Here, $\sC(m,n,\v)$ is tangent to the central curve of a maximal horizontal
or vertical box $B$. Observe that $(m,n) \in \barK$, so $\sC(m,n,\v)=(m+c,n,\bw)$ is the first time the curve enters this box. Since 
$\bw$ is horizontal, the box $B$ must be a horizontal box. Let $\ell$ denote the length parameter of the maximal horizontal box $B$. 
If $c=1$, this means that 
$$\omega_{m+k}=(-1)^k \quad \text{for $k=1, \ldots, 2\ell$}.$$
If $c=-1$, this means that 
$$\omega_{m+k-2\ell-1}=(-1)^k \quad \text{for $k=1, \ldots, 2 \ell$}.$$
Note that $\ell$ is the maximal number with this property. Therefore, $\kappa_1(m'+c)=m+c(2\ell+1)$. That is, $\kappa_1$ must skip over $2 \ell$ numbers to reach $\kappa_1(m'+c)$. The orbit $\sC^i(m,n,\v)$ follows the central curve of $B$ and then returns to
$\barK$ by Proposition \ref{prop:leaving boxes}. By inspection of horizontal boxes, we can then observe
\begin{enumerate}
\item[(a)] $R(m,n,\v)=1+4 \ell.$
\item[(b)] $\hsC(m,n,\v)=\sC^{1+4\ell}(m,n,\v)=\big(m+c(2\ell+1),n,\bw\big)$.
\item[(c)] $E(m,n,\bw)=2 \ell+1.$
\end{enumerate}
Statements (a) and (c) imply $R(m,n, \v)=2E(m,n,\bw)-1$. By (b) and observations above, 
$$\hsC(m, n, \v)=\big(m+c(2\ell+1),n,\bw\big)=\til \kappa(m+c,n,\bw).$$
This finishes the proof of statement (2) of the Tiling Renormalization Theorem and proof of the Return Time Theorem.

Statement (3) of Theorems \ref{thm:tiling renormalization} follows from Lemma \ref{lem:xor} and Proposition \ref{prop:leaving boxes}.
If $(m,n) \not \in \barK$ and $\sC^4(m,n,\v) \neq (m,n,\v)$ then $(m,n,\v)$ is tangent to a central curve of a maximal horizontal or vertical box. Under positive or negative iteration by $\sC$ it must leave the box, and when it does it enters the set $\barK \times N$. 

We now consider statement (4). Suppose $(m,n,\v)$ is periodic under $\sC$ and never visits the set $(\Z^2 \sm \barK) \times N$. Then this periodic orbit is confined to a region of the tiling consisting of tiles with centers in the set 
$$X=\{a, a+1, \ldots, b\} \times \{a', a'+1, \ldots, b'\},$$
where there are $c$ and $c'$ so that for all $(m,n) \in X$
$$\omega(m)=\begin{cases}
1 & \textrm{if $m < c$} \\
-1 & \textrm{if $m \geq c$}
\end{cases}
\and
\eta(n)=\begin{cases}
1 & \textrm{if $n < c'$} \\
-1 & \textrm{if $n \geq c'$.}
\end{cases}$$
But such a portion of a tiling can have no closed curves. See the example below.
\begin{center}
\includegraphics[scale=0.5]{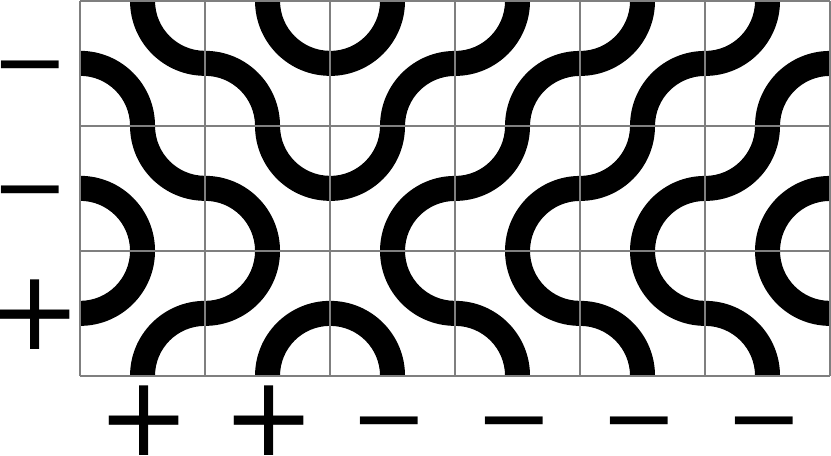}
\end{center}
\end{proof}

\section{Dynamical Renormalization}
\name{sect:corner}
This section culminates in a description of a renormalization
of the dynamical system $\Phi:X \to X$ defined in Equation \ref{eq:Phi2}. 

\subsection{Background on shift spaces}
\name{sect:topology}
Recall that $\Omega_\pm$ denotes the space of all bi-infinite sequences in the alphabet $\{\pm 1\}$. We will now describe some of the general
structure associated with shift spaces in this context. For further background on shift spaces see \cite{LindMarcus}, for instance.

A {\em word} in the alphabet $\{\pm 1\}$ is an element $w$ of a set $\{\pm 1\}^{\{1,\ldots,n\}}$ for some $n$, called the {\em length} of $w$. 
We write $w=w_1 \ldots w_n$ with $w_i \in \{\pm 1\}$ to denote a word.
To simplify notation of the elements in $\{\pm 1\}$, we use $+$ to denote $1$ and $-$ to denote $-1$. So the word 
$w$ where $w_1=1$ and $w_2=-1$ can be written $w=+-$. 
Adjacency indicates the {\em concatenation} of words; if $w$ and $w'$ are words of length $n$ and $n'$ respectively, then 
$$ww'=w_1 \ldots w_{n} w'_1 \ldots w'_{n'}.$$ 

The choice of a word $w=w_1 \ldots w_n$ and an integer $b$ determines a {\em cylinder set},
$$\cyl(w,b)=\set{\omega \in \Omega_{\pm}}{$\omega_{i-b}=w_{i}$ for all $i=1, \ldots, n$}.$$
Whenever $b \in \{1, \ldots, n\}$, we can also denote the cylinder set $\cyl(w,b)$ by
$$\cyl(w_1 \ldots \widehat w_b \ldots w_n),$$ 
with the hat indicating that $w_b$ represents the zeroth entry of the 
those $\omega$ in the cylinder set.
We equip $\Omega_\pm$ with the topology generated by the cylinder sets. The topological space $\Omega_\pm$ is homeomorphic to a Cantor set.

Recall that the shift map $\sigma:\Omega_\pm \to \Omega_\pm$ is defined by $\sigma(\omega)_n=\omega_{n+1}$ as in Equation \ref{eq:shift}.
A {\em shift-invariant measure} on $\Omega_\pm$ is a Borel measure $\mu$ satisfying 
$$\mu \circ \sigma^{-1}(A)=\mu(A) \quad \textrm{for all Borel subsets $A \subset \Omega_\pm$}.$$
Full shift spaces admit a plethora of shift-invariant probability measures.

\subsection{Invariant measures}
\name{sect:tilings from shift spaces}
Recall the definition of $\Phi:X \to X$ where  $X=\Omega_\pm \times \Omega_\pm \times N$ as in equation \ref{eq:Phi2},
\begin{equation}
\name{eq:Phi2b}
\Phi\big(\omega, \omega', (a,b)\big)=\big(\sigma^{sb} (\omega), \sigma^{sa} (\omega'), s(b,a)\big) \quad \textrm{with $s=\omega_0 \omega'_0 \in \{\pm 1\}$.}
\end{equation}

The following gives a natural construction of $\Phi$-invariant measures.

\begin{proposition}
\name{prop:product measures}
Suppose $\mu$ and $\mu'$ are shift invariant probability measures on $\Omega_\pm$. Let $\mu_N$ be the discrete probability measure
on $N$ so that $\mu_N(\{\v\})=\frac{1}{4}$ for each $\v \in N$. Then $\mu \times \mu' \times \mu_N$ is a $\Phi$-invariant probability measure on $X$.
\end{proposition}

The proof is just to observe that each Borel set $A \subset X$ can be decomposed into pieces on which the action of $\Phi$ is
a power of a shift on each $\Omega_\pm$-coordinate and a permutation on $N$. The power and permutation are taken to be constant on each piece.

\subsection{Periodic orbits}
\name{sect:periodic orbits}
Suppose $(\omega, \eta,\v) \in X$ is periodic under $\Phi$. We say $(\omega, \eta,\v)$ has a {\em stable periodic orbit of period $n$} if $n$ is the smallest positive integer for which there are open neighborhoods $U$ and $V$ of $\omega$ and $\eta$ respectively for which 
$$\omega' \in U \and \eta'\in V \quad \text{implies} \quad \Phi^n(\omega', \eta',\v)=(\omega', \eta',\v).$$

\begin{remark}
Not all periodic orbits are stable. When $\omega_n=1$ and $\eta_n=1$ for all $n \in \Z$, we have
$\Phi^2(\omega, \eta,\v)=(\omega, \eta,\v)$ for all $\v$, but $(\omega, \eta,\v)$ is not a stable 
periodic orbit of any period.
\end{remark}

The following proposition characterizes the points with stable periodic orbits.

\begin{proposition}[Stability Proposition]
\name{prop:stable}
The following statements hold.
\begin{enumerate}
\item $(\omega, \eta,\v) \in X$ has a stable periodic orbit if and only if the curve of the tiling $[\tau_{\omega, \eta}]$ passing through the normal $\v$ to the square centered at the origin is closed. 
\item If $(\omega, \eta,\v) \in X$ has a periodic orbit but not a stable periodic orbit, then either $\omega$
or $\eta$ is periodic under the shift map $\sigma$. 
\end{enumerate}
\end{proposition}

%

\begin{proof}[Proof of Proposition \ref{prop:stable}.]
First suppose the curve of the tiling $[\tau_{\omega,\eta}]$ through the normal $\v$ to the square centered at the origin is closed. 
There are integers $m$ and $n$ so that all tiles visited by this closed curve have centers in the set $[-m,m] \times [-n,n]$. We define
$$U=\cyl(\omega_{-m} \omega_{-m+1} \ldots \wh \omega_0 \ldots \omega_m) \and 
V=\cyl(\eta_{-n} \eta_{-n+1} \ldots \wh \eta_0 \ldots \eta_n).$$
Observe that every tiling determined by $\omega' \in U$ and $\eta' \in V$ looks the same for the set of tiles with centers in $[-m,m] \times [-n,n]$.
In particular, every such tiling has the same closed curve through the normal $\v$ to the square centered at the origin. This always
gives a periodic orbit of the same period as $(\omega, \eta, \v)$.  

Now suppose $(\omega, \eta, \v)$ has period $k$ but the associated curve of the tiling $[\tau_{\omega,\eta}]$ is not closed. 
Recall the definition of the curve following map given in Section \ref{sect:curve following}. 
Define $m$ and $n$ so that the curve following map for $[\tau_{\omega, \eta}]$ satisfies
$\sC^k(0,0,\v)=(m,n,\v)$. Because the loop has not closed, $m \neq 0$ or $n \neq 0$. 
But because $(\omega, \eta,\v)$ has period $k$, we have
$\sigma^m(\omega)=\omega$ and $\sigma^n(\eta)=\eta$. 
See equation \ref{eq:curve following relation}. So, $\omega$ is periodic or $\eta$ is periodic. 
We can see that $(\omega, \eta, \v)$ does not have a stable periodic orbit, since we can always perturb $\omega$ and $\eta$
within any $U$ and $V$ to destroy periodicity but to ensure that the curve following map $\sC_0$ of the perturbed tiling satisfies $\sC^k_0(0,0,\v)=(m,n,\v)$.
\end{proof}

\begin{remark}[Closed curves in the arithmetric graph]
In polygonal billiards and polygonal outer billiards, a periodic orbit is called {\em stable} if periodic paths with the same combinatorial type do not disappear when sufficiently small changes are made to the polygon. 
The fact that closed curves in the arithmetic graph correspond to stable periodic orbits also holds true in the study of outer billiards in polygons. See \cite{S09} for the case when the polygon is a kite.
A periodic billiard path in a triangle gives rise to a so-called hexpath in the hexagonal tiling of the plane. This hexpath is always periodic up to a translation, and
the periodic billiard path is stable if and only if this translation is trivial, i.e. the hexpath closes up. See \cite{HooSch09}. Both these statements have generalizations to all polygons which can be obtained
by appropriately interpreting known combinatorial criteria for stability. See \cite{Tab} for these combinatorial criteria.
\end{remark}

\subsection{The collapsing map}
\name{sect:collapsing}
In the tiling renormalization procedure described in section \ref{sect:renormalizing tilings}, we took any $\omega$ and $\eta$ in $\Omega_\pm$
and removed all subwords of the form $-+$ to build new elements $\omega'$ and $\eta'$ in $\Omega_\pm$. The tiling $[\tau_{\omega',\eta'}]$ was shown to 
have a similar structure to the tiling $[\tau_{\omega,\eta}]$. The choice of $\omega'$ and $\eta'$ was only canonical up to a power of the shift map. In order to use this tiling renormalization procedure to understand the map $\Phi$ will will need to make the choice canonical. We do this via a map we call the {\em collapsing map}. 

The idea of the collapsing function $c$ mentioned at the beginning of this section is to remove any substrings of the form $-+$ and then slide the remaining
entries together toward the zeroth entry. For example,
$$c(\ldots \underline{-+}+--\underline{-+}\underline{-+}\widehat+-\underline{-+}++\ldots)=\ldots +--\widehat+-++\ldots,$$
where underlined entries have been removed. There are two potential reasons why $c(\omega)$ may not be well defined. 
First, the zeroth entry might be removed by this process, so we lose track of the indexing. Second, the remaining list may not be bi-infinite.

We will now build up to a formal definition of the collapsing map. We define the set $S \subset \Omega_\pm$ to be the union of two cylinder sets,
$$S=\cyl(\widehat{-}+) \cup \cyl(-\widehat{+}).$$ We can restate the definition of the set $K(\omega)$ given in equation \ref{eq:K1} as
\begin{equation}
\name{eq:K2}
K(\omega)=\{k \in \Z~:~\sigma^k(\omega) \not \in S\}.
\end{equation}
We call $\omega$ {\em unbounded-collapsible} if $K(\omega)$ has no upper nor lower bound. 
Our definition of $\omega'$ depended on an order preserving bijection $\Z \to K(\omega)$. 
Such a bijection is guaranteed to exist if $\omega$ is unbounded-collapsible, but
there are many possible choices. If $0 \in K(\omega)$, we call $\omega$ {\em zero-collapsible}
and define $i \mapsto k_i$ to be the unique order preserving bijection $\Z \to K(\omega)$ so that 
$k_0=0$. We call $\omega$ {\em collapsible} if it is both unbounded- and zero-collapsible. We use $C \in \Omega_\pm$ to denote
the set of collapsible $\omega$, and define the collapsing map to be 
$$c:C \to \Omega_\pm; \quad [c(\omega)]_i=\omega_{k_i}.$$

We briefly record some properties of the collapsing map.
\begin{theorem}[Properties of the collapsing map]\quad
\name{thm:collapsing}
\begin{enumerate}
\item The map $c:C \to \Omega_\pm$ is a continuous surjection.
\item \name{item:crenormalization}
If $\wh \sigma:C \to C$ is the first return map of $\sigma$ to $C$, then 
$$c \circ \wh \sigma(\omega)= \sigma \circ c(\omega) \quad \text{for all $\omega \in C$.}$$
\item \name{item:collapsing measures}
If $\mu$ is a shift-invariant measure on $\Omega_\pm$ then so is $\mu \circ c^{-1}$. 
\item \name{item:omegaalt2}
Define $\omegaalt \in \Omega_\pm$ by $\omegaalt_n=(-1)^n$. 
If $\mu$ is a finite shift-invariant measure on $\Omega_\pm$, then 
$$\mu(\set{\omega \in \Omega_\pm}{$\omega$ is not unbounded-collapsible})=2 \mu(\{\omegaalt\}).$$
\end{enumerate}
\end{theorem}
\begin{proof}[Sketch of proof.]
Suppose $\eta \in \Omega_\pm$. Then the collection of preimages, $c^{-1}(\eta)$, is contained in the collection of all $\omega \in \Omega_\pm$ obtained by 
inserting a non-negative power of the word $-+$ between each of the symbols in $\eta$. The only restriction is that a positive power must be inserted between every pair of symbols of the form $-+$. In particular, $c$ is a surjection. This discussion can also be used to prove that
the preimage of a cylinder set is a union of cylinder sets intersected with $C$. So, $c$ is continuous.

To see statement (\ref{item:crenormalization}), observe that $\wh \sigma(\omega)=\sigma^n(\omega)$ where $n$ is the smallest positive entry in $K(\omega)$. The proof then follows from the definition of the collapsing map. 

Statement (\ref{item:collapsing measures}) follows from two observations. The restriction of a $\sigma$-invariant measure to $C$ is $\widehat \sigma$-invariant. The 
pullback of a $\widehat \sigma$-invariant measure under $c$ is $\sigma$-invariant
by (\ref{item:crenormalization}).

Statement (\ref{item:omegaalt2}) follows from the Poincar\'e Recurrence Theorem. If $\omega$ is not unbounded-collapsible, then 
$\sigma^n(\omega)$ converges to the periodic orbit $\{\omegaalt, \sigma(\omegaalt)\}$ either as $n \to +\infty$ or $n \to -\infty$. The Poincar\'e 
Recurrence Theorem implies that the set of $\omega$ which are not unbounded-collapsible
and do not belong to $\{\omegaalt, \sigma(\omegaalt)\}$ has $\mu$-measure zero. 
\end{proof}

We close with the definition of two functions which will be important in the next subsection.
These are the forward and backward return times of $\sigma$ to $C$.
\begin{equation}
\name{eq:r+}
\begin{array}{c}
\displaystyle r_+:C \to \Z_+; \quad r_+(\omega)=\min \{n>0~:~\sigma^n(\omega) \in C\}.\\
\displaystyle r_-:C \to \Z_+; \quad r_-(\omega)=\min \{n>0~:~\sigma^{-n}(\omega) \in C\}.
\end{array}
\end{equation}
Observe that these functions are well-defined for every $\omega \in C$.

\subsection{Renormalization Theorems}
\name{sect:renormalization2}
In this section, we describe general renormalization results for the map $\Phi:X \to X$, where $X=\Omega_\pm \times \Omega_\pm \times N$.

Define $\sR_1 \subset X$ to be the set of ``once renormalizable'' elements of $X$,
\begin{equation}
\name{eq:R1}
\sR_1=C \times C \times N.
\end{equation}
That is, $\sR_1$ is the collection of all $(\omega, \eta, \v)$ where $\omega$ and $\eta$ are both collapsible. The renormalization mentioned is the map
\begin{equation}
\name{eq:rho}
\rho:\sR_1 \to X; \quad (\omega, \eta, \v) \mapsto \big(c(\omega), c(\eta), \v\big).
\end{equation}
The manner in which $\rho$ renormalizes the map $\Phi$ is described by the theorem below.

Before stating the theorem, we define some important subsets of $X$:
$$P_4=\{\textrm{$x\in X$~:~$x$ has a stable periodic orbit of period $4$}\}.$$
$$\NUC=\set{(\omega, \eta, \v)\in X}{either $\omega$ or $\eta$ is not unbounded-collapsible}.$$
The points in $P_4$ correspond to loops in a tiling of smallest possible size. 
The points in $\NUC$ consist of all $(\omega, \eta,\v)$ so that $\omega$ and $\eta$ fail to satisfy the assumption \ref{eq:assumption}
necessary for the Tiling Renormalization Theorem of Section \ref{sect:renormalizing tilings} to hold. 
With this in mind, we restate that theorem in this context.

\begin{theorem}[Dynamical Renormalization]
\name{thm:renormalization} \hspace{1em} 
\begin{enumerate}
\item The first return map $\wh \Phi:\sR_1 \to \sR_1$ of $\Phi$ to $\sR_1$ is well defined and invertible.
\item If $x \in \sR_1$, we have $\rho \circ \wh \Phi(x)=\Phi \circ \rho(x).$
\item The following statements are equivalent for any $x \in X \sm \NUC$.
\begin{enumerate}
\item There is no $k>0$ so that $\Phi^k(x) \in \sR_1$. 
\item There is no $k<0$ so that $\Phi^k(x) \in \sR_1$. 
\item $x \in P_4$.
\end{enumerate}
\item A point $x \in \sR_1$ has a stable periodic orbit if and only if $\rho(x)$ has a stable periodic orbit. Moreover, $\rho(x)$
has strictly smaller period than $x$.
\end{enumerate}
\end{theorem}

We omit the proof of this theorem. It follows from Theorem \ref{thm:tiling renormalization}
using the connection between curve following and the map $\Phi$ described in Section \ref{sect:curve following}. See equation \ref{eq:curve following relation}.

Statements (3) and (4) of the Renormalization Theorem are useful for detecting stable periodic orbits. A periodic orbit is shortened when applying $\rho$. 
If we can apply $\rho$ infinitely many times, then eventually the orbit becomes period four, and then the orbit vanishes under one more application of $\rho$. This is the basic observation enabling us to compute the total measures of periodic points for some measures.

For applications, we will need to compute the return time function $R_1:\sR_1 \to \Z_+$ of $\Phi$ to $\sR_1$. We do this in terms of the functions $r_+$ and $r_-$ defined in equation \ref{eq:r+} below.

\begin{lemma}[Dynamical Return Time]
\name{lem:returns}
Fix $(\omega, \eta, \v) \in \sR_1$. Let $(a,b)=\v$. Define $s=\omega_0 \eta_0$ and
$\bw=(sb,sa)$. Then,
$$R_1(\omega, \eta, \v)=\begin{cases}
2 r_+(\omega)-1 & \text{if $\bw=(1,0)$,} \\
2 r_-(\omega)-1 & \text{if $\bw=(-1,0)$,} \\
2 r_+(\eta)-1 & \text{if $\bw=(0,1)$,} \\
2 r_-(\eta)-1 & \text{if $\bw=(0,-1)$.}
\end{cases}$$
\end{lemma}
This lemma follows directly from Theorem \ref{thm:return time}, so we omit the proof.

\section{Renormalization of the Rectangle Exchange Maps}
\name{sect:rectangle exchanges}

In this section, we explain how the renormalization of the map $\Phi$ described in section \ref{sect:renormalization2} induces a renormalization of the polygon exchange maps $\til \Psi_{\alpha, \beta}$ defined in the introduction. 
The first subsection provides necessary prerequisite details involving coding of rotations.

\subsection{Coding Rotations}
\name{sect:rot}
Let $\alpha \in \R$. The rotation by $\alpha$ is the map
$$T_\alpha:\R/\Z \to \R/\Z; \quad x \mapsto x+\alpha.$$ 
Given any $x \in \R/\Z$ we construct an element of $\Omega_\pm$ via coding,
$$\varsigma:\R/\Z \to \Omega_\pm; \quad
\varsigma(x)_n= 
\begin{cases} 1 & \textrm{if $T_\alpha^n(x) \in [0, \half)$}\\
-1 & \textrm{otherwise.}
\end{cases}$$
Observe that $\varsigma$ semiconjugates the rotation to the shift map on $\Omega_\pm$. That is,
\begin{equation}
\name{eq:semiconj}
\sigma \circ \varsigma(x)=\varsigma \circ T_\alpha(x) \quad \text{for all $x \in \R/\Z$.} 
\end{equation}
The map $\varsigma$ is an embedding so long as $\alpha$ is irrational.
Since Lebesgue measure $\lambda$ is invariant under $T_\alpha$, we can pull back Lebesgue measure to obtain a $\sigma$-invariant measure on $\Omega_\pm$, namely
\begin{equation}
\name{eq:mu alpha}
\mu_\alpha=\lambda \circ \varsigma^{-1}.
\end{equation}


Recall that a rotation is conjugate to its inverse via an orientation reversing isometry of the circle. Moreover, if we choose the particular orientation reversing isometry
$$\iota:t \mapsto \half-t \pmod{1},$$ 
we see that $\varsigma \circ \iota$ is the coding map of $T_{-\alpha}$ (modulo a set of Lebesgue measure zero consisting of the orbits of $0$ and $\half$). In particular, the two measures
$\mu_\alpha$ and $\mu_{-\alpha}$ are equal. Because we will be primarily interested in the measures which arise from this construction, it is natural for us to only consider rotations $T_\alpha$ with $\alpha \in [0,\half]$.

This observation explains the connection between rotations and the group $G$ of isometries of $\R$ preserving $\Z$. Explicitly, $G$ is the group of maps of the form 
$$g:\R \to \R; \quad t \mapsto rt+n \quad \text{with $n \in \Z$ and $r \in \{\pm 1\}$.}$$
For the following theorem, we will need to make more observations and definitions involving this group. The interval $[0,\half]$ is a fundamental domain for the $G$-action on $\R$. We define the map
$$o:\R \to \{\pm 1\}; \quad t \mapsto \begin{cases} 
1 & \text{if $\exists n \in \Z$ so that $t+n \in [0, \half]$} \\
-1 & \text{otherwise.}
\end{cases}
$$
This map records the orientation of the element $g \in G$ which caries $t$ into $[0,\half]$.
If there is ambiguity, the map chooses positive sign.

Recall the definition of the collapsible elements $C \subset \Omega_\pm$ and the collapsing map $c:C \to \Omega_\pm$. The shift map $\sigma$ on $\Omega_\pm$ was renormalized in a sense by the collapsing map, because the collapsing map semiconjugates the first return $\wh \sigma$ of the shift map to $C$ to the shift map. See statement (\ref{item:crenormalization}) of Theorem \ref{thm:collapsing}. 

The following theorem explains how the collapsing map interacts with the rotation via coding. The theorem observes the existence of a renormalization in the sense used in the theory of interval exchange maps. In this setting a {\em renormalization} is simply a return map to an interval which is conjugate up to a dilation to an interval exchange map on the same number of intervals. (A rotation is an interval exchange defined using two intervals.)

\begin{theorem}[Rotation Renormalization] \name{thm:rotren} Assume $\alpha \in [0, \half)$. 
\begin{enumerate}
\item \name{item:collapsing interval} The preimage of the collapsible sequences, $\varsigma^{-1}(C)$, is the interval $C_\alpha=[\alpha, 1-\alpha)$. 
\item \name{item:ret id} In particular, the first return map $\wh T_\alpha$ of the rotation $T_\alpha$ to $C_\alpha$ satisfies 
$$\varsigma \circ \wh T_\alpha(x)=\wh \sigma \circ \varsigma(x) \quad \text{for all $x \in C_\alpha$.}$$
\item \name{item:iet ret} The first return map $\wh T_\alpha:C_\alpha \to C_\alpha$ is the rotation by $\alpha$ modulo $1-2\alpha$.
\item \name{item:psi} Let $\gamma=f(\alpha)$, where $f(\alpha)$ denotes the element 
of $[0, \half]$ which is $G$-equivalent to $\frac{\alpha}{1-2\alpha}$ as in equation \ref{eq:f} of the introduction. As in equation \ref{eq:psi}, define the dilation
$$\psi=\psi_\alpha:[\alpha, 1-\alpha) \to \R/\Z; \quad \psi(x)=
\begin{cases}
\frac{x-\frac{1}{2}}{1-2\alpha}+\frac{1}{2} & \textrm{if $o(\frac{\alpha}{1-2\alpha})=1$,} \\
\frac{\frac{1}{2}-x}{1-2\alpha} & \textrm{if $o(\frac{\alpha}{1-2\alpha})=-1$.} 
\end{cases}$$
This dilation has the following properties:
\begin{enumerate}
\item $\psi \circ \wh T_\alpha(x)=T_\gamma \circ \psi(x)$ for all $x \in C_\alpha$. 
\item If $\varsigma'$ is the coding map for $T_\gamma$, then $c \circ \varsigma(x)=\varsigma' \circ \psi(x)$ for $\lambda$-almost every $x \in C_\alpha$. 
\end{enumerate}
\end{enumerate}
\end{theorem}

We make several comments about this theorem. First, it should be observed that the return map $\wh T_\alpha$ defines a renormalization in the interval exchange sense. Once we know that 
$\wh T_\alpha$ is a rotation by $\alpha$ modulo $1-2\alpha$, we know that any surjective dilation $C_\alpha \to \R/\Z$ 
will conjugate $\wh T_\alpha$ to either $T_\gamma$ or $T_{-\gamma}$, depending on orientation.
However, when $\alpha$ is irrational, there is a unique choice of a dilation that respects codings as in statement (4b). Second, when $\alpha$ is irrational, we can really think of this as a pullback of the renormalization happening on $\Omega_\pm$. This is because
$\varsigma$ is injective, and $\varsigma \circ T_\alpha=\sigma \circ \varsigma$. In this case, we could alternately define 
$$\wh T_\alpha=\varsigma^{-1} \circ \wh \sigma \circ \varsigma \and
\psi=(\varsigma')^{-1} \circ c \circ \varsigma.$$

The following describes the action of the renormalizing map $c$ on measures of the form $\mu_\alpha$
as defined in equation \ref{eq:mu alpha}.

\begin{corollary}[Action on Measures]
\name{cor:collapsing rotations}
Suppose $0 \leq \alpha < \frac{1}{2}$ and let $\gamma$ be as in statement \ref{item:psi} of Theorem \ref{thm:rotren}. Then, 
$$\mu_\alpha \circ c^{-1}=(1-2\alpha) \mu_\gamma.$$
\end{corollary}
\begin{proof}
This follows from statement (\ref{item:psi}b) and the fact that the length of $C_\alpha$ is $1-2\alpha$.
\end{proof}

It will also be useful to record the values of the function $r_+$ and $r_-$ defined in equation \ref{eq:r+}. For $x \in C_\alpha$, the quantities $r_+\circ\varsigma(x)$
and $r_-\circ\varsigma(x)$ record the first return times of $T_{\alpha}$ and $T^{-1}_{\alpha}$ to $C_\alpha$, respectively.
\begin{lemma}[Rotation Return Times]
\name{lem:rrt}
For Lebesgue-almost every $x \in C_\alpha$, we have
$$r_+\circ\varsigma(x)=2 \left\lfloor \frac{x}{1-2\alpha} \right\rfloor+1 \and
r_-\circ\varsigma(x)=2 \left\lfloor \frac{1-x}{1-2\alpha} \right\rfloor+1,$$
where $\floor{t}$ denotes the greatest integer less than or equal to $t$. 
\end{lemma}

We have an alternate formula for the return times, which will be useful later.

\begin{corollary}
\name{cor:rot ret}
Suppose $f(\alpha)=r(\frac{\alpha}{1-2\alpha}-n)$ for $n \in \Z$ and $r \in \{\pm 1\}.$ 
If $r=1$, then for $\mu_\alpha$-a.e. collapsible $\omega$, we have:
$$r_+(\omega)=\begin{cases} 
2n+3 & \text{if $c(\omega) \in \cyl(\wh -+)$} \\
2n+1 & \text{otherwise,}
\end{cases} \and 
r_-(\omega)=\begin{cases} 
2n+3 & \text{if $c(\omega) \in \cyl(-\wh +)$} \\
2n+1 & \text{otherwise.}
\end{cases}$$
If $r=-1$, then for $\mu_\alpha$-a.e. collapsible $\omega$, we have:
$$r_+(\omega)=\begin{cases} 
2n-1 & \text{if $c(\omega) \in \cyl(\wh +-)$} \\
2n+1 & \text{otherwise,}
\end{cases}
\and
r_-(\omega)=\begin{cases} 
2n-1 & \text{if $c(\omega) \in \cyl(+\wh -)$} \\
2n+1 & \text{otherwise.}
\end{cases}$$
\end{corollary}

We now give proofs of the Rotation Renormalization Theorem and Rotation Return Time Lemma. We will conclude this subsection with a proof of the Corollary.

\begin{proof}[Proof of Theorem \ref{thm:rotren} and Lemma \ref{lem:rrt}]
We begin by proving statement (\ref{item:collapsing interval}) of the Theorem. 
If $x \in [1-\alpha,1)$, then $\varsigma(x) \in \cyl(\wh - +)$. And, if $x \in [0,\alpha)$, then $\varsigma(x) \in \cyl(- \wh +)$. In either case $\varsigma(x)$ is not zero-collapsible. If $x \in [\alpha, \half)$, then $\varsigma(x) \in \cyl(+ \wh +)$, and if $x \in [\half, 1-\alpha)$ then $\varsigma(x) \in \cyl(\wh - -)$. In these cases, $x$ is zero-collapsible. We also observe that $x \in [\alpha, 1-\alpha)$ is always unbounded-collapsible, because the only way an infinite sequence of alternating signs can appear from coding a $T_\alpha$ is when $\alpha=\half$. 

Statement (\ref{item:ret id}) follows from statement (\ref{item:collapsing interval}) and equation \ref{eq:semiconj}.

We now prove statement (\ref{item:iet ret}) of the Theorem and the formula for $r_+$ given in the Lemma. To do this, 
we provide an pseudo-code algorithm to produce the first return $\wh T_\alpha(x) \in [\alpha, 1-\alpha)$:
\begin{enumerate}
\item[(0)] Set $i=0$ and $x_0=x$.
\item[(1)] If $x_i+\alpha \in [\alpha, 1-\alpha)$ then $\wh T_\alpha(x)=x_i+\alpha$. Stop, because we have found $\wh T_\alpha(x)$.
\item[(2)] Set $x_{i+1}=x_i-1+2 \alpha$. 
\item[(3)] Iterate $i$. (Set $i$ to be $i+1$.) Return to step $1$. 
\end{enumerate}
For the moment assume this procedure terminates with $\wh T_\alpha(x)=x_n+\alpha$. (We prove this occurs for some $n$ below.) 
Observe that $x_i=x_0+i(2 \alpha-1)$ for all $i$, so that $x_n+\alpha$ is indeed equivalent to $x+\alpha$ modulo $1-2\alpha$. Therefore, $\wh T_\alpha$ is indeed a rotation by $\alpha$ modulo $1-2\alpha$. 

We now explain why the algorithm terminates. If it fails to terminate with $i=0$, then $x_0+\alpha \geq 1-\alpha$. Observe that $x_i+\alpha=x+i(2\alpha-1)+\alpha$ is a decreasing sequence and that $x_i-x_{i+1}=1-2\alpha$. Since the sequence $\{x_i+\alpha\}$ iteratively decreases by an amount equal to the length of $[\alpha, 1-\alpha)$, there is precisely one integer $n$ for which $x_n+\alpha \in [\alpha, 1-\alpha)$. This integer is given by $n=\lfloor \frac{x}{1-2\alpha}\rfloor$.

Now suppose $0 \leq i<n$. Then $T(x_i)=x_i+\alpha \in [1-\alpha, 1)$, and $T^2(x_i)=x_i+2\alpha-1=x_{i+1}$. 
So by induction, $T^{2i}(x)=x_i$ and $T^{2i+1}(x)=x_i+\alpha$ for $0 \leq i \leq n$. Moreover, we have $T^{2n+1}(x)=x_n+\alpha$ is the first return of $T$ to $[\alpha, 1-\alpha)$. Thus, $r_+\circ\varsigma(x)=2n+1$ as desired.

Now we verify the formula for $r_-\circ\varsigma(x)$ given in the Lemma. Observe that the map $x \mapsto 1-x$ conjugates $T_\alpha$ to $T^{-1}_\alpha$ and 
sends $C_\alpha$ to $C_\alpha$ almost-everywhere. Therefore, we have 
$$r_-\circ\varsigma(x)=r_+\big(\varsigma(1-x)\big)$$
almost everywhere. (This only fails at the point $\alpha \in C_\alpha$.)

Finally, we prove statement (\ref{item:psi}) of the Theorem. Observe that $\psi$ is a bijective dilation $C_\alpha \to \R/\Z$. By the remarks below the theorem, the dilation conjugates $\wh T_\alpha$ to a rotation by $\pm \gamma$. Since the orientation preserving nature of the element of $g \in G$ carrying $\frac{\alpha}{1-2\alpha}$ matches the orientation preserving nature of $\psi$, we know the dilation conjugates $\wh T_\alpha$ to $T_\gamma$. This proves statement (a). Statement (b) follows from the fact that the map $\psi$ respects the labeling of intervals by $\pm 1$ almost-everywhere. As in the definition
of $\varsigma$, we have labeled the interval $[0, \half)$ by $+1$ and $[\half, 1)$ by $-1$. Observe
$$\psi\big([0, {\textstyle \half}) \cap C_\alpha\big)=[0, {\textstyle \half}) \and \psi\big([{\textstyle \half},1) \cap C_\alpha\big)=[{\textstyle \half},1)$$
almost-everywhere (with ambiguities at endpoints). Since this is true almost-everywhere, statement (b) follows from statement (a). 
\end{proof}

\begin{proof}[Proof of Corollary \ref{cor:rot ret}]
Since we only need this statement $\mu_\alpha$-a.e., we can assume that $\omega=\varsigma(x)$. 
By the lemma, the formulas in the corollary for $r_+$ are equivalent to formulas for the function 
$$m(x)=\left\lfloor \frac{x}{1-2\alpha} \right\rfloor.$$
Note that $m(x)$ takes only two values on $[\alpha, 1-\alpha)$. Since $m$ is an increasing function,
these two values are $m(\alpha)$ and $m(\alpha)+1$. 

Consider the case when $r=1$. Then, 
$n \leq \frac{\alpha}{1-2\alpha} \leq n+\half$. 
The function $m(x)$ takes the two values $n$ and $n+1$, with the discontinuity happening at the point 
$y=(1-2\alpha)(n+1)$. We compute 
$$\psi(y)=n+1-\frac{1}{2(1-2\alpha)}+\half=1+n-\frac{\alpha}{1-2\alpha}=1-\gamma,$$
where $\gamma=f(\alpha)$. 
So $z=\psi(y)$ is also the first point (from left to right) for which $z \in[\half,1)$ but $T_\gamma(z) \in [0, \half)$.
The points $x \in [\alpha,1-\alpha)$ to the right of $y$ are characterized by the fact that $\psi(x) \geq z$. 
This is equivalent to the condition that 
$$\varsigma \circ \psi(x)=c \circ \varsigma(x) \in \cyl(\wh -+).$$ 
In other words, the larger value is taken if and only if $c(\omega) \in \cyl(\wh -+)$.

The proof in the case of $r=-1$ is similar. We have 
$n-\half < \frac{\alpha}{1-2\alpha} < n.$
This time the discontinuity occurs when $y=n(1-2\alpha)$. Then we have
$$\psi(y)=\frac{1}{2(1-2\alpha)}-n=\half+\frac{\alpha}{1-2\alpha}-n=\half-\gamma.$$
Let $z=\psi(y)$. Recall that $\psi$ is orientation reversing and sends the endpoints of $[\alpha, 1-\alpha)$ to $\half$. 
So, the points $x \in [\alpha, 1-\alpha)$ to the left of $y$ are characterized by the fact that $\psi(x) \in [\half-\gamma,\half)$. 
Equivalently, we have
$$\varsigma \circ \psi(x)=c \circ \varsigma(x) \in \cyl(\wh +-).$$ 

To see the equations for $r_-$, it suffices to use the orientation reversing involution $t \mapsto 1-t$, which conjugates
$T_\alpha$ to its inverse, nearly preserves $[\alpha, 1-\alpha)$, and switches the labeling of subintervals by $\pm 1$. 
\end{proof}

\subsection{Rectangle exchange transformations}
\name{sect:ret rot}
Fix $\alpha$ and $\beta$ in $[0,\half)$. Recall we defined an embedding 
$\pi:\til Y \times N \to X$ from the discussion of the arithmetic graph in equation \ref{eq:pi}. An alternate definition for this
embedding can be given using the coding maps $\varsigma$ and $\varsigma'$ of the rotations $T_\alpha$ and $T_\beta$, respectively. Namely, we have
\begin{equation}
\name{eq:pi2}
\pi:\til Y \times N \to X; \quad \pi(x,y,\v)=\big(\varsigma(x), \varsigma'(y), \v\big). 
\end{equation}
By Proposition \ref{prop:ag conj}, we have $\pi \circ \til \Psi_{\alpha, \beta}=\Phi \circ \pi,$ where
$\til \Psi_{\alpha, \beta}$ is the rectangle exchange map defined in equation \ref{eq:psi til} of the introduction.

So long as $\alpha$ and $\beta$ are irrational, the map $\pi$ is an embedding. We can use this embedding to pullback the renormalization of the map $\Phi$
defined in section \ref{sect:renormalization2} to a renormalization of these rectangle exchange maps. This yields the following theorem.

\begin{theorem}[Rectangle Exchange Renormalization] \name{thm:rectren} Assume $\alpha, \beta \in [0, \half)$. 
\begin{enumerate}
\item \name{item:once ren} Let $Z$ be the rectangle 
$Z=[\alpha,1-\alpha) \times [\beta, 1-\beta)$. The union of four rectangles $Z \times N$
is the preimage $\pi^{-1}(\sR_1)$ of the once renormalizable elements of $X$. 
\item \name{item:ret id2} The first return map $\wh \Psi$ of the rectangle exchange $\til \Psi_{\alpha, \beta}$ to $Z \times N$ satisfies 
$$\pi \circ \wh \Psi(x)=\wh \Phi \circ \pi(x) \quad \text{for all $x \in Z \times N$.}$$
\item \name{item:phi} The map $\phi:Z \times N \to \til Y \times N$ defined by $\phi=\psi_\alpha \times \psi_\beta \times \textit{id}$ as in 
equation \ref{eq:phi} satisfies the following statements:
\begin{enumerate}
\item $\phi \circ \wh \Psi(z)=\til \Psi_{f(\alpha),f(\beta)} \circ \phi(z)$ for all $z \in Z \times N$. 
\item Let $\pi'$ be the embedding $\til Y \times N \to X$ defined as in equation \ref{eq:pi2}, but
using the coding maps for $T_{f(\alpha)}$ and $T_{f(\beta)}$. Then,
$$\pi' \circ \phi(z)=\rho \circ \pi(z) \quad \text{for Lebesgue-almost every $z \in Z \times N$.}$$
\end{enumerate}
\end{enumerate}
\end{theorem}
These statements indicate that the first return map $\til \Psi$ of $\til \Psi_{\alpha, \beta}$ to the union $Z \times N$ of rectangles is 
affinely conjugate to $\til \Psi_{f(\alpha), f(\beta)}$. So, this describes a renormalization in the rectangle exchange sense. The theorem also indicates
compatibility with the renormalization of the map $\Phi:X \to X$. In fact, so long as $\alpha$ and $\beta$ are irrational, we have the alternate 
almost everywhere equivalent definitions,
$$\wh \Psi=\pi^{-1} \circ \wh \Phi \circ \pi 
\and
\phi=(\pi')^{-1} \circ \rho \circ \pi.$$

We also record the action on measures. The pushforward of Lebesgue measure under the embedding $\pi$
is the measure $\mu_\alpha \times \mu_\beta \times \mu_N$ on $X$. Here $\mu_\alpha$ and $\mu_\beta$ are defined as in the previous section and 
$\mu_N$ is the uniform measure on $N$. 
\begin{corollary}[Action of $\rho$ on Measures]
\name{cor:rho act}
Let $\nu=\mu_\alpha \times \mu_\beta \times \mu_N$ and $\nu'=\mu_{f(\alpha)} \times \mu_{f(\alpha)} \times \mu_N$. Then, 
$$\nu \circ \rho^{-1}=(1-2\alpha)(1-2\beta) \nu'.$$
\end{corollary}
The proof follows from the above renormalization theorem and Corollary \ref{cor:collapsing rotations}.

\section{The Return Time Cocycle}
\name{sect:cocycle}
In this section, we state our main formula for computing the total measure of the set
$$NS=\{x\in X~:~ \text{$x$ does not have a stable periodic orbit under $\Phi$}\}$$
with respect to the measures coming from rectangle exchange maps.

\subsection{The Cocycle Limit Formula} \name{sect:limit formula}
Assume $\alpha$ and $\beta$ are irrationals in $(0,\half)$. Set $\nu=\mu_\alpha \times \mu_\beta \times \mu_N$. 
Our formula is given using the following data:
\begin{enumerate}
\item \name{item:O} We find a nested sequence of Borel sets, 
$$X=\sO_0 \supset \sO_1 \supset \sO_2\ldots \quad \text{so that} \quad \NS=\bigcap_{i=0}^\infty \sO_n$$
up to a set of $\nu$-measure zero. Thus we have 
$\nu(\NS)=\lim_{n \to \infty} \nu(\sO_n).$
\item We now define the {\em return time cocycle} $N(\alpha,\beta,k):\R^4 \to \R^4$ over the dynamics of $f \times f$. (The transformation $f$ acting on the irrationals in $(0,\half)$ was defined in equation \ref{eq:f}.) Using $\alpha$ and $\beta$, we define 
$m,n \in \Z$ and $r,n \in \{\pm 1\}$ according to the formula:
$$f(\alpha)=r(\frac{\alpha}{1-2 \alpha}-m) \and f(\beta)=s(\frac{\beta}{1-2\beta}-n),$$
We define $N(\alpha, \beta,0)$ to be the identity matrix and define 
\begin{equation}\name{eq:N}
N(\alpha,\beta,1)=\left[\begin{array}{rrrr}
2m+r & 1 & 0 & 2m+r \\
2m & 1 & 0 & 2m \\
0 & 2 n+s & 2 n+s & 1 \\
0 & 2 n & 2 n & 1 \\
\end{array}\right].
\end{equation}
This matrix has determinant $rs \in \{\pm 1\}$.
We extend inductively by defining
$$N(\alpha, \beta,k+1)=N\big(f^k(\alpha),f^k(\beta), 1\big) N(\alpha, \beta,k) \quad \text{for $k \geq 1$.}$$
\item We define a one-dimensional cocycle $D$ over the dynamics of $f \times f$. This cocycle is defined by setting
$D(\alpha, \beta,0)=1$ and
\begin{equation}\name{eq:D}
D(\alpha, \beta,k)=\prod_{j=0}^{k-1} \big(1-2f^j(\alpha)\big)\big(1-2f^j(\beta)\big) \quad \text{for $k \geq 1$.}
\end{equation}
\item We define the vector 
$$\bn_{\alpha, \beta}=\left(\alpha(1-2\beta), \frac{1-2\alpha}{2}, \beta(1-2\alpha), \frac{1-2\beta}{2}\right).$$
\end{enumerate}
\begin{theorem}[Cocyle formula]\name{thm:cocycle formula}
Let $\alpha, \beta \in (0,\half)$ be irrational and let $k>0$. Define
$$\nu=\mu_\alpha \times \mu_\beta \times \mu_N, \quad d_k=D(\alpha,\beta,k) \and \bn_k=\bn_{f^k(\alpha),f^k(\beta)}.$$
Letting $\1 \in \R^4$ denote the vector all of whose entries are one, we have
$$\nu(\sO_{k+1})=d_k \bn_k \cdot N(\alpha,\beta,k) \1.$$
\end{theorem}

We have the following consequence by statement (1) above.

\begin{corollary}[Limit formula]
\name{cor:limit formula} 
For irrationals $\alpha, \beta \in (0,\half)$ we have
$$\nu(\NS)=\lim_{k \to \infty} d_k \bn_k \cdot N(\alpha,\beta,k) \1.$$
\end{corollary}

\subsection{The return time cocycle}
Our renormalization of $\Phi:X \to X$ described in section \ref{sect:renormalization2} is useful for measuring the prevalence of stable periodic trajectories
on $X=\Omega_\pm \times \Omega_\pm \times N$. To begin to understand this,
we recall some of the structure of the renormalization. We defined $\wh \Phi: \sR_1 \to \sR_1$ to be
the first return map to a Borel subset $\sR_1 \subset X$. We found a Borel measurable map
$\rho:\sR_1 \to X$ so that 
$$\rho \circ \wh \Phi(x)=\Phi \circ \rho(x) \quad \text{for each $x \in \sR_1$}.$$

We showed that the $\Phi$-orbit of an $x \in X$ always visits $\sR_1$ unless it belongs to the set $P_4$ of stable periodic orbits of period four, or if
it belongs to the set $\NUC$ of points $x=(\omega, \eta, \v)$ with $\omega$ or $\eta$ not unbounded collapsible. We view the case of $x \in \NUC$ as rare,
and justify this because $\NUC$ has zero measure with respect to many product measures
$\mu \times \mu' \times \mu_N$. (A criterion for this can be found in statement (\ref{item:omegaalt2}) of Theorem \ref{thm:collapsing}.) 
We make the following definition:
\begin{definition}
Let $\nu$ be a Borel measure on $X$. We say $\nu$ is {\em robustly renormalizable} if
for all integers $n \geq 0$ we have $\nu \circ \rho^{-n}(\NUC)=0$.
\end{definition}
\begin{remark}
So long as $\alpha$ and $\beta$ are irrational, the measures $\nu=\mu_\alpha \times \mu_\beta \times \mu_N$
are robustly renormalizable. Corollary \ref{cor:rho act} describes $\nu \circ \rho^{-n}$ in this case and 
statement (\ref{item:omegaalt2}) of Theorem \ref{thm:collapsing} implies $\nu \circ \rho^{-n}(\NUC)=0$.
\end{remark}

To understand iterations of $\rho$, for each $n \geq 1$ define the subsets
$$\sR_n=\rho^{-n}(X) \and \sO_n=\bigcup_{m \in \Z} \Phi^m(\sR_n).$$
We say that $x \in \sR_n$ is {\em $n$-times renormalizable.} The set $\sO_n$ is the smallest $\Phi$-invariant subset of $X$ containing $\sR_n$. When $x \in \sO_n$, we say that {\em the orbit of $x$ is $n$-times renormalizable}.

Recall that the renormalization $\rho$ has the property that $x \in \sR_1$ has a stable periodic orbit if and only if $\rho(x)$ has a stable periodic orbit,
and that $\rho(x)$ has a strictly smaller period. By the discussion above the definition, if $\nu$ is robustly renormalizable, then 
$$\rho^n(\sR_n \sm \sO_{n+1})=P_4, \qquad \text{$\nu \circ \rho^{-n}$-a.e..}$$
(If we can't apply $\rho$ once more at some point in the orbit of $x \in \rho^n(\sR_n \sm \sO_{n+1})$, it must be that either $x \in P_4$ or $x \in \NUC$.) 
In particular, almost every point in $\sO_n \sm \sO_{n+1}$ has a stable period orbit. Conversely, suppose $x$ has a stable periodic orbit of period larger than four. The fact that $\rho$ decreases periods guarantees that $x \in \sO_n \sm \sO_{n+1}$ for some $n$.

We can use the above argument to compute the measure of all points with a stable periodic orbit. The complement of this set is
$$NS=\{\textrm{$(\omega, \eta, \v)\in X$ without a stable periodic orbit}\}.$$

\begin{corollary}
\name{cor:limit}
If $\nu$ is robustly renormalizable, then
$$\nu(\NS)=\lim_{n \to \infty} \nu(\sO_n).$$
\end{corollary}
\begin{proof}
The above argument shows that the following holds $\nu$-a.e., taking $\sO_0=X$. 
$$X \sm \NS=\bigcup_{n=0}^\infty (\sO_n \sm \sO_{n+1}) \and \NS=\bigcap_{n=0}^\infty \sO_n.$$
This is a nested intersection, so the conclusion follows.
\end{proof}

Because of this Corollary, we wish to iteratively compute the measures of the sets  $\sO_n$. For this, we need some understanding of the return times to $\sR_n$. For integers $n>0$, we define
$$R_n:\sR_n \to \Z_+; \quad R_n(x)=\min \{m>0~:~\Phi^m(x) \in \sR_n\}.$$
The existence of this number is provided by statement (1) of the Theorem \ref{thm:renormalization}.
Observe that if $\nu$ is $\Phi$-invariant then we have 
\begin{equation}
\name{eq:ri}
\nu(\sO_n)=\int_{\sR_n} \ret{n}{x}~d\nu(x).
\end{equation}
This demonstrates the importance of knowing the return times.

Let $\nu$ be a $\Phi$-invariant measure on $X$. 
We interpret $\rho$ as a measure preserving map from the measure space
$(\sR_1, \sB, \nu|_{\sR_1})$ to the space $(X,\sB, \nu \circ \rho^{-1})$ with $\sB$ denoting the Borel $\sigma$-algebra. 
Recall that $\rho$ is a {\em measurable isomorphism $(\textit{mod}~0)$} if there are subsets $Z_1 \subset \sR_1$ with
$\nu(Z_1)=0$ and $Z_2 \subset X$ with $\nu \circ \rho^{-1}(Z_2)=0$ so that the restriction of $\rho$ to $\sR_1 \sm Z_1$
is a bijection onto $X \sm Z_2$ with measurable inverse. In this case, there is an inverse map
$$\rho_{\nu}^{-1}:X \sm Z_2 \to \sR_1 \sm Z_1.$$
We call this map the {\em measurable inverse} of $\rho$ with respect to $\nu$. 
We abuse notation by considering $\rho_{\nu}^{-1}$ to be a map from $X$ to $\sR_1$,
but note that it is defined only $\nu$-almost everywhere.

\begin{remark}
\name{rem:invertibility}
So long as $\alpha$ and $\beta$ are irrational, $\rho$ has a measurable inverse with respect to
$\nu=\mu_\alpha \times \mu_\beta \times \mu_N$. This is because the coding map $\pi:\til Y \times N \to X$ 
given in equation \ref{eq:pi2} is a measurable isomorphism from $\til Y \times N$ equipped with Lebesgue measure to $X$ equipped with the measure 
$\nu$. This follows from the facts that $\pi$ is injective and $\nu$ is the pushforward of Lebesgue measure under $\pi$. 
Utilizing statement (\ref{item:phi}b) of Theorem \ref{thm:rectren}, we can explicitly describe the measurable inverse as
$$\rho_{\nu}^{-1}=\pi\circ \phi^{-1} \circ (\pi')^{-1}$$
\end{remark}

\begin{remark}
Measures for which $\rho$ is not measurably invertible can be analyzed as below utilizing conditional expectations.
See \cite{Htruchet1}.
\end{remark}

We now generalize the return time definition to a linear operator on the space of all Borel measurable functions on $X$. Suppose $f$ is a Borel measurable function on $X$. We define the {\em retraction of $f$ to $\sR_1$} to be the function 
$r_f:\sR_1 \to \R$ given by 
$$r_f(x)=\sum_{i=0}^{\ret{1}{x}-1} f \circ \Phi^i(x).$$
We think of this as a generalization of the return time, since for the constant function $\bbone$ we have $\ret{1}{x}=r_\bbone(x)$.

Now assume that $\rho_\nu^{-1}$ is a measurable inverse of $\rho$ with respect to $\nu$ as above.
Then for any $\Phi$-invariant set $A \subset \sO_1$ and any $\nu$-integrable $f:X \to \R$, we have 
$$\int_{A} f~d \nu=\int_{A \cap \sR_1} r_f(x)~d\nu=\int_{\rho(A \cap \sR_1)} r_f \circ \rho^{-1}_\nu(y)~d(\nu \circ \rho^{-1})(y).$$
This motivates the definition of a linear operator on functions $X \to \R$:
\begin{equation}
\name{eq:cocycle1}
C(\nu, 1): L^1(\nu) \to L^1(\nu \circ \rho^{-1}); \quad f \mapsto r_f \circ \rho^{-1}_\nu.
\end{equation}
From the above remarks, it has the property that 
\begin{equation}
\name{eq:cocycle i}
\int_{A} f~d\nu=\int_{\rho(A \cap \sR_1)} C(\nu,1)(f)~d(\nu \circ \rho^{-1}).
\end{equation}
We would like to apply this operation repeatedly, so we make the following definition.

\begin{definition}
Let $\nu$ be a robustly renormalizable measure, and define $\nu_n=\nu \circ \rho^{-n}$ for integers $n \geq 0$. 
We say $\nu$ is {\em robustly invertible} if for each $n \geq 1$, the renormalization $\rho$ 
thought of as a measurable map from $(X,\sB,\nu_{n-1})$ to $(X,\sB,\nu_{n})$ has a measurable inverse 
$\rho^{-1}_n:X \to \sR_1$. This means for $\nu_{n-1}$-a.e. $x \in X$ and $\nu_{n}$-a.e. $y \in X$ we have 
$$\rho_n^{-1} \circ \rho(x)=x \and \rho \circ \rho_n^{-1}(y)=y.$$
\end{definition}

Suppose that $\nu$ is robustly invertible, and define $\nu_n=\nu \circ \rho^{-n}$ and $\rho^{-1}_n$ as in the definition above
so that $\nu_0=\nu$. 
Observe that we can compose the operators $C(\nu_n,1)$ constructed as in equation \ref{eq:cocycle1}.
Each operator $C(\nu_n,1)$ sends $L^1(\nu_n)$ to $L^1(\nu_{n+1})$, so for integers $n \geq 0$ and $m \geq 1$ define
$$C(\nu_n,m):L^1(\nu_n) \to L^1(\nu_{n+m}); \quad C(\nu_n,m)= C(\nu_{n+m-1},1)\circ \ldots \circ C(\nu_{n+1},1) \circ C(\nu_n,1).$$
Taking $C(\nu,0)$ to be the identity operator on $L^1(\nu)$, these operators form cocycle over the renormalization dynamics of $\rho$ acting on the space of robustly invertible $\Phi$-invariant measures. That is they satisfy the identity
$$C(\nu,m+k)=C(\nu \circ \rho^{-m},k) \circ C(\nu,m) \quad \text{for all $m,k \geq 0$}.$$
We prove that this cocycle satisfies a generalization of equation \ref{eq:cocycle i}.
\begin{lemma}[Integral Formula]
\name{lem:cocycle}
Suppose $\nu$ is robustly invertible. Then for all integers $n \geq 1$, 
all Borel measurable $\Phi$-invariant sets $A \subset \sO_n$, and all $\nu$-integrable $g:A \to \R$,
we have
$$\int_{A} f~d\nu=\int_{\rho^n(A\cap \sR_n)} C(\nu, n)(f)(x)~d\nu \circ \rho^{-n}(x).$$
\end{lemma}
\begin{proof}
This formula follows by inductively applying equation \ref{eq:cocycle i}. We demonstrate how it works for the first iteration.
Let $B=\rho(A \cap \sR_1)$. The set $B$ is $\Phi$-invariant by statement (2) of Theorem \ref{thm:renormalization}. In addition, $B \subset \sO_{n-1}$. We let $g=C(\nu,1)(f)$. Then by equation 
\ref{eq:cocycle i}, we have
$$\int_{A} f~d\nu=\int_B g~d\nu \circ \rho^{-1}.$$
Assuming $n-1 \geq 1$, we can apply equation \ref{eq:cocycle i} again.
\end{proof}

\subsection{Step functions}
We are interested in the behavior of cocycle $C(\nu,n)$, when $\nu$ is taken from the space of measures coming from our rectangle exchange maps.
This space of measures is $\rho$ invariant up to scaling. The scaling constant is given by
the one-dimensional cocycle $D$ defined in equation \ref{eq:D}.

\begin{proposition}\name{prop:measure scaling}
Let $\nu=\mu_\alpha \times \mu_\beta \times \mu_N$. Then, for all $k \geq 0$ we have 
$$\nu \circ \rho^{-k}=D(\alpha, \beta, k)~\mu_{f^k(\alpha)} \times \mu_{f^k(\beta)} \times \mu_N.$$
\end{proposition}
\begin{proof}
This follows from an inductive application of Corollary \ref{cor:rho act}.
\end{proof}

We will see that $C(\nu,k)$ preserves a finite dimensional subspace of step functions containing the constant function $\bbone$, so long as $\nu$ has the form above. This reduces the equation for integrating such a step function over $\sO_n$ given in Lemma \ref{lem:cocycle}
to working with a finite dimensional cocycle. In this subsection, we find a $6$-dimensional invariant subspace. In the following subsection,
we observe that $\bbone$ belongs to a four dimensional invariant subspace. This allows us to drop the dimension of the cocycle to four.

We partition the space $X$ into six non-empty pieces $\Step_1, \Step_2, \ldots, \Step_6$ and define the linear embedding into the space of of Borel measurable functions on $X$,
\begin{equation}
\name{eq:epsilon}
\epsilon:\R^6 \to \sM(X); \quad \text{$\epsilon(\bp)(x)=\bp_i$ if $x \in \Step_i$.}
\end{equation}
We say $x \in \Step_i$ has {\em step class $i$}. These sets have combinatorial definitions given below.

First we define two sets. The {\em set of directions} consists of the terms {\em horizontal} and {\em vertical}. We define the {\em set of sign pairs} to be $\{--,-+,+-,++\}$. This is the set of words of length $2$ in the alphabet $\{\pm 1\}$. To each element $x=(\omega, \eta, \v) \in X$ with $\v=(a,b)$, we assign a unique direction and 
sign pair. Recall the definition of $\Phi$, 
$$\Phi(x)=\big(\sigma^{sb}(\omega), \sigma^{sb}(\eta),(sb,sa)\big) \quad \text{with $s=\omega_0 \eta_0$}.$$
This assignment of direction and sign pair to $x$ is given by the following chart.

\begin{center}
\begin{tabular}{|c|c|c|}
\hline 
Value of $(sb,sa)$ & Direction & Sign pair \\ 
\hline 
$(1,0)$ & horizontal & $\omega_0 \omega_1$ \\ 
\hline 
$(-1,0)$ & horizontal & $\omega_{-1} \omega_0$ \\ 
\hline 
$(0,1)$ & vertical & $\eta_0 \eta_1$ \\ 
\hline 
$(0,-1)$ & vertical & $\eta_{-1} \eta_0$ \\ 
\hline 
\end{tabular} 
\end{center}
If $x \in X$ has horizontal direction and sign pair $-+$, we call $x$ a {\em $-+$-horizontal step}.
We use similar language to describe all combinations of directions with sign pairs.

We use these terms to define the six step classes. Each $x \in X$ belongs to exactly one class. 
\begin{citemize}
\item We say $x$ has {\em step class $1$} if $x$ is a $(-+)$-horizontal step.
\item We say $x$ has {\em step class $2$} if $x$ is a $(+-)$-horizontal step.
\item We say $x$ has {\em step class $3$} if $x$ is a $(++)$- or $(--)$-horizontal step. 
\item We say $x$ has {\em step class $4$} if $x$ is a $(-+)$-vertical step.
\item We say $x$ has {\em step class $5$} if $x$ is a $(+-)$-vertical step.
\item We say $x$ has {\em step class $6$} if $x$ is a $(++)$- or $(--)$-vertical step.
\end{citemize} 
This defines a partition of $X$ into the six sets $\Step_1, \ldots, \Step_6 \subset X$, and defines the function $\epsilon$ as in equation \ref{eq:epsilon}. 

We will need to integrate a step function $\epsilon(\bp)$ over $X$ with respect to the measure
$$\nu=\mu_\alpha \times \mu_\beta \times \mu_N.$$ To do this, we define the vector 
$\m_{\alpha,\beta} \in \R^6$ according to the rule
$$\m_{\alpha,\beta}=\big(\nu(\Step_1), \ldots, \nu(\Step_6)\big).$$
This choice guarantees that we have the formula 
\begin{equation}
\name{eq:meas vect}
\int_X \epsilon(\bp)~d\mu=\m_{\alpha,\beta} \cdot \bp.
\end{equation}
We have the following explicit formula for $\m_{\alpha,\beta}$:

\begin{proposition}\name{prop:meas vect}
We have $\m_{\alpha, \beta}=\thalf(\alpha, \alpha, 1-2\alpha, \beta, \beta, 1-2\beta).$
\end{proposition}
\begin{proof}
Let $x=(\omega,\eta,\v)$ be taken at random from $X$ according to the measure $\nu$. 
Let $\v=(a,b)$ and $s=\omega_0 \eta_0$ so that the directional component of $\Phi(x)$ is $\bw=(sb,sa)$.
The probability that $\bw=(1,0)$ is $1/4$. Given this, $x$ is a $(-+)$ step if $\omega \in \cyl(\wh - +)$.
The $\mu_\alpha$ measure of $\cyl(\wh - +)$ is $\alpha$. Similarly, we see that the probability of $x \in \Step_i$
given that $\v=(1,0)$ is given by the $i$-th entry of the vector
$$(\alpha, \alpha, 1-2\alpha,0,0,0).$$
The same holds vector holds for the case $\v=(-1,0)$. Given that $\v=(0,1)$ or $\v=(0,-1)$, the probability of $x \in \Step_i$
is given by the entries of
$$(0,0,0,\beta,\beta,1-2\beta).$$
We get $\m_{\alpha, \beta}$ by averaging the two vectors above.
\end{proof}

For the following theorem, we define $\chi_i$ to be the characteristic function of $\Step_i$, and define $\be_i \in \R^6$ to be the 
standard basis vector with $1$ in position $i$. 

\begin{theorem}[Collapsed Steps]
\name{thm:returns2}
Suppose $x \in \sR_1$ has return time $\ret{1}{x}=4k+1$ and $\rho(x) \in \Step_j$. Then, for all $i \in \{1, \ldots, 6\}$ we have 
$r_1(\chi_i;x)=\be_j \cdot K \be_i$ with
$$K=\left[\begin{array}{rrrrrr}
k & k-1 & 2 & 0 & 0 & 2k \\
k & k+1 & 0 & 0 & 0 & 2k \\
k & k & 1 & 0 & 0 & 2k \\
0 & 0 & 2k & k & k-1 & 2 \\
0 & 0 & 2k & k & k+1 & 0 \\
0 & 0 & 2k & k & k & 1 \\
\end{array}\right].$$
\end{theorem}

Observe that the $j$-th row of $K$ gives the number of each step type which appears in the set $\{x, \Phi(x), \ldots, \Phi^{4k}\}$
provided $\rho(x) \in \Step_j$ and $\ret{1}{x}=4k$. We prove this theorem at the end of this subsection.

The utility of the Lemma is the following. So long as the return time function is $\nu$-a.e. constant on each $\rho^{-1}(\Step_i)$, the cocycle
$C(\nu,1)$ will preserve the subspace $\epsilon(\R^6)$. Indeed, if $\ret{1}{x}=4k+1$ for $\nu$-a.e. $x \in \rho^{-1}(\Step_j)$, then 
\begin{equation}
\name{eq:cocycle obs}
C(\nu,1)\big(\epsilon(\be_i)\big)(y)=\epsilon(K \be_i)(y) \quad \text{for $\nu \circ \rho^{-1}$-a.e. $y \in \Step_j$,}
\end{equation}
with $K$ as in the above Theorem. 
In the case of measures of the form $\nu=\mu_\alpha \times \mu_\beta \times \mu_N$ the condition of being almost everywhere constant on $\rho^{-1}(\Step_j)$ is guaranteed (indirectly) by Corollary \ref{cor:rot ret}. In this case, we can extend linearly to understand the action of $C(\nu,1)$ on the subspace $L$. Details are in the proof of the following Lemma.

\begin{lemma}[Finite Dimensional Cocycle]
\name{lem:matrices rectangle case}
Let $\nu=\mu_\alpha \times \mu_\beta \times \mu_N$ with $\alpha, \beta \in [0,\half)$. The operator
$C(\nu,1)$ preserves the space of step functions $\epsilon(\R^6)$. Determine $r,s \in \{\pm 1\}$ and $m,n \in \Z$ according to the rule
$$f(\alpha)=r(\frac{\alpha}{1-2 \alpha}-m) \and f(\beta)=s(\frac{\beta}{1-2\beta}-n).$$
Then, the action of $C(\nu,1)$ on $\epsilon(\R^6)$ satisfies
$$C(\nu,1)\big(\epsilon(\bp)\big)=\epsilon (M \bp) \qquad \text{$\nu \circ \rho^{-1}$-a.e.,}$$ 
with the matrix $M=M(\alpha,\beta,1)$ given by 
$$M=\left[\begin{array}{rrrrrr}
m+\frac{r+1}{2} & m+\frac{r-1}{2} & 2 & 0 & 0 & 2m+1+r \\
m+\frac{r-1}{2} & m+\frac{r+1}{2} & 0 & 0 & 0 & 2m-1+r \\
m & m & 1 & 0 & 0 & 2m \\
0 & 0 & 2 n+1+s & n+\frac{s+1}{2} & n+\frac{s-1}{2} & 2 \\
0 & 0 & 2n-1+s & n+\frac{s-1}{2} & n+\frac{s+1}{2} & 0 \\
0 & 0 & 2n & n & n & 1 \\
\end{array}\right]. 
$$
\end{lemma}
The matrix $M$ above has entries which are all non-negative integers, and has determinant $rs \in \{\pm 1\}$. 

We extend the definition of $M(\alpha, \beta, 1)$ to a cocycle. We define $M(\alpha, \beta,0)$ to be the identity matrix. We inductively define
$$M(\alpha, \beta,k+1)=M\big(f^k(\alpha),f^k(\beta),1\big) M(\alpha, \beta, k) \quad \text{for $k \geq 1$.}$$
We have the following.

\begin{corollary}\name{cor:cocycle integral 1}
Let $\alpha, \beta \in (0,\half)$ be irrational and set $\nu=\mu_\alpha \times \mu_\beta \times \mu_N$. Let $\bp \in \R^6$ and set $g=\epsilon(\bp)$. 
For $k \geq 0$, define 
$$\m_k=\m_{f^k(\alpha),f^k(\beta)}\in \R^6, \quad d_k=D(\alpha, \beta,k) \in \R, \and M_k=M(\alpha,\beta,k).$$ Then,
$$\int_{\sO_k} g~d\nu=d_k (\m_k \cdot M_k \bp).$$
\end{corollary}
\begin{proof}
It follows by inductively applying Lemma \ref{lem:matrices rectangle case} that 
$$C(\nu,k)(g)=\epsilon(M_k \bp) \quad \text{$\nu \circ \rho^{-k}$-a.e..}$$
And therefore by Proposition \ref{prop:measure scaling} and equation \ref{eq:meas vect}, we have
$$\int_X C(\nu,k)(g)(x)~d\nu \circ \rho^{-k}(x)=d_k (\m_k \cdot M_k \bp).$$
Then the conclusion follows from Lemma \ref{lem:cocycle} with $A=\sO_k$ so that $\rho^k(A \cap \sR_k)=X$.
\end{proof}

The remainder of this section is devoted to proofs of Theorem \ref{thm:returns2} and Lemma \ref{lem:matrices rectangle case}.
\begin{proof}[Proof of Lemma \ref{lem:matrices rectangle case} assuming Theorem \ref{thm:returns2}]
Fix $\alpha$ and $\beta$ as in the Lemma. This determines the constants $m$, $n$, $r$ and $s$ as well as the matrix $M$. 
By linearity, it is sufficient to prove that for each $i, j \in \{1, \ldots, 6\}$, we have
$$
C(\nu,1)\big(\epsilon(\be_i)\big)(y)=\epsilon(M \be_i)(y) \quad \text{for $\nu \circ \rho^{-1}$-a.e. $y \in \Step_j$,}
$$
By definition of $\epsilon$, for all $y \in \Step_j$ we have
$$\epsilon(M \be_i)(y)=\be_j \cdot (M \be_i).$$
By equation \ref{eq:cocycle obs}, to prove the Lemma, it is sufficient to check the following statements:
\begin{enumerate}
\item There is a constant $k$ so that $\ret{1}{x}=4k+1$ for $\nu$-a.e. $x \in \rho^{-1}(\Step_j)$.
\item Defining $K$ using $k$ as in the Theorem, we have $\be_j \cdot (M \be_i)=\be_j \cdot (K \be_i)$.
\end{enumerate}

We will carry this argument out for one $j$, and leave the remaining cases to the reader. Suppose $j=1$. Then we are interested in the case
when $x \in \rho^{-1}(\Step_1)$. This means that $y=\rho(x)$ is a $(-+)$-horizontal step. Let $x=(\omega, \eta, \v)$.
Since $y=\big(c(\omega), c(\eta), \v\big)$ is a horizontal step, we know $\v=(0,b)$ for $b \in \{\pm 1\}$. Let $s=\omega_0 \eta_0$ and 
$\bw=(sb,0)$. By Theorem \ref{thm:return time}, we know that the return time of $x$ to $\sR_1$ is given by 
$$R_1(x)=2 E(0,0, \bw)-1,$$ 
with this quantity $e=E(0,0,\bw)$ indicating one more than the number of columns removed to the right by the operation $\rho$ 
if $sb=1$ and one more than the number of columns to the left removed if $sb=-1$. See above Theorem \ref{thm:return time}. 
We break into cases depending on $r$ and $sb$. If $sb=1$, since $y$ is a $(-+)$-horizontal step, we have $y \in \cyl(\wh - +)$ and therefore
$$E(0,0,\bw)=r_+(\omega)=\begin{cases}
2m+3 & \text{if $r=1$}\\
2m+1 & \text{if $r=-1$}
\end{cases}$$
by Corollary \ref{cor:rot ret}. Similarly, if $sb=-1$, we have $y \in \cyl(- \wh +)$ and 
$$E(0,0,\bw)=r_-(\omega)=\begin{cases}
2m+3 & \text{if $r=1$}\\
2m+1 & \text{if $r=-1$}
\end{cases}$$
Either way, the following choice of $k$ satisfies statement (1) above:
$$k=m+\frac{r+1}{2}$$
Statement (2) claims $\be_1 \cdot (M \be_i)=\be_1 \cdot (K \be_i)$ for all $i$. This is just an observation.
\end{proof}

\begin{proof}[Proof of Theorem \ref{thm:returns2}]
We will prove the theorem in the case that $\rho(x)$ is a horizontal step, so $j=1,2,3$. The vertical case follows from symmetry.

Suppose $\rho(x)$ is a $(r s)$-horizontal step with $r,s \in \{\pm 1\}$. 
Let $x=(\omega, \eta, \v)$. 
First observe that if the return time $R_1(x)=1$, then we know $(rs) \neq (-+)$. Otherwise, $\omega$ would fail to be zero-collapsible. 
In addition, if $R_1(x)=1$ then $\rho(x)$ is also an $(rs)$-horizontal step. 
When $R_1(x)=1$, we have $k=0$. In this case we can check that for $j=2,3$ and $i = 1, \ldots, 6$ we have
$\be_j \cdot K \be_i$ equals one if $i=j$ and zero otherwise. (This is the observation that the second and third rows of $K$ are the corresponding rows of the identity matrix when $k=0$.)

Now suppose that $R_1(x)=4k+1$ and $k \geq 1$. Then, $\Phi(x) \not \in \sR_1$. Recall the definition of horizontal box given in section \ref{sect:ren proofs}. 
The corresponding curve in the tiling associated to $(\omega, \eta)$ enters a maximal horizontal box $B$. 
Observe that the length parameter of $\ell$ can be computed using Theorem \ref{thm:return time}. (It is half of $E(0,0,\bw)-1$ if $\bw$ is the directional component of $\Phi(x)$.) Therefore, $\ell=k$. So, the curve of the tiling
follows the central curve of the horizontal box $B$. When it leaves the horizontal box it returns to $\sR_1$  
Possible pictures of this curve are shown below in the case that $\rho(x)$ is a $(-+)$-horizontal step, and the length parameter of the box is $\ell=2$:
\begin{center}
\includegraphics[scale=0.5]{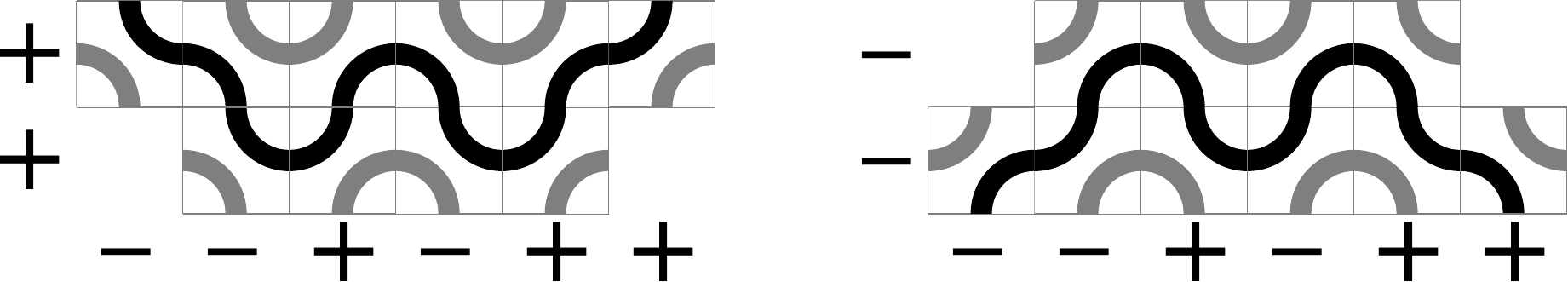}
\end{center}
The sequence of step classes associated to $\Phi^i(x)$ for $i=0, \ldots 4k$ can be determined by examining the adjacent pairs of tiles passed through
by the central curve extended into the two neighboring squares. 
The cases of $x$ and $\Phi^{4k}(x)$ correspond to the pairs of tiles at the two ends. The left one is an $(r-)$-horizontal step
and the right end is a $(+s)$-horizontal step. Along the central curve of the horizontal box, we pass through $\ell=2k$ $(++)$- and $(--)$-vertical steps,
$k$ $(-+)$-horizontal steps, and $k-1$ $(+-)$-horizontal steps. The total count of each type of step $\Step_i$ gives the values of $e_j \cdot K \be_i$
for $j\in \{1,2,3\}$ and $i\in \{1, \ldots, 6\}$ as desired. 
\end{proof}

\subsection{Simplifications}
\name{sect:simp}
In this subsection, we perform some minor optimizations to the formula given in Corollary \ref{cor:cocycle integral 1}. 

The first observation is that we can integrate $\epsilon(\bp)$ over $\sO_1$ without applying the cocycle. We define a new vector
$\bq_{\alpha, \beta}$ to be the vector whose $i$-th entry is $\nu(\sO_1 \cap \Step_i)$. Then by definition we have
\begin{equation}\name{eq:q int}
\int_{\sO_1} \epsilon(\bp)~d\nu=\bq_{\alpha,\beta} \cdot \bp.
\end{equation}

\begin{proposition}
\name{prop:q} For $\alpha, \beta \in [0,\half)$ and $\nu=\mu_\alpha \times \mu_\beta \times \mu_N$, we have
$$\bq_{\alpha,\beta}=\half \Big(\alpha(1-2\beta), \alpha(1-2\beta), 1-2\alpha, \beta(1-2\alpha), \beta(1-2\alpha), 1-2\beta\Big).$$
\end{proposition}
\begin{proof}
Corollary \ref{cor:cocycle integral 1} gives an alternate version of the integral in equation
\ref{eq:q int}. Namely,
$$\int_{\sO_1} \epsilon(\bp)~d\nu=(1-2\alpha)(1-2\beta) \m_{f(\alpha),f(\beta)} \cdot M(\alpha,\beta,1) \bp.$$
Let $M=M(\alpha,\beta,1)$. We must have 
$$\bq_{\alpha,\beta}=(1-2\alpha)(1-2\beta) M^T \m_{f(\alpha),f(\beta)}.$$
Thus, we have reduced the problem to a calculation. The matrix $M$ is defined as in Lemma \ref{lem:matrices rectangle case}
using the constants $m, n\in \Z$ and $r,s \in \{\pm 1\}$ which satisfy
$$f(\alpha)=r(\frac{\alpha}{1-2\alpha}-m) \and f(\alpha)=s(\frac{\alpha}{1-2\alpha}-n).$$
Using these equations, we can show by direct computation that 
$$M^T \m_{f(\alpha),f(\beta)}=\half\left(\frac{\alpha}{1-2\alpha},\frac{\alpha}{1-2\alpha},\frac{1}{1-2\beta},\frac{\beta}{1-2\beta},
\frac{\beta}{1-2\beta},\frac{1}{1-2\alpha}\right).$$
The conclusion follows by multiplying through by $(1-2\alpha)(1-2\beta)$.
\end{proof}

It follows that we have the following slightly simpler formula:
\begin{corollary}\name{cor:cocycle integral 2}
Let $g=\epsilon(\bp)$. 
For $k \geq 0$, set 
$$\bq_k=\bq_{f^k(\alpha),f^k(\beta)}\in \R^6, \quad d_k=D(\alpha, \beta,k) \in \R, \and M_k=M(\alpha,\beta,k).$$ Then,
$$\int_{\sO_{k+1}} g~d\nu=d_k (\bq_k \cdot M_k \bp).$$
\end{corollary}
\begin{proof}
This proof mirrors the proof of Corollary \ref{cor:cocycle integral 1}.
By Lemma \ref{lem:matrices rectangle case},
$$C(\nu,k)(g)=\epsilon(M_k \bp) \quad \text{$\nu \circ \rho^{-k}$-a.e..}$$
By Proposition \ref{prop:measure scaling} and equation \ref{eq:q int},
$$\int_{\sO_1} C(\nu,k)(g)(x)~d\nu \circ \rho^{-k}(x)=d_k (\bq_k \cdot M_k \bp).$$
We apply Lemma \ref{lem:cocycle} to the case of $A=\sO_{k+1}$ so that $\sO_1=\rho^k(A \cap \sR_k).$
\end{proof}

For our final trick, we reduce the dimension of the cocycle to four. We observe that the right multiplication by the cocycle $M(\alpha, \beta,n)$ leaves invariant a four-dimensional subspace. To explain this,
we introduce the following linear projection $\bpi:\R^6 \to \R^4$ and section $\s:\R^4 \to \R^6$ satisfying $\bpi \circ \s=\id$.
\begin{equation}
\name{eq:pi1}
\bpi(a,b,c,d,e,f)=(a+b,c,d+e,f) \and \s(a,c,d,f)=(\frac{a}{2}, \frac{a}{2}, c, \frac{d}{2},\frac{d}{2},f).
\end{equation}

\begin{proposition}
Multiplication by $M^T=M(\alpha,\beta,k)^T$ leaves invariant the subspace $\s(\R^4)$. Moreover, for all $\v \in \R^4$ we have
$$M(\alpha,\beta,k)^T \circ \s(\v)=\s \circ N(\alpha,\beta,k)^T(\v),$$
where $N$ is the cocycle defined in equation \ref{eq:N}.
\end{proposition}
The proof is just a calculation to verify the equation in the Proposition in the case $n=1$. The general case follows from the cocycle identity.

We define the projection of the vector $\bq_{\alpha,\beta}$ defined in Proposition \ref{prop:q} to be the row vector
\begin{equation}
\name{eq:n}
\bn_{\alpha,\beta}=\bpi(\bq_{\alpha,\beta})=\left(\alpha(1-2\beta), \frac{1-2\alpha}{2}, \beta(1-2\alpha), \frac{1-2\beta}{2}\right).
\end{equation}
Also note that $\bq_{\alpha,\beta}=\s(\bn_{\alpha,\beta})$.

We apply Corollary \ref{cor:cocycle integral 2} to the special case when $\bp=\1$. 
To ease notation, for $k \geq 0$ make the following definitions:
$$M_k=M(\alpha, \beta,k). \qquad N_k=N(\alpha, \beta,k).$$
$$\bq_k=\bq_{f^k(\alpha),f^k(\beta)}. \qquad \bn_k=\bn_{f^k(\alpha),f^k(\beta)}.$$
We use $\1_6 \in \R^6$ and $\1_4 \in \R^4$ to denote vectors all of whose entries are one. 
We have
$$\nu(\sO_{k+1})=d_k \bq_k \cdot M_k \1_6=d_k \pi(M_k^T \bq_k)\cdot \1_4=d_k (N_k^T \bn_k)\cdot \1_4
.$$
This is our Cocycle Formula, proving Theorem \ref{thm:cocycle formula}.

\section{Cocycle Calculations}
\name{sect:calc}

\subsection{The recurrent case}
\name{sect:recurrent case}
The goal of this subsection is to prove the following theorem concerning measures of the set $\NS \subset X$ of points without a stable periodic $\Phi$-orbit. Recall that our renormalization action on parameters for rectangle exchange maps was closely related to the map $f:[0,\half) \to [0,\half]$ defined in equation \ref{eq:f}. 

\begin{theorem}[Recurrent Case]
\name{thm:recurrent case}
Let $\alpha, \beta \in (0,\half)$ be irrational and let $\nu=\mu_\alpha \times \mu_{\beta} \times \mu_N$. 
If the sequence of points 
$$\big\{\big(f^k(\alpha),f^k(\beta)\big)\big\}$$ has an accumulation point 
$(x,y)$ with $x \geq 0$ and $y \geq 0$, then $\nu(\NS)=0$. 
\end{theorem}

We make use of statement (\ref{item:O}) of section \ref{sect:limit formula} which provides the limit formula 
$$\nu(\NS)=\lim_{k \to \infty} \nu(\sO_k).$$ Recall that the sequence of sets $\sO_k$ were nested,
so this sequence is decreasing. Therefore to show $\nu(\NS)=0$, it is sufficient to 
show that there is an $\epsilon>0$ so that for infinitely many $k$ we have $\nu(\sO_{k+1})<(1-\epsilon) \nu(\sO_{k})$. 
Thus the theorem is implied by the following lemma.

\begin{lemma}[Scaling Lemma]
\name{lem:scaling}
For all $\alpha, \beta \in (0,\half)$ and all $k \geq 1$, we have 
$$\nu(\sO_{k+1}) \leq g\big(f^{k+1}(\alpha), f^{k+1}(\beta)\big) \nu(\sO_{k})$$
where $g(x,y)=1-\frac{4}{3} xy$.
\end{lemma}

\begin{remark}[Slow divergence]
Observe that the scaling lemma actually implies that if the $f \times f$ orbit of $(\alpha, \beta)$ satisfies
$$\prod_{k=1}^\infty g\big(f^k(\alpha),f^k(\beta)\big)=0,$$
then almost every orbit is periodic, $\nu(\NS)=0$.
\end{remark}

For two vectors $\v$ and $\bw$ in $\R^n$, we say $\v \leq \bw$ {\em entrywise} if $\v_i \leq \bw_i$ for $i=1, \ldots n$. 
We will see that it is sufficient to show the following:

\begin{lemma}[Scaling Lemma II]
\name{lem:scaling2}
For all $\gamma, \delta \in (0,\half)$, we have the entrywise inequality
$$D(\gamma,\delta,1) N(\gamma,\delta,1)^T \bn_{f(\gamma),f(\delta)} \leq g\big(f(\gamma),f(\delta)\big) \bn_{\gamma, \delta},$$
where $g(x,y)$ is as defined in the previous lemma.
\end{lemma}
\begin{proof}[Proof of Lemma \ref{lem:scaling} given Lemma \ref{lem:scaling2}]
We utilize the limit formula in Theorem \ref{thm:cocycle formula} to relate $\nu(\sO_k)$ to 
$\nu(\sO_{k+1})$. Define
$$\v=D(\alpha,\beta,k-1) N(\alpha,\beta,k-1) \1.$$
Set $\gamma=f^{k-1}(\alpha)$ and $\delta=f^{k-1}(\beta)$. Then by Theorem \ref{thm:cocycle formula}, we have
$$\nu(\sO_k)=\bn_{\gamma,\delta} \cdot \v \and \nu(\sO_{k+1})=\Big(D(\gamma,\delta,1) N(\gamma,\delta,1)^T \bn_{f(\gamma),f(\delta)}\big) \cdot \v.$$
So the entrywise inequality implies
$$\nu(\sO_{k+1}) \leq g\big(f(\gamma), f(\delta)\big) \nu(\sO_{k}).$$
\end{proof}

\begin{proof}[Proof of Lemma \ref{lem:scaling2}]
We use notation similar to that of section \ref{sect:limit formula} for 
defining $m$, $n$, $r$ and $s$. We define these constants so that 
$$r f(\gamma)+m=\frac{\gamma}{1-2\gamma} \and s f(\delta)+n=\frac{\delta}{1-2\delta}.$$
We break the vector $\bn_{f(\gamma),f(\delta)}$ into two pieces, writing
$\bn_{f(\gamma),f(\delta)}=\ba-\bb$ with 
$$\ba=\big(f(\gamma),\frac{1-2f(\gamma)}{2},f(\delta),\frac{1-2f(\delta)}{2}\big) \and
\bb=2f(\gamma)f(\delta)\big(1,0,1,0\big).$$
Define $d=D(\gamma,\delta,1)=(1-2\gamma)(1-2\delta)$ and $N=N(\gamma,\delta,1)$. The matrix $N$ is given exactly as in equation \ref{eq:N}. We have
$$\bn_{\gamma,\delta}= d N^T \ba.$$
This is not an accident; it comes from the fact that $\ba=\pi(\m_{\gamma,\delta})$ and the meaning of these quantities (which were defined in section \ref{sect:cocycle}). But the statement can also be verified by calculation. For instance, the first entry of $d N^T \ba$ is given by 
$$\begin{narrowarray}{3pt}{rcl}
(d N^T \ba)_1 & = & (1-2\gamma)(1-2\delta)\Big[(2m+r)f(\gamma)+(2m)(\frac{1-2f(\gamma)}{2})\Big] \\
& = & (1-2\gamma)(1-2\delta)[r f(\gamma)+m]=(1-2\delta)\gamma=(\bn_{\gamma,\delta})_1.
\end{narrowarray}$$

We have shown that
$$dN^T \bn_{f(\gamma),f(\delta)}=\bn_{\gamma,\delta}-dN^T \bb.$$
To simplify expressions below let $\bz=dN^T \bb$ and $\bn=\bn_{\gamma,\delta}.$ 
We will show that 
$$\bz_i/\bn_i \geq \frac{4}{3} f(\gamma) f(\delta) \quad \text{for $i \in \{1,2,3,4\}$.}$$
This will conclude the proof. We have
$$\bn=\left(\gamma(1-2\delta),\frac{1-2\gamma}{2},\delta(1-\gamma),\frac{1-2\delta}{2}\right).$$
$$\bz=(1-2\gamma)(1-2\delta) 2f(\gamma)f(\delta)\big(2m+r,2n+s+1,2n+s,2m+r+1\big).$$
To prove the theorem, we will provide lower bounds for the quantities $\bz_i/(f(\gamma)f(\delta)\bn_i).$ In the cases below, we use the observation
\begin{equation}
\name{eq:useful ineq}
\frac{\gamma}{1-2\gamma} \leq m+\half.
\end{equation}
We begin with $i=1$, and break into two cases. 
In case $m=0$, we have $r=1$ and therefore, 
$$\frac{\bz_1}{f(\gamma)f(\delta) \bn_1} =\frac{2(1-2\gamma)}{\gamma} \geq 4 > \frac{4}{3}.$$
In the remaining cases, we have $m \geq 1$ and we use the fact that $2m+r \geq 2m-1$. 
$$\frac{\bz_1}{f(\gamma)f(\delta) \bn_1} = \frac{(4m+2r)(1-2\gamma)}{\gamma} \geq \frac{2(2m-1)(1-2\gamma)}{\gamma} \geq  \frac{4m-2}{m+\half} \geq \frac{4}{3}.$$
The case of $i=4$ is given by:
$$\frac{\bz_4}{f(\gamma)f(\delta) \bn_4}=4(1-2\gamma)(2m+r+1)$$
In case $m=0$, we have $r=1$ and $\gamma<1/4$. Therefore, when $m=0$, we have
$$\frac{\bz_4}{f(\gamma)f(\delta) \bn_4}=8(1-2\gamma) > 2 > \frac{4}{3}.$$
Otherwise, we have $m \geq 1$ and $\gamma>1/4$. Using equation \ref{eq:useful ineq}, we see
$$\frac{\bz_4}{f(\gamma)f(\delta) \bn_4} \geq \frac{4\gamma(2m+r+1)}{m+\half}> \frac{2m+r+1}{m+\half} \geq \frac{2m}{m+\half} \geq \frac{4}{3}.$$
The remaining two indices follow by symmetry. (Observe that the action of switching $\gamma$ with $\delta$ has the effect of swapping
the first and third and second and fourth entries of all vectors involved.)
\end{proof}

\subsection{The non-recurrent case}
\name{sect:non-recurrent case}
Consider any four sequences of integers $m_i, n_i \geq 0$ and $r_i, s_i \in \{\pm 1\}$ defined for $i \geq 0$ so that 
$$(m_i,r_i) \neq (0,-1) \and (n_i,s_i) \neq (0,-1) \quad \text{for all $i \geq 0$.}$$
We call a collection of these four sequences an {\em itinerary}.
Theorem \ref{thm:coding} implies that for any itinerary, there is a unique pair $(\alpha, \beta)$ so that
\begin{equation}
\name{eq:forward}
f^{i+1}(\alpha)=r_i\left(\frac{f^i(\alpha)}{1-2f^i(\alpha)}-m_i\right) \and f^{i+1}(\beta)=s_i\left(\frac{f^i(\beta)}{1-2f^i(\beta)}-n_i\right).
\end{equation}
We call $(\alpha, \beta)$ the pair {\em determined} by the itinerary.
Theorem \ref{thm:coding} also gives a mild restriction on the itinerary which guarantees the irrationality of $\alpha$ and $\beta$. 

The itinerary is relevant for computing the cocycle $N(\alpha, \beta,k)$, which is the main ingredient in the formula 
$$\mu_\alpha \times \mu_\beta \times \mu_N(\sO_{k+1})=D(\alpha, \beta,k) \bn_{f^k(\alpha),f^k(\beta)} \cdot N(\alpha, \beta, k) \1.$$
The limit of these quantities as $n \to \infty$ gives the measure of all points without stable periodic orbits. See Section \ref{sect:limit formula}.

We will investigate itineraries of a particular form, and show that we can make choices which guarantee that 
the measure of $\sO_n$ decays as slow as we wish.

\begin{definition}[Upward and Downward Itineraries]
Consider an itinerary $I$ consisting of sequences $\{m_i\}$, $\{n_i\}$, $\{r_i\}$ and $\{s_i\}$ as above.
Let $k \geq 1$. 
\begin{itemize}
\item We say $I$ is a $k$-upward itinerary if 
$$(m_i,r_i,n_i,s_i)=\begin{cases}
(0,1,1,1) & \text{if $i=0$,} \\
(0,1,0,1) & \text{if $1 \leq i \leq k-1$,} \\
(1,1,0,1) & \text{if $i=k$.}
\end{cases}$$
\item We say $I$ has a $k$-rightward itinerary if  
$$(m_i,r_i,n_i,s_i)=\begin{cases}
(1,1,0,1) & \text{if $i=0$,} \\
(0,1,0,1) & \text{if $1 \leq i \leq k-1$,} \\
(0,1,1,1) & \text{if $i=k$.}
\end{cases}$$
\end{itemize}
\end{definition}
If $(\alpha, \beta)$ has a $k$-upward itinerary then $\beta \geq \frac{1}{3}$. And, if 
$(\alpha, \beta)$ has a $k$-rightward itinerary then $\alpha \geq \frac{1}{3}$. (See Proposition \ref{prop:special itineraries}.) This explains our choice of terminology.

We make use of a shift map on itineraries. If $I$ is an itinerary and $k \geq 1$ is an integer, we define $\sigma^k(I)$ to be the collection
of sequences formed by dropping the first $k$ values of each of the four sequences making up $I$, and re-indexing so that each sequence begins at zero.

\begin{definition}[Understandable Itineraries]
Let $\{k_j \geq 1\}$ be a sequence of integers defined for $j \geq 0$. Using $\{k_j\}$, we define the auxiliary sequence $\{\aux_j\}$ inductively by the rule
$$\aux_{0}=0, \and \aux_{j+1}=\aux_j+k_j+1 \quad \text{for $j \geq 0$}.$$
We say the {\em $\{k_j\}$-understandable itinerary} $I$ is the itinerary determined by the following rules.  
\begin{enumerate}
\item For any even $j \geq 0$, the itinerary $\sigma^{\aux_j}(I)$ is a $k_j$-upward itinerary.
\item For any odd $j \geq 1$, the itinerary $\sigma^{\aux_j}(I)$ is a $k_j$-rightward itinerary.
\end{enumerate}
\end{definition}

Because of Theorem \ref{thm:recurrent case}, we are interested in pairs $(\alpha, \beta)$ such that the collection of all limit points of the $f \times f$-orbit of $(\alpha, \beta)$ is contained in the set 
$$\set{(x,y) \in [0, \half] \times [0, \half]}{$x=0$ or $y=0$}.$$ 
The holds for the pair $(\alpha, \beta)$ determined by a $\{k_j\}$-understandable itinerary 
precisely when $\liminf k_j=\infty$. The following results will imply that we get a (large) positive measure set of non-periodic points
if $\{k_j\}$ grows sufficiently quickly.

\begin{proposition}[Decay Control]\name{prop:decay control}
There is a function $K_0:(0,1) \to \Z$ satisfying the following statement. For each $\epsilon_0>0$,
whenever $(\alpha, \beta)$ has a $k_0$-upward itinerary with $k_0>K_0(\epsilon_0)$ then 
$\nu(\sO_1)>1-\epsilon_0,$
where $\nu=\mu_\alpha \times \mu_\beta \times \mu_N.$
\end{proposition}

\begin{theorem}[Decay Control] \name{thm:decay control}
There is a function $K:\Z \times (0,1) \to \Z$ satisfying the following statement. 
For any sequence $\{\epsilon_j\}_{j \geq 1}$ with $0<\epsilon_j<1$, if $\{k_j\}_{j \geq 0}$ is a sequence satisfying
$$k_{j} \geq K(k_{j-1},\epsilon_{j}) \quad \text{for all} \quad j \geq 1,$$
then the pair $(\alpha, \beta)$ determined by the $\{k_j\}$-understandable itinerary with auxiliary sequence $\{a_j\}$ satisfies
$$\nu(\sO_{\aux_{j}+1})>(1-\epsilon_{j}) \nu(\sO_{\aux_{j-1}+1}) \quad \text{for all $j \geq 1$,}$$
where $\nu=\mu_\alpha \times \mu_\beta \times \mu_N$.
\end{theorem}

The Decay Control Proposition and Theorem together imply Theorem \ref{thm:4} of the introduction, which can be restated as saying that for any $\eta>0$, there exists
a pair of irrationals $(\alpha, \beta)$ so that 
$\nu(\NS)>1-\eta.$
\begin{proof}[Proof of Theorem \ref{thm:4} given the Decay Control results]
Fix any $\eta>0$, and fix any sequence $\{\epsilon_j\}_{j \geq 0}$ of numbers in $(0,1)$ so that 
$$\prod_{j=0}^\infty (1-\epsilon_j)>1-\eta.$$
Choose a sequence $\{k_j\}_{j \geq 0}$ so that 
$k_0>K_0(\epsilon_0)$ and $k_j \geq K(k_{j-1},\epsilon_j)$ for all $j \geq 1$. 
Then the Decay Control Proposition implies  
$\nu(\sO_1)>1-\epsilon_0$. The theorem implies that  
$$\nu(\sO_{\aux_{j}+1})>(1-\epsilon_{j}) \nu(\sO_{\aux_{j-1}+1})$$
for all $j \geq 1$. By statement (\ref{item:O}) of section \ref{sect:limit formula}, we have 
$$\nu(\NS)=\lim_{j \to \infty} \nu(\sO_{\aux_{j}+1}) \geq \lim_{j \to \infty} \prod_{i=0}^j (1-\epsilon_i) > 1-\eta$$
as desired. Finally, we observe that $\alpha$ and $\beta$ are irrational, by Theorem \ref{thm:coding}. This is true  for any 
pair determined by an understandable itinerary.
\end{proof}

We will now give a proof of Corollary \ref{cor:5}, which states that the set 
$$P=\{(\alpha,\beta)~:~\mu_\alpha \times \mu_\beta \times \mu_N(\NS)>0\}$$
is dense.

\begin{proof}[Proof of Corollary \ref{cor:5} given the Decay Control results]
Theorem \ref{thm:4} gives a pair $(\alpha, \beta) \in P$.
Whenever $(\alpha', \beta')$ satisfies
$f^n(\alpha')=\alpha$ and $f^n(\beta')=\beta$, we also have 
$(\alpha',\beta') \in Y$ because Theorem \ref{thm:ren} gives a return map of $\tPsi_{\alpha', \beta'}$ which is affinely conjugate to $\tPsi_{\alpha, \beta}$. 
Theorem \ref{thm:coding} implies that the itinerary map gives a semiconjugacy from the shift map on a shift space to the action of $f \times f$. 
Since the collection of preimages of a point in a shift space is dense, the collection of preimages of $(\alpha,\beta)$ under $f \times f$ must be dense.
\end{proof}

We will build up to a proof of the Decay Control Proposition and Theorem Theorem. 
The following is a necessary calculation, which follows from an inductive argument using equation \ref{eq:forward}. (We carry out a similar calculation in the proof of Lemma \ref{lemma:decay ineq} below.)

\begin{proposition}[Starting Points of Itineraries]
\name{prop:special itineraries} If $(\alpha, \beta)$ has a
$k$-upward itinerary then 
$$\frac{1}{3+2k} \leq \alpha \leq \frac{3}{8+6k} \and
\frac{1}{3} \leq \beta \leq \frac{3+2k}{8+6k}.$$
If $(\alpha, \beta)$ has a
$k$-rightward itinerary then 
$$\frac{1}{3} \leq \alpha \leq \frac{3+2k}{8+6k} \and 
\frac{1}{3+2k} \leq \beta \leq \frac{3}{8+6k}.$$
\end{proposition}

We will now prove the Decay Control Proposition.
This proof reveals some of the ideas appearing in the proof of the Decay Control Theorem.
\begin{proof}[Proof of the Decay Control Proposition]
It suffices to show that for any $\epsilon_0>0$ and for sufficiently large $k_0$ we have 
$\nu(\sO_{1})>1-\epsilon_{0}.$
By Theorem \ref{thm:cocycle formula}, we have
$$\mu_\alpha \times \mu_\beta \times \mu_N(\sO_{1})=1-4\alpha \beta.$$
We know that $(\alpha, \beta)$ will have a $k_0$-upward itinerary. Proposition \ref{prop:special itineraries} then confines $(\alpha, \beta)$ to a rectangle where
$$1-4\alpha\beta \geq 1-\frac{3(3+2k_0)}{(8+6k_0)^2}.$$
The quantity on the right tends to $1$ as $k_0 \to \infty$. So choosing a sufficiently large $k_0$ makes $1-4\alpha \beta$ larger than $1-\epsilon_0$.
\end{proof}

Now consider the Decay Control Theorem. We begin by interpreting the quantities under consideration using the cocycle. Fix $j \geq 1$ and set 
$$\alpha'=f^{\aux_{j-1}}(\alpha), \quad \beta'=f^{\aux_{j-1}}(\beta), \quad \gamma=f^{\aux_{j}}(\alpha), \and \delta=f^{\aux_{j}}(\beta).$$
Define the vector 
\begin{equation}
\name{eq:z2}
\bz=D(\alpha, \beta,\aux_{j-1}) N(\alpha, \beta, \aux_{j-1}) \1.
\end{equation}
By our cocycle formula (Theorem \ref{thm:cocycle formula}), we 
have the following two identities:
\begin{equation}
\name{eq:decay1}
\nu(\sO_{\aux_{j-1}+1})=\bn_{\alpha',\beta'} \cdot \bz.
\end{equation}
\begin{equation}
\name{eq:decay2}
\nu(\sO_{\aux_{j}+1})=\big(D(\alpha',\beta',k_{j-1}+1) N(\alpha',\beta',k_{j-1}+1)^T \bn_{\gamma,\delta} \big) \cdot \bz.
\end{equation}
To show that the value of the second equation is nearly as large as the first equation, it suffices to prove an entrywise inequality involving the vectors that show up in these equations. Specifically, we will show that there is a function $K$ as in the theorem so that whenever 
$k_{j}>K(k_{j-1},\epsilon_j)$ we have the entrywise inequality
\begin{equation}
\name{eq:entrywise ineq}
D(\alpha',\beta',k_{j-1}+1) N(\alpha',\beta',k_{j-1}+1)^T \bn_{\gamma,\delta} > (1-\epsilon_j) \bn_{\alpha',\beta'}
\end{equation}
This statement is proved in the following lemma.

\begin{lemma}[Decay Control Inequality]
\name{lemma:decay ineq}
Fix any $\epsilon>0$. Assume that $(\alpha, \beta)$ has $k$-upward itinerary. Define
$$\gamma=f^{k+1}(\alpha) \and \delta=f^{k+1}(\beta).$$
There is a $K=K(k,\epsilon)$ so that for any $k'>K$, if $(\gamma,\delta)$ has a $k'$-rightward itinerary,
then we have the entrywise inequality 
$$D(\alpha,\beta,k+1) N(\alpha,\beta,k+1)^T \bn_{\gamma,\delta} > (1-\epsilon) \bn_{\alpha,\beta}.$$
The same statement holds with the same function $K(k, \epsilon)$ when the notion of `upward' is swapped with `rightward.'
\end{lemma}

\begin{remark}[On the proof of the lemma]
We prove Lemma \ref{lemma:decay ineq} by calculation, which seems unfortunate. However, we can go back to the argument above the lemma to see why this is necessary. The argument we use only utilizes knowledge of the portion of the itinerary with indices $j$ satisfying $a_{j-1} \leq j < a_{j+1}$. Since the vector $\bz$ defined in equation \ref{eq:z2} depends on an earlier part of the itinerary, we have no a priori control over the value of $\bz$. 
So, to control the decay in moving from the value of equation \ref{eq:decay1} to equation \ref{eq:decay2}, we are forced to prove the entrywise inequality in equation \ref{eq:entrywise ineq}. 

For further commentary, let $\bx=\bn_{\alpha',\beta'}$ and
$$\by=D(\alpha,\beta,k+1) N(\alpha,\beta,k+1)^T \bn_{\gamma,\delta}.$$
These vectors have some geometric meaning. The vector $\bx$
represents the $\nu'=\mu_{\alpha'} \times \mu_{\beta'} \times \mu_N$ measures of the intersections of $\sO_1$ with four subsets of $X$. (These four sets are $\sS_1 \cup \sS_2$, $\sS_3$, $\sS_4 \cup \sS_5$ and $\sS_6$. See equation \ref{eq:n} and the definition of $\bq$ above equation \ref{eq:q int}.) By definition of the cocycle, the entries of the vector $\by$ represent the $\nu'$ measures of $\sO_{k+2}$  intersected with the same four sets. 
Since $\sO_{k+2} \subset \sO_1$, we see $\by < \bx$ entrywise. Moreover, we could prove directly that as $k_j \to \infty$ that we have 
$\nu'(\sO_{k+2}) \to 1$ and $\nu'(\sO_1) \to 1.$ (This explains the situation when $\bz=\1$.) 
However, as $k_j \to \infty$ some of the entries of $\bx$ tend to zero. If $\by$ is obtained from $\bx$ by disproportionately decreasing the values of small entries of $\bx$, then $\by$ could be arranged not to satisfy $\by > (1-\epsilon) \bx$ while still satisfying the conditions
$$\by \cdot \1 \approx \bx \cdot \1 \and \by<\bx.$$
We view this as forcing us to do a calculation to guarantee
$\by>(1-\epsilon) \bx$ as $k_j \to \infty$. 
\end{remark}

\begin{proof}We will only do the version of the lemma without swapping terms. This second case follows from the first by symmetry.

We will do all our calculations in terms of $\gamma$ and $\delta$. Let $\alpha_i=f^i(\alpha)$ and $\beta_i=f^i(\beta)$. Knowing that $(\alpha, \beta)$ has a $k$-downward itinerary allows us to give formulas for some of these values via equation \ref{eq:forward}:
$$\alpha_i=\begin{cases}
\frac{1+\gamma}{1+2(1+\gamma)(1+k-i)} & \text{if $0 \leq i \leq k$,}\\
\gamma & \text{if $i=k+1$.}
\end{cases}
\qquad
\beta_i=\begin{cases}
\frac{1+\delta(2k+1)}{3+2 \delta(3k+1)} & \text{if $i=0$,}\\
\frac{\delta}{1+2 \delta(k+1-i)} & \text{if $1 \leq i \leq k+1$,}
\end{cases}
$$
It is relevant to compute $1-2\alpha_i$ and $1-2\beta_i$. We have
$$1-2\alpha_i=\begin{cases}
\frac{1+2(\gamma+1)(k-i)}{1+2(\gamma+1)(k+1-i)} & \text{if $0 \leq i \leq k$,}\\
1-2\gamma & \text{if $i=k+1$.}
\end{cases}
\qquad
1-2\beta_i=\begin{cases}
\frac{1+2\delta k}{3+2\delta(1+3k)} & \text{if $i=0$,}\\
\frac{1+2 \delta(k-i)}{1+2 \delta(k+1-i)} & \text{if $1 \leq i \leq k+1$,}
\end{cases}
$$
To simplify notation, we define the following quantities:
$$\begin{array}{ccc}
d=D(\alpha,\beta,k+1). & \quad & N=N(\alpha, \beta, k+1). \\
\v=\frac{1}{d} \bn_{\alpha, \beta}. & \quad & \bw=N^T \bn_{\gamma, \delta}.
\end{array}$$
Fix $\epsilon>0$ as in the lemma. The goal of this proof is to show that for $k'$ sufficiently large, we have 
$\bw>(1-\epsilon)\v$ entrywise. We do this by calculation.

There is a lot of cancellation in the product for $d$:
$$d=\prod_{j=0}^k (1-2\alpha_j)(1-2\beta_j)=\frac{1}{\big(1+2(1+\gamma)(k+1)\big)\big(3+2\delta(3k+1)\big)}.$$
By a calculation, we observe that $\v$ has a relatively simple expression:
$$\begin{array}{cc}
\displaystyle \v_1=(1+\gamma)(1+2\delta k) &
\displaystyle \v_2=
\thalf\big(1+2(1+\gamma)k\big)\big(3+\delta(2+6k)\big) 
\\
\displaystyle \v_3=\big(1+2(1+\gamma)k\big)(1+\delta(1+2k)) &
\displaystyle \v_4=\thalf\big(1+2(1+\gamma)(1+k)\big)(1+2\delta k) \\
\end{array}$$
The value of $N$ can be computed from the knowledge that $(\alpha,\beta)$ has a $k$-upward itinerary:
$$\begin{narrowarray}{3pt}{rcl} N & = & \left[\begin{array}{rrrr} 
3 & 1 & 0 & 3 \\
2 & 1 & 0 & 2 \\
0 & 1 & 1 & 1 \\
0 & 0 & 0 & 1 \end{array}\right]
\left(\left[\begin{array}{rrrr} 
1 & 1 & 0 & 1 \\
0 & 1 & 0 & 0 \\
0 & 1 & 1 & 1 \\
0 & 0 & 0 & 1 \end{array}\right]^{k-1}\right)
\left[\begin{array}{rrrr} 
1 & 1 & 0 & 1 \\
0 & 1 & 0 & 0 \\
0 & 3 & 3 & 1 \\
0 & 2 & 2 & 1 \end{array}\right] \\
& = &
\left[\begin{array}{rrrr} 
3 & 1+9k & 6k & 3(1+k) \\
2 & 1+6k & 4k & 2(1+k) \\
0 & 3(1+k) & 3+2k & 1+k \\
0 & 2 & 2 & 1 \end{array}\right]
\end{narrowarray}
$$
By definition of $\bn_{\gamma,\delta}$ we have
$$\textstyle \bn_{\gamma,\delta}=\big(\gamma(1-2\delta), \frac{1-2\gamma}{2}, (1-2\gamma)\delta,\frac{1-2\delta}{2}\big).$$
This allows us to compute $\bw=N^T \bn_{\gamma,\delta}$:
$$\bw=\left[\begin{array}{r}
1+\gamma-6 \gamma \delta 
\\
\thalf\Big(3\big(1+2(1+\gamma)k\big)
-2\delta(8 \gamma -1)(1+3k)\Big) 
\\
1+2(1+\gamma)k-2 \delta(1-6 \gamma+2k-16 \gamma k) 
\\
\thalf\Big(1+2(1+\gamma)(1+k)
-2\delta(8 \gamma(k+1)-k\Big) \\
\end{array}\right]$$
Now observe that the limits of $\v$ and $\bw$ as $\delta \to 0$ are equal and positive:
$$\lim_{\delta \to 0} \v=\lim_{\delta \to 0} \bw=
\left[\begin{array}{r}
1+\gamma
\\
{\textstyle \frac{3}{2}} \big(1+2(1+\gamma)k\big)
\\
1+2(1+\gamma)k
\\
\thalf\Big(1+2(1+\gamma)(1+k)\big) \\
\end{array}\right]$$
By continuity of $\v$ and $\bw$ as functions of $\gamma$ and $\delta$, we can take this convergence to be uniform in $\gamma$. Therefore, there is a constant $C>0$ so that $\bw_i/\v_i>1-\epsilon$ for all $i$ whenever $\delta<C$. By Proposition \ref{prop:special itineraries}, we can force
$\delta<C$ by assuming that $(\gamma,\delta)$ has a $k'$-rightward itinerary with 
$$\frac{3}{8+6k'}<C.$$
\end{proof}

By the remarks made above the statement of the lemma, this lemma also proves the Decay Control Theorem.

\section{Dynamics on the Parameter Space}
\name{sect:param space}
In this section, we investigate the dynamical behavior of the map:
$$f:[0,\half) \to [0, \half]; \quad f(x)=\frac{x}{1-2x} \pmod{G},$$
where $G$ is the group of isometries of $\R$ preserving $\Z$. This map was first mentioned in equation \ref{eq:f} of the introduction.  We also study the product map $f \times f$. 

The map $f$ is somewhat similar an analog Gauss map which appears when studying continued fractions. In the first subsection, we develop this point of view with an emphasis on coding and detecting irrationality. 

In the second subsection, we show that $f \times f$ is recurrent with respect to Lebesgue measure.
\subsection{Coding and rationality}\name{sect:rationality f}
We define $A$ to be the infinite alphabet
$$A=\{ (n,r) \in \Z \times \{\pm 1\}~:~\text{$n \geq 0$ and $(n,r) \neq (0,-1)$}\}.$$
For each $(n,r) \in A$, we define the interval
$$I_{n,r}=\left\{x \in [0,\half)~:~r(\frac{x}{1-2x}-n) \in [0,\half].\right\}.$$
Observe that the union of these intervals covers $[0,\half)$. 

Let $\{(n_k,r_k) \in A\}$ be a sequence defined for $k \geq 0$. 
We say $\{(n_k,r_k)\}$ is a {\em coding sequence} for $x \in [0,\half)$ if $f^k(x)$ is well defined for all $k \geq 0$
and 
$$f^k(x) \in I_{n_k,r_k} \quad \text{for all $k \geq 0$}.$$
(The value $f^k(x)$ is always well defined unless there is a $k$ so that $f^k(x)=\half$.)

The main goal of this subsection is to prove the following:
\begin{theorem}[Coding]
\name{thm:coding}
For each sequence $\{(n_k,r_k) \in A\}_{k \geq 0}$, there is a unique $x \in [0,\half)$ so that $\{(n_k,r_k)\}$
is a coding sequence for $x$. This $x$ depends continuously on the choice of $\{(n_k,r_k)\}$, when the collection of all
sequences is given the shift space topology. Moreover $x$ is irrational unless there is an $K$ so that $(n_k,r_k)=(0,1)$
for all $k \geq K$.
\end{theorem}

\begin{remark}[Continued Fractions] We may think of $x$ as determined by $\{(n_k,r_k)\}$ via:
$$x=\cfrac{1}{2+\cfrac{1}{n_0+\cfrac{r_0}{2+\cfrac{1}{n_1+\cfrac{r_1}{2+\ldots}}}}}.$$
\end{remark}

The proof follows from understanding the action of $f$ on the intervals $I_{n,r}$. We compute
$$I_{n,1}=\left[\frac{n}{1+2n}, \frac{2n+1}{4n+4}\right] \and
I_{n,-1}=\left[\frac{2n-1}{4n}, \frac{n}{1+2n}\right].$$
Observe that $f$ restricts to a bijection $I_{n,r} \to [0,\half]$. The inverse of this restriction is
\begin{equation}
\name{eq:g}
g_{n,r}:[0,\half] \to I_{n,r}; \quad g_{n,r}(x)=\frac{r x + n}{1 + 2 (r x + n)}.
\end{equation}

We first prove the existence, uniqueness and continuity comments of the theorem. Further below, we give the proof of the irrationality condition.

\begin{proof}[Proof of existence, uniqueness and continuity]
This follows from standard dynamics arguments involving Markov partitions for maps of the interval. The 
collection of all $I_{n,r}$ form a Markov Partition. The image of each interval under $f$ is $[0,\half]$. 

Consider a finite sequence $\{(n_k,r_k)\}$ defined for $0 \leq k  \leq K$. The set of points $x$ for which 
$f^k(x) \in I_{n_k,r_k}$ is given by 
\begin{equation}
\name{eq:g seq}
g_{n_0,r_0} \circ g_{n_1,r_1} \circ \ldots \circ g_{n_K,r_K}([0,\thalf]).
\end{equation}
Each such set is a closed interval. This applies continuity of the dependence of $x$ on the sequence (assuming $x$ is uniquely determined).

Now let $\{(n_k, r_k)\}$ be an infinite sequence. Let $J$ be the set of points $x$ so that $\{(n_k, r_k)\}$ is a coding sequence for $x$. 
The set $J$ is a nested intersection of sets of the form given in equation \ref{eq:g seq}. Therefore, $J$ is a non-empty closed interval. 
We must prove that $J$ has no more than one point. Suppose $J$ is not just a single point. Then, it contains an irrational $x$. Observe that under iteration, an irrational must visit the set $(\frac{1}{4}, \half)$ infinitely often. This is because if $f^i(x)<\frac{1}{4}$ for $i=1, \ldots k-1$, then 
$$f^k(x)=\frac{x}{1-2kx}.$$ 
So, eventually $f^k(x)>\frac{1}{4}$. If $x>\frac{1}{4}$ and is irrational, then $f$ is locally strictly expanding by a factor larger than one. Therefore, the length of $f^k(J)$ would have to tend to infinity as $k \to \infty$. This contradicts the assumption that $J$ was not just a single point.
\end{proof}

\begin{proof}[Proof of the Irrationality Condition]
Observe that there is a unique coding sequence for zero, consisting of the infinite sequence with $(n_k,r_k)=(0,1)$
for all $k$. So, it suffices to show that for all rational $p/q \in [0,\half)$ there is a $k$ so that 
$f^k({\textstyle \frac{p}{q}}) \in \{0,\thalf\}.$

Observe that the action of $f$ on reduced fractions in $\Q \cap [0,\half)$ is given by the formula
$$f({\textstyle\frac{p}{q}})=\frac{p}{q-2p} \pmod{G}.$$ 
We define the ``complexity function'' 
$$\chi:\Q \to \N; \quad \frac{p}{q} \mapsto q,$$
where $\frac{p}{q}$ is assumed to be a reduced fraction with $q>0$. For all such $\frac{p}{q} \in (0,\half)$,
we have 
$$\chi \circ f(\textstyle{\frac{p}{q}})=q-2p<q=\chi(\textstyle{\frac{p}{q}})$$ 
so the complexity drops by at least two when applying $f$. So, eventually the denominator must drop to a value of one or two.
\end{proof}

\subsection{Measurable dynamics}\name{sect:measurable f}
In this section we treat $f$ as a map on the set $I$ of irrationals in the interval $(0,\half)$. This set has full Lebesgue measure,
and is invariant under $f$. Our main result is the following:

\begin{theorem}[Recurrence] 
\name{thm:recurrence} Let $\lambda$ denote Lebesgue measure on $I$. The action of $f \times f$ on $I^2$ is 
recurrent in the sense that for any Borel subset $A \subset I^2$, for $\lambda^2$-a.e. $(x,y) \in A$ there is an $n \geq 1$ so that $(f \times f)^n(x,y) \in A$.
\end{theorem}

We actually prove the above statement by replacing $\lambda$ with an equivalent measure $m$. (Two measures are {\em equivalent} if they have the same null sets.) 

\begin{lemma}[Invariant measure] The $\lambda$ equivalent measure $m$ on $I$ defined by 
$$\meas(A)=\int_A \frac{1}{x}+\frac{1}{1-x}~dx.$$
is $f$-invariant (i.e., $m \circ f^{-1} = m$).
\end{lemma}
The measure $m$ should be thought of as analogous to the Gauss measure for continued fractions.
\begin{proof}
Let $\lambda$ denote Lebesgue measure on $I$. 
Consider the Radon-Nikodym derivative 
$$h(x)=\frac{d \meas}{d\lambda}(x)=\frac{1}{x}+\frac{1}{1-x}.$$ 
Observe that 
$$\frac{d (\meas \circ f^{-1})}{d \lambda}(x)=\sum_{y \in f^{-1}(\{x\})} \frac{h(y)}{|f'(y)|}.$$ 
So, it suffices to show that this sum yields $h(x)$. We compute that $$\frac{h(y)}{|f'(y)|}=\frac{(1-2y)^2}{y(1-y)}.$$
We can write each $y \in f^{-1}(\{x\})$ as $y=g_{n,r}(x)$ as in equation \ref{eq:g}. We evaluate the sum in two portions. In the cases $r=1$ and $r=-1$, we respectively have 
$$\sum_{n=0}^\infty  \frac{h \circ g_{n,1}(x)}{|f' \circ g_{n,1}(x)|}=\sum_{n=0}^\infty \big(\frac{1}{n+x}-\frac{1}{1+n+x}\big)=\frac{1}{x}, \and$$
$$\sum_{n=1}^\infty  \frac{h \circ g_{n,-1}(x)}{|f' \circ g_{n,-1}(x)|}=\sum_{n=1}^\infty \big(\frac{1}{n-x}-\frac{1}{1+n-x}\big)=\frac{1}{1-x}.$$
Combining these two sums, we see $\frac{d (\meas \circ f^{-1})}{d \lambda}(x)=h(x)$, as desired.
\end{proof}

Note that the measure $m$ is infinite. If this were not the case, we would have Recurrence by the Poincar\'e recurrence theorem. However, the measure of sets of the form $(\epsilon,\half) \cap I$ is finite. Our proof of 
recurrence depends on controlling the possibility of the backward iterates of a set tending toward the set of points where one coordinate is zero.
This control is given by the following.

\begin{lemma}[Plug Lemma]
For $\epsilon>0$, define the following subsets of $I^2$:
$$N_\epsilon=[0, \epsilon] \times I \cup I \times [0, \epsilon] \and P_\epsilon=(f \times f)(N_\epsilon) \sm N_\epsilon.$$
We have $\lim_{\epsilon \to 0} m \times m (P_\epsilon)=0.$
\end{lemma}

We call this the plug lemma, because
any orbit starting in $N_\epsilon$ must pass through $P_\epsilon$
in order to reach the complement of $N_\epsilon$. 
The observation of the lemma is that not much can pass through the plug.
Using this Lemma, we can prove recurrence:

\begin{proof}[Proof of the Recurrence Theorem]
The proof is a general principal following from the Hopf Decomposition. (See \S 1.3 of \cite{Krengel85}, for instance.) If $f \times f$ were not recurrent, then there would be a wandering set $W \subset I \times I$ of positive $m \times m$ measure. That is, $W$ is a set so that the preimages $(f \times f)^{-k}(W)$
are pairwise disjoint for $k \geq 0$. 

To simplify notation let $\mu=m \times m$ and $\phi=f \times f$. 

We use the facts that $\cap_{\epsilon>0} N_\epsilon=\emptyset$ and $\lim_{\epsilon \to 0} \mu (P_\epsilon)=0$. By possibly making $W$ smaller, we can assume that there is an $\epsilon>0$
so that
\begin{enumerate}
\item $W \cap N_\epsilon= \emptyset$.
\item $\mu(W)>\mu(P_\epsilon)$. 
\end{enumerate}

Now consider the sequence of sets $P_k$ and $W_k$ defined inductively according to the following rules. We define $P_0=W \cap P_\epsilon$ and $W_0=W \sm P_0$. For $k \geq 0$ define
$$P_{k+1}=\phi^{-1}(W_k) \cap P_\epsilon \and 
W_{k+1}=\phi^{-1}(W_k) \sm P_{k+1}.$$
By invariance of $\mu$, the sequence $\mu(W_k)$ is decreasing. Because each $W_k$ is disjoint and lies in the complement of $N_\epsilon$ (which has finite measure with respect to $\mu$), we have 
$$\lim_{k \to \infty} \mu(W_k)=0.$$
Again by invariance of $\mu$, we have $\mu(W_k)=\mu(W_{k+1})+\mu(P_{k+1})$. It follows that for all $k \geq 0$,
$$\mu(W)=\mu(W_k)+ \sum_{i=0}^k \mu(P_k) \quad
\text{and so} \quad 
\mu(W)=\sum_{i =0}^\infty \mu(P_i).$$ 
Note that the sets $P_i$ are pairwise disjoint and lie in $P_\epsilon$. This contradicts the statement that $\mu(W)>\mu(P_\epsilon)$. 
\end{proof}

\begin{proof}[Proof of the Plug Lemma]
We will assume $\epsilon<\frac{1}{4}$. We can write $P_\epsilon$ as a union of rectangles, 
$$P_{\epsilon}=\textstyle (\epsilon,\frac{\epsilon}{1-2\epsilon}] \times (\epsilon,\frac{1}{2}) \cup (\epsilon,\frac{1}{2}) \times (\epsilon,\frac{\epsilon}{1-2\epsilon}].$$
Therefore, $m \times m(P_\epsilon)$ is less than twice the product of the measures of these two intervals with respect to $m$. We have 
$$\textstyle m\big([\epsilon,\frac{\epsilon}{1-2\epsilon}]\big)=\log(\frac{1-\epsilon}{1-3 \epsilon}) \and
m\big([\epsilon,\frac{1}{2}] = \log(\frac{1-\epsilon}{\epsilon}).$$
A calculation shows that as $\epsilon$ tends to zero, the product of these quantities tends to zero.
\end{proof}

\bibliographystyle{amsalpha}
\bibliography{/home/pat/archive/my_papers/bibliography}
\end{document}